\documentclass[12pt,reqno]{amsart}
\usepackage{amsthm,amsfonts,amssymb,euscript,fancyhdr}
\usepackage[colorlinks,linkcolor=blue,citecolor=blue]{hyperref}

\usepackage{graphicx} 
\usepackage{subcaption}

\newtheorem{theorem}{Theorem}[section]
\newtheorem{lemma}[theorem]{Lemma}
\newtheorem{proposition}[theorem]{Proposition}
\newtheorem{corollary}[theorem]{Corollary}
\newtheorem{definition}[theorem]{Definition}
\newtheorem{remark}[theorem]{Remark}
\newtheorem{conjecture}[theorem]{Conjecture}
\newtheorem{claim}[theorem]{Claim}

\numberwithin{equation}{section}

\def\a{{\alpha}}

\def\be{{\beta}}
\def\ga{\gamma}

\def\de{\delta}

\def\ep{\epsilon}

\def\la{\lambda}

\def\rh{\rho}
\def\si{\sigma}
\def\Si{\Sigma}
\def\om{\omega}
\def\Om{\Omega}
\def\th{\theta}
\def\Th{\Theta}
\def\ze{\zeta}

\def\BB{{\mathcal B}}

\def\MM{{\mathcal M}}

\def\II{{\mathcal I}}

\def\EE{{\mathcal E}}
\def\HH{{\mathcal H}}

\def\TT{{\mathcal T}}

\def\OO{{\mathcal O}}
\def\SS{{\mathcal S}}

\def\Lie{{\mathcal L}}

\def\DD{{\mathcal D}}

\def\RR{{\mathcal R}}

\def\Lie{{\mathcal L}}

\def\A{{\bf A}}

\def\D{{\bf D}}

\def\g{{\bf g}}

\def\SSS{{\mathbb S}}
\def\RRR{{\mathbb R}}

\def\Lb{{\,\underline{L}}}

\def\chih{{\widehat \chi}}
\def\chib{{\underline \chi}}
\def\chibh{{\underline{\chih}}}

\def\etab{{\underline \eta}}

\def\alphab{{\underline{\alpha}}}
\def\betab{{\underline{\beta}}}

\newcommand{\Divd}{\Div \mkern-17mu /\ }
\newcommand{\Curld}{\Curl \mkern-17mu /\ }
\newcommand{\Nd}{\nabla \mkern-13mu /\ }

\newcommand{\Ld}{\triangle \mkern-12mu /\ }

\newcommand{\gd}{{g \mkern-8mu /\ \mkern-5mu }}

\newcommand{\wt}[1]{\widetilde{#1}}

\newcommand{\RRt}{{\mathcal{R}}}

\newcommand{\Lt}{{{L}}}
\newcommand{\Lbt}{{{\Lb}}}
\newcommand{\Omt}{{{\Om}}}

\newcommand{\etA}{{{e}_A}}
\newcommand{\etB}{{{e}_B}}

\newcommand{\chit}{{{\chi}}}
\DeclareRobustCommand{\chibt}{{\underline{\chi}}}

\newcommand{\zet}{{{\ze}}}
\newcommand{\etabt}{{{\etab}}}

\newcommand{\chiht}{{\widehat{{\chi}}}}
\newcommand{\chibht}{{\widehat{{\underline{\chi}}}}}
\newcommand{\trchit}{{\tr \chit}}

\newcommand{\trchibt}{{\tr \chibt}}
\newcommand{\alphat}{{{\alpha}}}
\newcommand{\betat}{{{\beta}}}
\newcommand{\rhot}{{{\rho}}}
\newcommand{\sigmat}{{{\sigma}}}

\newcommand{\Ltt}{{\widetilde{L}}}

\newcommand{\etatt}{{\tilde{\eta}}}

\newcommand{\nutt}{{\tilde{\nu}}}

\providecommand{\lrpar}[1]{\left(#1\right)}
\providecommand{\ol}[1]{\overline{#1}}

\def\CMD{{\varep_{\mathrm{ball}}}}

\providecommand{\norm}[1]{\left\Vert#1\right\Vert}

\newcommand{\intSI}{{\int\limits_\Si}}
\newcommand{\intpSI}{{\int\limits_{\partial \Si}}}

\DeclareMathOperator{\Div}{\mathrm{div}}
\DeclareMathOperator*{\Curl}{\mathrm{curl}}

\def\nab{\nabla}
\def\varep{\varepsilon}

\def\pr{{\partial}}
\def\les{\lesssim}

\def\f12{{\frac 1 2}} 
 
\def\tr{\mathrm{tr}}

\def\f{\widetilde{f}}

\def\half{\frac{1}{2}}
\newcommand{\RRRic}{\mathrm{Ric}}

\newcommand{\Rbf}{\mathbf{R}}

\newcommand{\Rscal}{\mathrm{R}_{scal}}

\newcommand{\Lied}{\mathcal{L} \mkern-9mu/\ \mkern-7mu}

\newcommand{\R}{{\bf R}}
\newcommand{\Ric}{\text{Ric}}
\newcommand{\Lieh}{{\hat \Lie}}
\newcommand{\trTh}{\tr\Th}
\newcommand{\Thh}{\hat \Th}
\newcommand{\nut}{{{\nu}}}

\newcommand{\trchi}{\tr \chi}
\newcommand{\trchib}{\tr \chib}

\newcommand{\ttt}{{\tilde{t}}}

\newcommand{\ktt}{{\tilde{k}}}
\newcommand{\nabtt}{{\widetilde{\nab}}}
\newcommand{\ntt}{{\tilde{n}}}

\newcommand{\Ttt}{{\widetilde{T}}}

\newcommand{\ett}{{\tilde{e}}}

\newcommand{\Sitt}{{\widetilde{\Si}}}

\makeatletter
\def\l@section{\@tocline{1}{0pt}{1pc}{}{}}
\renewcommand{\tocsection}[3]{%
\indentlabel{\@ifnotempty{#2}{\ignorespaces#1 #2.\quad}}#3}
\def\l@subsection{\@tocline{2}{0pt}{1pc}{5pc}{}}
\renewcommand{\tocsubsection}[3]{%
  \indentlabel{\hspace*{2.3em}\@ifnotempty{#2}{\ignorespaces#1 #2.\quad}}#3}
\renewcommand{\tocappendix}[3]{%
\indentlabel{#1\@ifnotempty{#2}{ #2}.\quad}#3}
\makeatother

\begin{document}

\title[The spacelike-characteristic Cauchy problem]{The spacelike-characteristic Cauchy problem \\ of general relativity in low regularity}

\address[Stefan Czimek]{Department of Mathematics, University of Toronto, Canada}

\author[Stefan Czimek]{Stefan Czimek}
\email{stefan.czimek@utoronto.ca}

\address[Olivier Graf]{Laboratoire Jacques-Louis Lions, Sorbonne University, France}

\author[Olivier Graf]{Olivier Graf}
\email{grafo@ljll.math.upmc.fr}

\begin{abstract} In this paper we study the spacelike-characteristic Cauchy problem for the Einstein vacuum equations. We prove that given initial data on a maximal compact spacelike hypersurface $\Si \simeq \ol{B(0,1)} \subset \RRR^3$ and the outgoing null hypersurface $\HH$ emanating from $\pr \Si$, the time of existence of a solution to the Einstein vacuum equations is controlled by low regularity bounds on the initial data at the level of curvature in $L^2$. \newline
The proof uses the bounded $L^2$ curvature theorem \cite{KRS}, the extension procedure for the constraint equations \cite{Czimek1}, Cheeger-Gromov theory in low regularity \cite{Czimek21}, the canonical foliation on null hypersurfaces in low regularity \cite{CzimekGraf1} and global elliptic estimates for spacelike maximal hypersurfaces.
\end{abstract}

\date{\today}
\maketitle
\setcounter{tocdepth}{2}
\tableofcontents

\section{Introduction}

\subsection{Einstein vacuum equations and the Cauchy problem of general relativity}

A Lorentzian $4$-manifold $(\mathcal{M},{\bf g})$ is called a \emph{vacuum spacetime} if it solves the Einstein vacuum equations
\begin{align} \label{EinsteinVacuumEquationsIntroJ}
\mathbf{Ric} = 0,
\end{align}
where $\mathbf{Ric}$ denotes the Ricci tensor of the Lorentzian metric ${ \bf g}$. The Einstein vacuum equations are invariant under diffeomorphisms, and therefore one considers equivalence classes of solutions. Expressed in general coordinates, \eqref{EinsteinVacuumEquationsIntroJ} is a non-linear geometric coupled system of partial differential equations of order $2$ for ${\bf g}$. In suitable coordinates, for example so-called \emph{wave coordinates}, it can be shown that \eqref{EinsteinVacuumEquationsIntroJ} is hyperbolic and hence admits an initial value formulation, see for example Chapter 10 in \cite{Wald} for background on the Cauchy problem of general relativity.\\

The corresponding initial data for the Einstein vacuum equations is given by specifying a triplet $(\Si,g,k)$ where $(\Si,g)$ is a Riemannian $3$-manifold and $k$ is a $g$-tracefree symmetric $2$-tensor on $\Si$ satisfying the \emph{constraint equations},
\begin{align} \begin{aligned}
\Rscal =& \, \vert k \vert_{g}^2-(\tr_g k)^2, \\
\Div k =& \, d(\tr_g k),
\end{aligned} \label{EQgeneralConstraints} \end{align}
where $\Rscal$ denotes the scalar curvature of $g$, $d$ the exterior derivative on $(\Si,g)$ and
\begin{align*}
\vert k \vert_g^2 := g^{ad}g^{bc} k_{ab} k_{cd}, \,\,\, \tr_g k := g^{ij}k_{ij}, \,\,\, (\Div k )_i := \nab^{j}k_{ij},
\end{align*}
where $\nab$ denotes the covariant derivative on $(\Si,g)$ and we use, as in the rest of this paper, the Einstein summation convention. In the future development $(\MM,\g)$ of such initial data $(\Si,g,k)$, $\Si \subset \MM$ is a spacelike hypersurface with induced metric $g$ and second fundamental form $k$.\\


For our purposes, it suffices to consider initial data posed on \emph{maximal} hypersurfaces, that is, satisfying $\tr k =0$, see also \cite{BartnikMaximal}. In this case, we say that $(\Si,g,k)$ is a \emph{maximal initial data set}, and the constraint equations \eqref{EQgeneralConstraints} reduce to
\begin{align*}
\Rscal =& \vert k \vert_{g}^2, \\
\Div k =& 0, \\
\tr_g k =&0.
\end{align*}

\subsection{Weak cosmic censorship and the bounded $L^2$ curvature theorem} 

One of the main open questions in general relativity is the so-called \emph{weak cosmic censorship conjecture} formulated by Penrose in 1969, see \cite{PenroseConjecture}.
\begin{conjecture}[Weak cosmic censorship conjecture] \label{conj2}
Generically, all singularities forming in the context of gravitational collapse are covered by black holes.
\end{conjecture}

In the pioneering work \cite{Chr9}, Christodoulou proves the weak cosmic censorship conjecture for the vacuum-scalar field \emph{in spherical symmetry}. In Christodoulou's proof, a low regularity control of the Einstein equations is essential for analysing the dynamical formation of black holes. This strongly suggests that a crucial step to prove the weak cosmic censorship in the absence of symmetry is to control the Einstein vacuum equations in very low regularity. \\

We remark that in the $(1+1)$-setting of spherical symmetry, Christodoulou bounds the regularity of initial data in a low scale-invariant BV-norm. Outside of spherical symmetry, however, this BV-norm is not suitable anymore and regularity should be measured with respect to $L^2$-based spaces; we refer the reader to the introduction of \cite{KRS}.\\






A breakthrough result in the low regularity control of the Einstein equations in absence of symmetry is the \emph{bounded $L^2$ curvature theorem} by Klainerman-Rodnianski-Szeftel \cite{KRS}. Before stating it, we define the volume radius of a Riemannian $3$-manifold.

\begin{definition}[Volume radius] Let $(\Si,g)$ be a Riemannian $3$-manifold, and let $r>0$ be a real. The \emph{volume radius of $\Si$ at scale $r$} is defined by
\begin{align*}
r_{vol}(\Si,r) := \inf\limits_{p \in \Si} \inf\limits_{0<r'<r} \frac{\mathrm{vol}_g( B_g(p,r') )}{r'^3},
\end{align*}
where $B_g(p,r')$ denotes the geodesic ball of radius $r'$ centered at $p \in \Si$.
\end{definition}

The following theorem is proved in \cite{KRS}, see also the companion papers \cite{J1}-\cite{J5}. We state a more technical version in Section \ref{SECliteratureResults2}, see Theorem \ref{THMsmalldataL2details}.
\begin{theorem}[The bounded $L^2$ curvature theorem, version 1] \label{thm:smallboundedL2thm} Let $(\Si, g,k)$ be asymptotically flat, maximal initial data for the Einstein vacuum equations such that $\Si \simeq \RRR^3$. Assume further that for some $\varep>0$, 
\begin{align*}
\Vert \RRRic \Vert_{L^2(\Si)} \leq \varep, \,\,\, \Vert k \Vert_{L^2(\Si)} + \Vert \nab k \Vert_{L^2(\Si)} \leq \varep \text{ and } r_{vol}(\Si,1) \geq \frac{1}{2}.
\end{align*}
Then:
\begin{enumerate}
\item {\bf $L^2$-regularity.} There is a universal constant $\varep_0>0$ such that if $0<\varep< \varep_0$, then the maximal globally hyperbolic development $(\MM,\g)$ of the initial data $(\Si,g,k)$ contains a foliation $(\Si_t)_{0\leq t \leq 1}$ of maximal spacelike hypersurfaces defined as level sets of a time function $t$ such that $\Si_0 = \Si$ and for $0 \leq t \leq 1$, 
\begin{align*}
\Vert \Rbf \Vert_{L^\infty_t L^2(\Si_t)} \lesssim \varep, \,\, \Vert k \Vert_{L^\infty_t L^2(\Si_t)} + \Vert \nab k \Vert_{L^\infty_t L^2(\Si_t)} \lesssim \varep, \,\,\inf\limits_{0\leq t \leq 1} r_{vol}(\Si_t,1) \geq \frac{1}{4}.
\end{align*}
\item {\bf Propagation of smoothness.} Smoothness of the initial data is propagated into the spacetime, and the spacetime remains smooth up to $\Si_1=\{t=1\}$.
\end{enumerate}
\end{theorem}

{\em Remarks on Theorem \ref{thm:smallboundedL2thm}.}

\begin{enumerate}
\item As Theorem \ref{thm:smallboundedL2thm} is local in nature due to the finite speed of propagation for \eqref{EinsteinVacuumEquationsIntroJ}, we do not specify here further the \emph{asymptotic flatness} condition, see also Remark 2.3 in \cite{KRS}.
\item Theorem \ref{thm:smallboundedL2thm} is primarily to be understood as a continuation result for smooth solutions of the Einstein vacuum equations, see Remark 1.2 in the introduction of \cite{KRS}. This holds similarly for the results of this paper.
\item The proof of Theorem \ref{thm:smallboundedL2thm} relies crucially on a plane wave representation formula for the wave equation on low regularity spacetimes developed in \cite{J1}-\cite{J5}. This plane wave representation formula is constructed as a Fourier integral operator which necessitates the assumption $\Si \simeq \RRR^3$.
\end{enumerate}

However, Christodoulou's work \cite{Chr9} as well as related results on the formation of trapped surfaces \cite{ChrFormationNonSpherical} \cite{KlRodTrapped} \cite{KlLukRod} \cite{LukAn} and gravitational impulses \cite{LukRod1} \cite{LukRod2} consider initial data posed on null hypersurfaces rather than on a spacelike hypersurface as assumed in Theorem \ref{thm:smallboundedL2thm}. This motivates the study of the Cauchy problem of general relativity in low regularity with initial data posed on null hypersurfaces. 

\subsection{The spacelike-characteristic Cauchy problem of general relativity} In this paper, we consider the \emph{spacelike-characteristic Cauchy problem of general relativity}, where initial data is posed on
\begin{enumerate}
\item a compact spacelike maximal hypersurface with boundary $\Si \simeq \ol{B(0,1)} \subset \RRR^3$, 
\item the outgoing null hypersurface $\HH$ emanating from $\pr \Si$,
\end{enumerate}
satisfying straight-forward compatibility conditions on $\pr \Si$, see Section 7.6 in \cite{ChruscielPaetz2} for example for details. Local existence for the spacelike-characteristic Cauchy problem for smooth initial data follows from \cite{BruhatExistence} \cite{Rendall} \cite{LukChar}, see also Proposition \ref{thm:LocalExistenceMixedMaximal}.

\begin{remark} In general, initial data posed on a null hypersurface is subject to constraint equations, namely the so-called \emph{null constraint equations}, see for example \cite{CzimekGraf1}. We do not state them as they do not play a role in this paper. \end{remark}

\begin{figure}[h!]
  \centering
  \begin{subfigure}[b]{0.4\linewidth}
    \includegraphics[height=2.8cm]{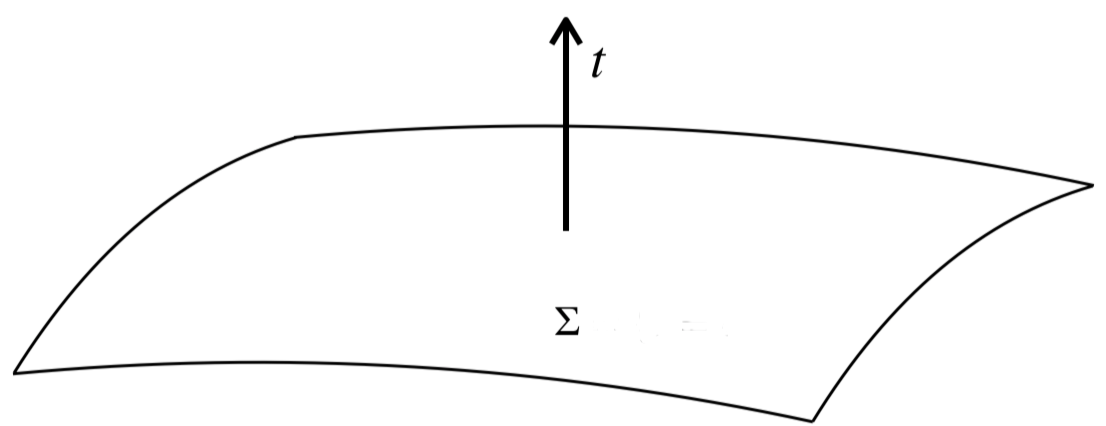}
    \caption{The spacelike Cauchy hypersurface of Theorem \ref{thm:smallboundedL2thm}.}
  \end{subfigure}
  \hspace{1cm}
  \begin{subfigure}[b]{0.4\linewidth}
    \includegraphics[width=\linewidth]{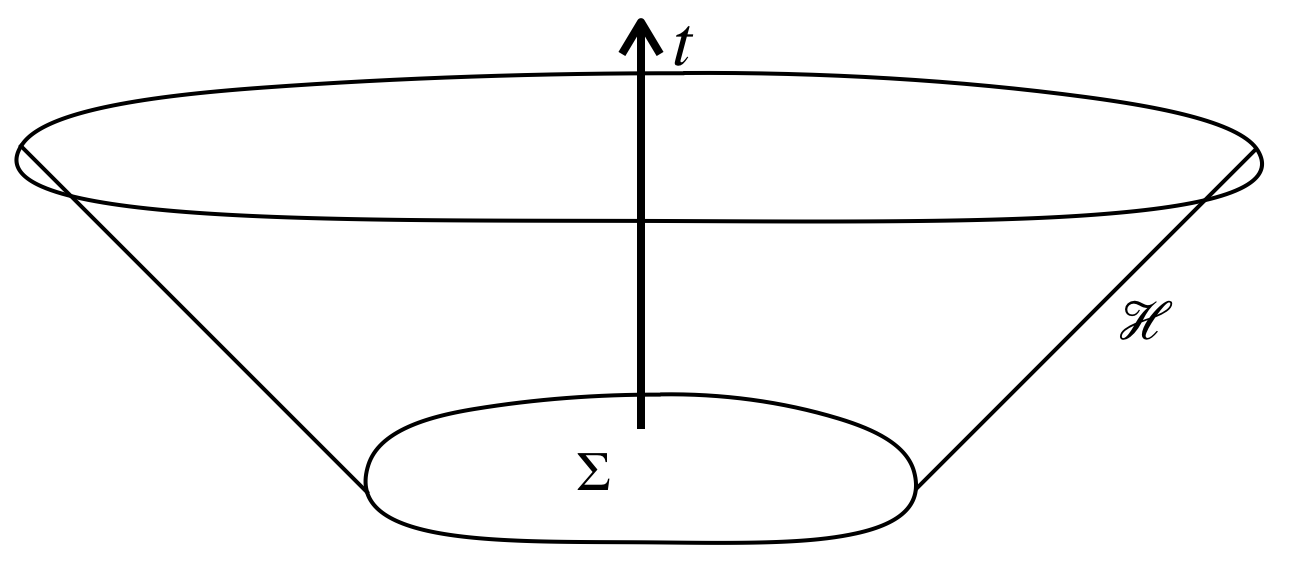}
    \caption{The spacelike-characteristic Cauchy hypersurface of Theorem \ref{thm:MainResultIntro1}.}
  \end{subfigure}
  \label{fig:IVPS}
\end{figure}

The next theorem is a rough version of our main result, see Theorem \ref{MixedTheoremVersion2} for a more precise statement.
\begin{theorem}[Main Theorem, version 1] \label{thm:MainResultIntro1} Consider initial data for the spacelike-characteristic Cauchy problem and let $(\MM,\g)$ denote its maximal globally hyperbolic development. Assume that for some real $\varep>0$,
\begin{align*}
\Vert \RRRic \Vert_{L^2(\Si)} \leq \varep, \,\, \Vert k \Vert_{L^2(\Si)} + \Vert \nab k \Vert_{L^2(\Si)} \leq \varep, \,\, r_{vol}(\Si,1/2) \geq 1/4, \,\, \mathrm{vol}_g(\Si) \leq 8 \pi,
\end{align*}
where $\RRRic$ and $k$ denote the intrinsic Ricci curvature and second fundamental form of $\Si \subset \MM$, respectively, and $\nab$ denotes the induced covariant derivative on $\Si$. Assume further that with respect to the so-called canonical foliation by spacelike $2$-spheres $(S_v)_{v\geq1}$ on $\HH$, see Definition \ref{DEFcanonicalFoliation}, it holds that
\begin{align} \begin{aligned}
\Vert \a \Vert_{L^2(\HH)} + \Vert \be \Vert_{L^2(\HH)} + \Vert \rh \Vert_{L^2(\HH)} + \Vert \si \Vert_{L^2(\HH)} + \Vert \betab \Vert_{L^2(\HH)} \leq& \,\varep, \\
\norm{ \trchi -\frac{2}{v} }_{L^\infty(\HH)} + \norm{ \trchib +\frac{2}{v} }_{L^\infty(\HH)} \leq&\, \varep,
\end{aligned} \label{EQRicciAssumptionsIntro1} \end{align}
where $(\a,\be,\rh,\si,\betab)$ denote null components on $\HH$ of the Riemann curvature tensor $\Rbf$ of $(\MM,\g)$, and $\trchi$ and $\trchib$ denote the two null expansions on $\HH$; see Section \ref{SECfoliationOfNullHypersurfaces} for definitions. Then:
\begin{enumerate}
\item {\bf $L^2$-regularity.} There is a universal constant $\varep_0>0$ such that if $0<\varep< \varep_0$, then $(\MM,\g)$ contains a foliation $(\Si_t)_{1\leq t \leq 2}$ by maximal spacelike hypersurfaces defined as level sets of a time function $t$ with $\Si_1 = \Si$ such that for $1\leq t \leq2$, 
\begin{align*}
\pr \Si_t = S_t
\end{align*}
and
\begin{align*}
&\Vert \Rbf \Vert_{L^\infty_t L^2(\Si_t)} \lesssim \varep, \,\, \Vert k \Vert_{L^\infty_t L^2(\Si_t)} + \Vert \nab k \Vert_{L^\infty_t L^2(\Si_t)} \lesssim \varep,\\
& \inf\limits_{1\leq t \leq 2} r_{vol}(\Si_t,1/2) \geq \frac{1}{8}, \,\, \mathrm{vol}_g(\Si_t) \leq 32 \pi.
\end{align*}
\item {\bf Propagation of regularity.} Smoothness of the initial data is propagated into the spacetime, and the spacetime remains smooth up to $\Si_2=\{t=2\}$.\end{enumerate}
\end{theorem}

{\em Remarks.}
\begin{enumerate}
\item Theorem \ref{thm:MainResultIntro1} assumes only initial data bounds at the level of curvature in $L^2$ and makes no symmetry assumptions. Until now, in the available literature the Cauchy problem for the Einstein vacuum equations with initial data on null hypersurfaces outside of symmetry is studied under the assumption of either
\begin{itemize}
\item higher regularity of the full initial data, see for example \cite{Rendall}, \cite{BruhatChruscielIVP}, \cite{ChruscielPaetz}, 
\item higher regularity of specific components of the initial data, see for example \cite{LukChar}, \cite{LukRod1}, \cite{LukRod2}. More precisely, in \cite{LukRod1} the null curvature component $\a$ is only assumed to be a distribution while $\be, \rh, \si$ and $\betab$ are assumed to be controlled up to two angular derivatives in $L^2$.
\end{itemize}
\item The assumed geometric control \eqref{EQRicciAssumptionsIntro1} of the canonical foliation $(S_v)_{v\geq1}$ on $\HH$ is essential for the regularity of the spacetime. In the authors' companion paper \cite{CzimekGraf1}, it is shown that assuming small bounded $L^2$ curvature flux on $\HH$ (in the geodesic foliation) and further low regularity geometry bounds on the initial sphere $S_1 = \Si \cap \HH$, the canonical foliation $(S_v)$ exists for $1 \leq v \leq 2$ and satisfies  \eqref{EQRicciAssumptionsIntro1}.
\item The assumptions $r_{vol}(\Si,1/2) \geq 1/4$ and $\mathrm{vol}_g(\Si) \leq 8 \pi$ on $\Si$ are solely used to invoke the Cheeger-Gromov theory developed in Section \ref{SectionGlobalExistence}, see Theorem \ref{prop:CGestimation1}.
\item The proof of Theorem \ref{thm:MainResultIntro1} uses the \emph{bounded $L^2$ curvature theorem}, see Theorem \ref{thm:smallboundedL2thm}, and the \emph{extension procedure for the constraint equations} \cite{Czimek1}, see Theorem \ref{THMextensionConstraintsCZ1}, as \emph{black boxes}.
\item The methods developed in this paper and \cite{Czimek1} \cite{Czimek21} \cite{Czimek22} \cite{CzimekGraf1} appear promising for a future study of the \emph{characteristic Cauchy problem of general relativity} where initial data is posed on two transversally intersecting null hypersurfaces.
\end{enumerate}
In the next section, we rigorously define the geometric setup, state the main results and give an overview of the proof of Theorem \ref{thm:MainResultIntro1}.

\subsection{Acknowledgements} Both authors are very grateful to J\'er\'emie Szeftel for many interesting and stimulating discussions. The second author is supported by the ERC grant ERC-2016 CoG 725589 EPGR.

\section{Geometric setup and main results} In this section, we introduce the notation and main equations of this paper, state the precise version of our main theorem (see Section \ref{sec:StatementMainResultandNORMS}) and give an overview of its proof (see Section \ref{sec:ProofOfMainTheorem}). \\

Lowercase Latin letters range over $\{ 1,2,3 \}$ and uppercase Latin letters over $\{ 1,2\}$. Greek letters range over $\{0,1,2,3\}$. We tacitly use the Einstein summation convention. In an inequality, a constant $C_{\a_1, \cdots, \a_k}$ depends on the quantities $\a_1, \cdots, \a_k$. We say that a scalar function is smooth if it is $k$-times continuously differentiable for each integer $k\geq0$.

\subsection{Weyl tensors on Lorentzian $4$-manifolds} \label{SECWeylTensors} In this section, we define Weyl tensors and the Bel-Robinson tensor of a Weyl tensor following the introduction and Sections 7 and 8 of \cite{ChrKl93}. The Bel-Robinson tensor is used in this paper to estimate the curvature tensor, see Sections \ref{sec:BelRobinsonCalculus}, \ref{SectionBAImprovementMAIN} and \ref{SectionHigherRegularity}.

\begin{definition}[Weyl tensor] Let $(\MM,\g)$ be a vacuum spacetime. A $4$-tensor $\mathbf{W}$ is a \emph{Weyl tensor} if it has the same symmetries as the Riemann curvature tensor and is tracefree, that is,
\begin{subequations}
\begin{align*}
\begin{aligned}
\mathbf{W}_{\a \be \ga \de} = - \mathbf{W}_{\be \a \ga \de} =&- \mathbf{W}_{\a \be \de \ga}, & \mathbf{W}_{\a \be \ga \de}=& \mathbf{W}_{\ga \de \a \be}, \\
\mathbf{W}_{\a \be \ga \de}+ \mathbf{W}_{\a \ga \de \be} + \mathbf{W}_{\a \de \be \ga} =&0, & \mathbf{W}_{\be \de}= \mathbf{W}^{\a}_{\,\,\,\be \a \de} =&0.
\end{aligned}
\end{align*}
\end{subequations}
Let the left dual ${}^\ast \mathbf{W}$ of a Weyl-tensor $W$ be
\begin{align*}
{}^\ast \mathbf{W}_{\a \be \ga \de} := \half \in_{\a \be \mu \nu} \mathbf{W}^{\mu\nu}_{\,\,\,\,\,\, \ga \de}
\end{align*}
where $\in$ denotes the volume form on $(\MM,\g)$.
\end{definition}


\begin{definition}[Bel-Robinson tensor] \label{def:BelRobTensor} Let $\mathbf{W}$ be a Weyl tensor on a vacuum spacetime $(\MM,\g)$. The Bel-Robinson tensor of $\mathbf{W}$ is defined by
\begin{align*}
Q(\mathbf{W})_{\a \be \ga \de}:= \mathbf{W}_{\a \nu \ga \mu} \mathbf{W}_{\be \,\,\, \de}^{\,\,\,\nu \,\,\, \mu} + {}^\ast \mathbf{W}_{\a \nu \ga \mu} {}^\ast \mathbf{W}_{\be \,\,\, \de}^{\,\,\,\nu \,\,\, \mu}.
\end{align*}
\end{definition}

The following \emph{modified Lie derivative} takes Weyl tensors into Weyl tensors, see Lemma 7.1.2 in \cite{ChrKl93}. Together with the Bel-Robinson tensor, it is used to derive higher regularity energy estimates for the Riemann curvature tensor in Section \ref{SectionHigherRegularity}.
\begin{definition}[Modified Lie derivative] \label{DEFmodifiedLiederivative} Let $\mathbf{W}$ be a Weyl field and $X$ a vectorfield on a vacuum spacetime $(\MM,\g)$. Define the modified Lie derivative by
\begin{align*}
\widehat{\Lie}_X \mathbf{W} := \Lie_X \mathbf{W} - \half {}^{(X)}[\mathbf{W}] + \frac{3}{8} \tr_\g {}^{(X)} \pi \, \mathbf{W},
\end{align*}
where ${}^{(X)}\pi := \Lie_X \g$ and
\begin{align*}
{}^{(X)}[\mathbf{W}]_{\a \be \ga \de} := \half {}^{(X)}\pi^\mu_{\,\,\, \a} \mathbf{W}_{\mu \be \ga \de} + \half {}^{(X)}\pi^\mu_{\,\,\, \be} \mathbf{W}_{\a \mu \ga \de}+ \half {}^{(X)}\pi^\mu_{\,\,\, \ga} \mathbf{W}_{\a \be \mu \de} + \half {}^{(X)}\pi^\mu_{\,\,\, \de} \mathbf{W}_{\a \be \ga \mu}.
\end{align*}
\end{definition}

\subsection{Foliations on null hypersurfaces $\HH$} \label{SECfoliationOfNullHypersurfaces} Let $(\MM,\g)$ be a vacuum spacetime and let $\HH$ be an outgoing null hypersurface emanating from a spacelike $2$-sphere $(S_1,\gd)$. Let moreover $T$ be a given timelike vectorfield on $S_1$. In the following we introduce the geometric setup of foliations on $\HH$ following the notations and normalisations of \cite{CzimekGraf1} and~\cite{KlRod1}.

\begin{definition}[Geodesic foliation on $\HH$] \label{defGeodesicFoliation} Let $L$ be the unique $\HH$-tangential null vectorfield on $S_1$ with $\g(L,T)=-1$. Extend $L$ as null geodesic vectorfield onto $\HH$. Let $s$ be the affine parameter of $L$ on $\HH$ defined by
\begin{align*}
Ls=1 \text{ on } \HH, \,\, s \vert_{S_1} =1.
\end{align*}
Denote the level sets of $s$ by $S'_s$ and the geodesic foliation by $(S'_s)$.
\end{definition}

\begin{definition}[General foliations on $\HH$]\label{def:nulllapse} Let $v$ be a given scalar function on $\HH$. We denote the level sets of $v$ by $S_{v_0} = \{v= v_0\}$ and the corresponding foliation by $(S_v)$. We define the \emph{null lapse} $\Om$ of $(S_v)$ on $\HH$ by 
\begin{align}\label{eq:subt}
  \Omt := \Lt v.
\end{align}
\end{definition}

\begin{remark} The geodesic foliation of $\HH$ corresponds to $\Omt=1$. \end{remark}

\begin{definition}[Orthonormal null frame] \label{def:orthonormalFrame}
  Let $(S_v)$ be a foliation on $\HH$. Let $\Lbt$ be the unique null vector field on $\HH$ orthogonal to each $S_v$ and such that $\g(L,\Lb) = -2$. The pair $(\Lt,\Lbt)$ is called a \emph{null pair for the foliation $(S_v)$}. Let $(e_1,e_2)$ be an orthonormal frame tangential to each $S_v$. The frame $(\Lt,\Lbt,e_1,e_2)$ is called an \emph{orthonormal null frame for the foliation $(S_v)$}.
\end{definition}

Let $(S_v)$ be a foliation on $\HH$ and let $(L,\Lb,e_1,e_2)$ be an orthonormal null frame for $(S_v)$.
\begin{itemize}

\item Denote by $\gd$ and $\Nd$ the induced metric and covariant derivative on $S_v$,

\item For a given $S_v$-tangential $k$-tensor $W$, define
\begin{align*}
\Nd_L W_{A_1\dots A_k} := \Pi_{A_1}^{\,\,\,\,\, \be_1} \cdots \Pi_{A_k}^{\,\,\,\,\, \be_k} \D_L W_{\be_1 \dots \be_k},
\end{align*}
where $\Pi$ denotes the projection operator onto the tangent space of $S_v$ and $\D$ is the covariant derivative on $(\MM,\g)$.

\item Let the \emph{null connection coefficients} be defined by
  \begin{align*} \begin{aligned} 
    \chit_{AB} & := \g(\D_A \Lt, \etB),& \chibt_{AB} & := \g(\D_A \Lbt, \etB), \\
    \zet_A & := \half \g(\D_A \Lt, \Lbt),& \etabt_A & := \half \g(\D_\Lt \Lbt, \etA).
  \end{aligned} 
    \end{align*}
Further decompose $\chit$ and $\chibt$ into their trace and tracefree parts,
\begin{align*}
\begin{aligned}
\trchit & := \gd^{AB}\chit_{AB},& \chiht_{AB} & := \chit_{AB} -\half \trchit \gd_{AB}, \\
\trchibt & := \gd^{AB}\chibt_{AB},& \chibht_{AB} & := \chibt_{AB} - \half \trchibt \gd_{AB}.
\end{aligned}
\end{align*}

\item For a given Weyl tensor $\mathbf{W}$ on $(\MM,\g)$, define its null decomposition by 
\begin{align*}
\alphab_{AB}(\mathbf{W}) :=& \mathbf{W}_{A\Lb B \Lb}, & \betab_A(\mathbf{W}) :=& \half \mathbf{W}_{A\Lb\Lb L}, & \rho(\mathbf{W}) :=& \frac{1}{4} \mathbf{W}_{\Lb L \Lb L}, \\
\si(\mathbf{W}) :=& \frac{1}{4} {}^\ast \mathbf{W}_{\Lb L \Lb L}, & \be_A(\mathbf{W}) :=& \half \mathbf{W}_{AL\Lb L}, & \a_{AB}(\mathbf{W}):=& \mathbf{W}_{A L B L}.
\end{align*}
In particular, for the Riemann curvature tensor in a vacuum spacetime, we denote the null curvature components by
\begin{align*}
\alphab_{AB} :=& \Rbf_{A\Lb B \Lb}, & \betab_A :=& \half \Rbf_{A\Lb\Lb L}, & \rho :=& \frac{1}{4} \Rbf_{\Lb L \Lb L}, \\
\si :=& \frac{1}{4} {}^\ast \Rbf_{\Lb L \Lb L}, & \be_A :=& \half \Rbf_{AL\Lb L}, & \a_{AB}:=& \Rbf_{A L B L}.
\end{align*}

\item For $S_v$-tangent vectorfields $X$ define
\begin{align*}
\Divd X:= \Nd_A X^A, \,\,\, \Curld X := \in_{AB}\Nd^A X^B,
\end{align*}
where $\in_{AB}:={\bf{\in}}_{ABL \Lb}$.

\item Define on $\HH$ the positive definite metric $\mathbf{h}^v$ with respect to the foliation $(S_v)$ by
\begin{align*}
\mathbf{h}^v_{\a \be} := \g_{\a \be} + \half (L+\Lb)_\a (L+\Lb)_\be.
\end{align*}
For a given $k$-tensor $\mathbf{W}$ on $\MM$, let on $\HH$
\begin{align} \label{DefTensorNormHH}
\vert \mathbf{W} \vert_{\mathbf{h}^v}^2 := \mathbf{W}_{\a_1 \dots \a_k} \mathbf{W}_{\a_1' \dots \a_k'} \left( \mathbf{h}^v \right)^{\a_1 \a_1'} \dots \left(\mathbf{h}^v \right)^{\a_k \a_k'}.
\end{align}

\end{itemize}

In a vacuum spacetime, the following \emph{Ricci equations} hold, see \cite{ChrKl93},
\begin{align}\begin{aligned}
    \D_\Lt \Lt & = 0,& \D_\Lt \Lbt & = 2 \etabt_A \etA, \\ 
    \D_A \Lt & = \chit_{AB} \etB -\zet_A \Lt,& \D_A \Lbt & = \chibt_{AB} \etB + \zet_A \Lbt, \\
    \D_\Lt \etA & = \Nd_\Lt \etA + \etabt_A \Lt, & \D_A \etB & = \Nd_A\etB + \half \chit_{AB} \Lbt + \half \chibt_{AB} \Lt.
  \end{aligned} \label{eq:Null_Id} \end{align}

In the rest of this paper, we choose the orthonormal frame $(e_A)_{A=1,2}$ tangential to a foliation $(S_v)$ on $\HH$ to be \emph{Fermi propagated}, that is, satisfying $\Nd_L \etA=0$. \\

We turn to the definition of the \emph{canonical foliation} on $\HH$. 
\begin{definition}[Canonical foliation on $\HH$] \label{DEFcanonicalFoliation} Let $(S_v)$ be a foliation on $\HH$. We say that $(S_v)$ is the \emph{canonical foliation} on $\HH$ if $v \vert_{S_1} =1$ and
\begin{align*}
\Ld \log \Om =& -\Divd \zeta + \left( \rho - \half \chih \cdot \chibh \right) - \left( \overline{\rho} - \half \overline{\chih \cdot \chibh}\right),\\
\overline{\log \Om} =&0,
\end{align*}
where for scalar functions $f$, $\overline{f}$ denotes the average of $f$ on the $2$-sphere $S_v$.
\end{definition}

In \cite{CzimekGraf1}, it is shown that the canonical foliation is well-defined under the assumption of small $L^2$ curvature flux and small low regularity foliation geometry on the initial sphere $S_1$, see the introduction of \cite{CzimekGraf1} for more details and background on the canonical foliation.

\subsection{Foliations of the spacetime $\MM$ by spacelike maximal hypersurfaces} \label{SECspacelikefoliations} Let $t$ be a scalar function on a vacuum spacetime $(\MM,\g)$ whose level sets $\Si_t$ constitute a foliation by spacelike maximal hypersurfaces.
\begin{itemize}

\item Let $g$ denote the induced metric on $\Si_t$ and $\nab$ its covariant derivative. Let $\triangle$ denote the Laplace-Beltrami operator of $g$.

\item Let $e_0:=T$ denote the future-pointing timelike unit normal to $\Si_t$. Define the second fundamental form $k$ of $\Si_t$ by
\begin{align*}
k_{ij} = & -\g(\D_{i} T,e_j),
\end{align*}
where $(e_i)_{i=1,2,3}$ is an orthonormal frame tangent to $\Si_t$. Define the foliation lapse $n$ by 
$$n^{-2} := \g(\D t, \D t),$$ 
satisfying in particular,
\begin{align} \label{eqRelationTandDt}
T = -n \D t.
\end{align}
We remark that the deformation tensor ${}^{(T)}\pi := \Lie_T \g$ can be expressed as
\begin{align} \label{eq:deformationTensorRelations}
{}^{(T)}\pi_{\a \be} = -2 k_{\a \be} - n^{-1} \left( T_\a \nab_\be n + T_\be \nab_\a n \right).
\end{align}
Moreover, define the connection $1$-form $\mathbf{A}$ by
\begin{align*}
(\mathbf{A}_\mu)_{\a \be} := \g(\D_\mu e_\be, e_\a).
\end{align*}

\item Let $\in_{abc}:={\in}_{abcT}$, and for two symmetric $g$-tracefree $2$-tensors $V$ and $W$, and a vectorfield $X$ on $\Si_t$ define
\begin{align*}
\Div V_{i} :=& \nabla^j V_{ji}, \\
\Curl V_{ij} :=& \half \Big( \in_{ilm} \nab^l V^{m}_{\,\,\,\,\, j} +\in_{jlm} \nab^l V^{m}_{\,\,\,\,\, i}  \Big), \\
 (V \times W)_{ij} :=& \in_{i}^{\,\,\,ab} \in_j^{\,\,\, cd} V_{ac} W_{bd} + \frac{1}{3} (V \cdot W)g_{ij},\\
  (V \wedge W)_{i} :=& \in_i^{\,\,\, mn} V_m^{\,\,\, l} W_{ln},\\
 (X \wedge V)_{ij} :=& \in_i^{\,\,\, mn} X_m A_{nj} + \in_j^{\,\,\, mn} X_m A_{in}.
\end{align*}

\item For a Weyl tensor $\mathbf{W}$, define its \emph{electric-magnetic decomposition with respect to $T$} as follows,
\begin{align*}
E(\mathbf{W})_{ab} := \mathbf{W}_{aTbT}, \,\, H(\mathbf{W})_{ab} := {}^\ast \mathbf{W}_{aTbT}.
\end{align*}
In particular, for the Riemann curvature tensor $\Rbf$ of a vacuum spacetime, let
\begin{align*}
E_{ab}  := \R_{aTbT}, \,\, H_{ab}  := {}^\ast \R_{aTbT}.
\end{align*}
The $2$-tensors $E(\mathbf{W})$ and $H(\mathbf{W})$ are $\Si$-tangent, symmetric and $g$-tracefree, see Section 7.2 in \cite{ChrKl93}. By definition of the modified Lie derivative, see Definition \ref{DEFmodifiedLiederivative}, it holds that
\begin{align} \begin{aligned}
\Lieh_T E(\mathbf{W})=& E\lrpar{\hat{\Lie}_T \mathbf{W}}-k \times E(\mathbf{W})+2 n^{-1} \nab n \wedge H(\mathbf{W}),\\
\Lieh_T H(\mathbf{W})=& H\lrpar{\hat{\Lie}_T \mathbf{W}}-k \times H(\mathbf{W})-2 n^{-1} \nab n \wedge E(\mathbf{W}),
\end{aligned} \label{EQTrelationEH} \end{align}
where $\Lieh_T H(\mathbf{W})$ and $\Lieh_T E(\mathbf{W})$ are the $g$-tracefree parts of $\Lie_T H(\mathbf{W})$ and $\Lie_T E(\mathbf{W})$, respectively. Moreover, by definition of the Bel-Robinson tensor, see Definition \ref{def:BelRobTensor},
\begin{align} \label{EQEquivalenceNORMS1}
\vert E(\mathbf{W}) \vert^2 + \vert H(\mathbf{W}) \vert^2=Q(\mathbf{W})_{TTTT}.
\end{align}

\item Define on $\MM$ the Riemannian metric $\mathbf{h}^t$ by
\begin{align*}
\mathbf{h}^t_{\a \be} := \g_{\a \be} + 2 T_\a T_\be,
\end{align*}
and for $k$-tensors $\mathbf{W}$ on $\Si_t$, let
\begin{align} \label{EQdefinitionOFhNorm}
\vert \mathbf{W} \vert_{\mathbf{h}^t}^2 := \mathbf{W}_{\a_1 \dots \a_k} \mathbf{W}_{\a_1' \dots \a_k'} \left( \mathbf{h}^t \right)^{\a_1 \a_1'} \dots \left( \mathbf{h}^t \right)^{\a_k \a_k'}.
\end{align}
In particular, for Weyl tensors $\mathbf{W}$ it holds by Section 7 in \cite{ChrKl93} that
\begin{align} \label{EQEquivalenceNORMS2}
\vert \mathbf{W} \vert_{\mathbf{h}^t}^2 \les Q(\mathbf{W})_{TTTT} \les \vert \mathbf{W} \vert_{\mathbf{h}^t}^2.
\end{align}

\end{itemize}

The Einstein vaccuum equations imply the following \emph{structure equations of the maximal foliation}, see equations (1.0.11a)-(1.0.14d) in~\cite{ChrKl93}. We have the \begin{subequations} \emph{first variation equation},
  \begin{align}\label{eq:firstvar}
    \Lie_Tg_{ij} = -\half k_{ij},
  \end{align}
the \emph{second variation equation},
  \begin{align}\label{eq:sndvar}
    \D_Tk_{ij} = E_{ij} -n^{-1} \nab_i\nab_j n +k_{il}k_{\,\,\,j}^l,
  \end{align}
 the \emph{Gauss-Codazzi equation}
  \begin{align}
    \Div k_i =& 0, \label{eq:divk} \\
        \Curl k_{ij}  =& H_{ij}, \label{eq:curlk}
  \end{align}
  the \emph{maximality of $\Si_t$},
\begin{align} \label{eqtrkiszero}
\tr_g k =0,
\end{align}
the \emph{lapse equation},
  \begin{align}\label{eq:Deltan}
    \triangle n = n \vert k\vert_g^2,
  \end{align}
the \emph{traced Gauss equation}, 
 \begin{align}\label{eq:RicE}
  \Ric_{ij} = E_{ij}+k_{ia}k^{a}_j,
\end{align}
and the \emph{twice-traced Gauss equation},
 \begin{align} \label{eq:HamiltonianConstraint}
  \Rscal = \vert k \vert_{g}^2.
 \end{align}
\end{subequations}

With respect to a folation $(\Si_t)$ by maximal hypersurfaces, the Bianchi equations can be written as follows, see Proposition 7.2.1 in \cite{ChrKl93}. 
\begin{proposition}[Maxwell's equations for $E(\mathbf{W})$ and $H(\mathbf{W})$] \label{prop:MaxwellsEq1} Let $(\MM,\g)$ be a vacuum spacetime. Let $E(\mathbf{W})$ and $H(\mathbf{W})$ be the electric-magnetic decomposition of a Weyl tensor $\mathbf{W}$ relative to a maximal foliation $(\Si_t)$ on $\MM$. Assume that $\mathbf{W}$ satisfies the inhomogeneous Bianchi equations
\begin{align*}
\D^\a \mathbf{W}_{\a \be \ga \de} = J_{\be \ga \de}.
\end{align*}
Then, with $J^\ast_{\be \ga \de} := \half J_{\be \mu \nu} \in^{\mu \nu}_{\,\,\,\,\,\, \ga \de}$,
\begin{align} \begin{aligned}
\Div E(\mathbf{W}) =& +k \wedge H(\mathbf{W}) + J, \\
\Div H(\mathbf{W}) =& -k \wedge E(\mathbf{W}) + J^\ast, \\
- \hat{\Lie}_T H(\mathbf{W}) + \Curl E(\mathbf{W}) =& - n^{-1} \nab n \wedge E(\mathbf{W}) - \half k \times H(\mathbf{W}) - J^\ast, \\
\hat{\Lie}_T E(\mathbf{W}) + \Curl H(\mathbf{W}) =& -n^{-1} \nab n \wedge H(\mathbf{W}) + \half k \times E(\mathbf{W}) - J.
\end{aligned} \label{eq:BianchiEH} \end{align}
\end{proposition}

\emph{Remarks.}
\begin{enumerate}
\item In Appendix \ref{sec:AppendixProofOfLowRegEllipticEstimateForK}, we interprete \eqref{eq:BianchiEH} with \eqref{EQTrelationEH} as $3$-dimensional Hodge system for $E(\mathbf{W})$ and $H(\mathbf{W})$ on $\Si_t$ and prove elliptic estimates.

\item In particular, it follows by Proposition \ref{prop:MaxwellsEq1} that in a vacuum spacetime $(\MM,\g)$ where $\Rbf$ satisfies by the Bianchi equations
\begin{align*}
\D^\a \Rbf_{\a \be \ga \de} = 0,
\end{align*}
it holds that
\begin{align} \begin{aligned}
\Div E =& k \wedge H, \\
\Div H =& -k \wedge E, \\
- \hat{\Lie}_T H + \Curl E =& - n^{-1} \nab n \wedge E - \half k \times H, \\
\hat{\Lie}_T E + \Curl H =& -n^{-1} \nab n \wedge H + \half k \times E.
\end{aligned} \label{eq:BianchiForEandH} \end{align}

\end{enumerate}

The following commutator identity allows us to derive elliptic estimates for $T(n)$, see (18.4) in Appendix E in \cite{KRS} for a proof.
\begin{lemma}[Commutator identity] \label{LEMcommutatorTriangleDT} It holds on $\Si_t$ that for scalar functions $f$ on $\Si_t$,
\begin{align*}
[\triangle, T]f = 2 k \nab^2 f - 2 n^{-1} \nab n  \nab T(f) - \vert k \vert^2 T(f) + 2 n^{-1} k \nab n \nab f.
\end{align*}
\end{lemma}

\subsection{Spherical coordinates on $\Si$} \label{SECfoliationSpheres} Let $(\Si,g)$ be a given maximal spacelike hypersurface in a vacuum spacetime $(\MM,\g)$ diffeomorphic to the closed unit ball in $\RRR^3$, that is, $\Si \simeq \overline{B(0,1)} \subset \RRR^3$. Using this diffeomorphism we can define standard spherical coordinates $(r,\th^1,\th^2)$ with $r \in [0,1]$ on $\Si$. We denote the level sets of $r$ by $\SS_r$, and for two reals $1\leq r_1,r_2 \leq 2$, let $A(r_1,r_2)$ denote the coordinate annulus
\begin{align*}
A(r_1,r_2) := \{ p \in \Si : r_1 \leq r(p) \leq r_2 \}.
\end{align*}
Then:
\begin{itemize}
\item The metric $g$ can be expressed in coordinates $(r,\th^1,\th^2)$ for $r>0$ as
\begin{align*}
g = a^2 dr^2 + \gd_{AB} (b^A dr + d\th^A)(b^Bdr + d \th^B),
\end{align*}
where 
\begin{itemize}
\item $a$ is called the \emph{foliation lapse}, 
\item $\gd$ is called the \emph{induced metric} on $\SS_r$,
\item $b$ is called the $\SS_r$-tangent \emph{shift vector}.
\end{itemize}

\item Let $N$ be the outward pointing unit normal to $\SS_r$ and let $(e_1,e_2)$ denote an orthonormal frame tangent to $\SS_r$. Define the second fundamental form of $\SS_r$ for $r>0$ by 
$$\Theta_{AB} := g(\nab_A N,e_B).$$
We split $\Th$ into its trace and tracefree part,
\begin{align*}
\trTh := \gd^{AB}\Th_{AB}, \,\,\, \Thh_{AB} := \Th_{AB} - \half \trTh\gd_{AB},
\end{align*}
Further, in coordinates $(r, \th^1, \th^2)$ we can express for $r>0$,
\begin{align*}
N= \frac{1}{a} \pr_r - \frac{1}{a} b, \,\,\, \Th_{AB} = -\frac{1}{2a} \pr_r (\gd_{AB}) + \frac{1}{2a} (\Lied_b \gd)_{AB},
\end{align*}
where $\Lied$ denotes the Lie derivative on $\SS_r$.

\item Let $\Nd$ and $\Ld$ denote the induced covariant derivative and Laplace-Beltrami operator on $\SS_r$, respectively. We note the relations (see Chapter 3 in \cite{ChrKl93})
\begin{align*}
\nab_N N = -a^{-1} \Nd a, \,\,\, \nab_A N = \Th_{AB} e_B, \,\,\, \Div N = \tr \Th.
\end{align*}

\item we decompose the second fundamental form $k$ on $\Si$ into $\SS_r$-tangential tensors as follows,
\begin{align*}
\de := k_{NN}, \,\,\, \ep_A :=& k_{NA}, \,\,\, \eta_{AB} := k_{AB}.
\end{align*}
We note that $\tr \eta = - \de$ because $\tr k =0$ on $\Si$ by maximality. 

\end{itemize}

Using the above, we can decompose $\nab k$ as follows (see Sections 3.1 and 4.4 in \cite{ChrKl93}),
\begin{align*}
\nab_A k_{BC} =& \Nd_A \eta_{BC} + \Th_{AB} \ep_C + \Th_{AC} \ep_{B}, & \nab_N k_{NN} =& N(\de) + 2 a^{-1} \Nd a \cdot \ep,  \\
\nab_N k_{AB} =& \Nd_N \eta_{AB} - a^{-1} \Nd_A a \ep_B - a^{-1} \Nd_B a \ep_A, & \nab_A k_{NN} =& \Nd_A \de - 2 \Th_{AC} \ep^C,\\
\nab_N k_{NA} =& \Nd_N \ep_A + a^{-1}\Nd_C a \, \eta^{C}_{\,\,\, A} - a^{-1} \Nd_A a \, \de, &\nab_B k_{NA} =& \Nd_B \ep_A + \eta_{A}^{\,\,\, C} \Th_{CB} - \de \Th_{AB}.
\end{align*}

Then the Gauss-Codazzi equations \eqref{eq:divk} and \eqref{eq:curlk} imply that (see Section 4.4 in \cite{ChrKl93})
\begin{subequations}
\begin{align}
N(\de) + \Divd \ep =& -2a^{-1} \Nd a \cdot \ep + \eta \cdot \Th - \de \tr \Th, \label{eq:FirstOfDivEq} \\
\Nd_N \ep_B + (\Divd \eta)_B =& -a^{-1} \Nd_C a \eta^{C}_{\,\,\, B} + a^{-1} \de \Nd_B a  - \tr \Th \ep_B - \Th_{BC} \ep^C, \label{eq:NepNabDeRelation} \\
\Nd_B \de - \Nd_N \ep_B=& \in^{A}_{\,\,\, NB} H_{NA} + 2 \Th_{BC} \ep^C + a^{-1} \Nd_{C}a \eta^C_{\,\,\,B} - a^{-1} \de \Nd_B a, \label{eq:FirstofCurleq} \\
\Nd_B \ep_A -\Nd_N \eta_{AB} =& \in^{C}_{\,\,\, NB} H_{AC} - \eta_{A}^{\,\,\, C} \Th_{CB} + \de \Th_{AB} + a^{-1} \Nd_C a \eta^{C}_{\,\,\,B} \\
&+ a^{-1}\Nd_C a \eta^C_{\,\,\, A}, \nonumber  \\
\Nd_B \eta_{CA} - \Nd_{A} \eta_{CB} =& \in^N_{\,\,\, AB} H_{CN} + \Th_{CB} \ep_A - \Th_{CA} \ep_B. \label{eq:divEtaFoundation}
\end{align}
\end{subequations}
Note that \eqref{eq:divEtaFoundation} implies in particular
\begin{align}\label{eq:DivEtaNabDeRelation}
(\Divd \eta)_A = - \Nd_A \de + \in^N_{\,\,\,AB} H^B_{\,\,\, N} + \Th_{AB} \ep^B + \tr \Th \ep_A.
\end{align}

\subsection{Relations between foliations on $\MM$ and $\HH$} Let $(\MM,\g)$ be a vacuum spacetime, let $\Si \simeq \ol{B(0,1)}$ be a spacelike maximal hypersurface, and let $\HH$ be the outgoing null hypersurface emanating from $S_1:= \pr \Si$. Let 
\begin{itemize}
\item $(\Si_t)_{t\geq1}$ be a foliation on $\MM$ by maximal spacelike hypersurfaces given as level sets of a time function $t$ such that $\Si_1=\Si$. Let $T$ denote the unit normal to $\Si_t$.
\item $(S_v)_{v\geq1}$ be a foliation on $\HH$ by spacelike $2$-spheres such that $S_1 = \{v=1\}$. Let $(L,\Lb,e_1,e_2)$ be an orthonormal null frame of $(S_v)$.
\end{itemize}
Assume furthermore that for $t\geq1$,
\begin{align*}
\pr \Si_t = S_t, \text{ i.e. } t=v \text{ on } \HH.
\end{align*}

\begin{definition} \label{definitionSLOPE} Let the \emph{slope} $\nut$ on $\HH$ be defined by
\begin{align} \label{def:nut}
\nut := & - \g(\Lbt, T).
\end{align}
\end{definition}

\begin{remark} Using Lemma \ref{lemma:TNrelations} below, the above definition of $\nu$ in \eqref{def:nut} is equivalent to $\nut^{-1} := - \g(\Lt,T)$. Thus, by Definition \ref{defGeodesicFoliation}, it follows that on $S_1$ we have the normalisation
\begin{align*}
\nut = 1.
\end{align*}
\end{remark}

The proof of the next lemma is left to the reader.
\begin{lemma} \label{lemma:TNrelations} On $\HH$ it holds that
  \begin{align}\label{eq:TNLLnut}
    \begin{aligned}
      T = & \half \nut \Lt + \half \nut^{-1}\Lbt, & N = & \half \nut \Lt - \half \nut^{-1}\Lbt, \\
      \Lt = & \nut^{-1}(T +  N), & \Lbt = & \nut (T - N),
    \end{aligned}
  \end{align}
and moreover on $\HH$,
  \begin{align*}
    \Th_{AB} & = \half \nut \chit_{AB} - \half \nut^{-1}\chibt_{AB}, \\
    \eta_{AB} & = -\half \nut \chit_{AB} - \half \nut^{-1}\chibt_{AB},\\
    \de  = - \tr\eta &= \half\nut\trchit + \half\nut^{-1}\trchibt. 
  \end{align*}
 \end{lemma}

\begin{lemma}[Slope equation] \label{lemma:slopeEquation} On $\HH$ it holds that
  \begin{align}\label{eq:slope}
    \nut^{-1}\Nd_A\nut = -\ep_A + \zet_A.
  \end{align}
\end{lemma}

\begin{proof} Using \eqref{eq:Null_Id} and Lemma \ref{lemma:TNrelations}, we have
\begin{align*}
\nut^{-1} \Nd_A \nut =& -\nut^{-1} \Nd_A (\g(\Lbt,T)) \\
=& - \nut^{-1} \big(\g(\D_{A} \Lbt, T) + \g(\Lbt, \D_A T) \big) \\
=& -\nut^{-1} \big(\g(\chib_{AB} e_B + \zeta_A \Lb, T) - \g(\Lbt, k_{Aj}e_j) \big)\\
=& \zeta_A +\nut^{-1} \ep_A \g(\Lb,N)\\
=&\zeta_A - \ep_A.
\end{align*}
This finishes the proof of \eqref{eq:slope}.
\end{proof}

In addition to the slope equation of Lemma \ref{lemma:slopeEquation}, we have the next transport equation for $\nu$ on $\HH$. It is used in Sections \ref{SectionBAImprovementMAIN} and \ref{SectionHigherRegularity} to estimate $T(n)$ on spacelike hypersurfaces.
\begin{lemma}[Transport equation for $\nu$ on $\HH$] \label{lemTransportNUalongL} It holds on $\HH$ that
\begin{align*}
L(\nu) = - n^{-1} N(n) - \de.
\end{align*}
\end{lemma}

\begin{proof} We have by \eqref{eq:Null_Id} and Lemma \ref{lemma:TNrelations} that
\begin{align*}
L(\nu) =&-L\g(\Lb,T) \\
=& -\g(\D_L \Lb, T) - \g(\Lb, \D_L T)\\ 
=& -\nu^{-1} \g(\Lb,\D_{T+N} T) \\
=& - \g(T-N,\D_{T+N} T) \\
=& -n^{-1} N(n) - \de,
\end{align*}
where we used that $\D_T T = - n^{-1} \nab n$. This finishes the proof of Lemma \ref{lemTransportNUalongL}.
\end{proof}

Further, we note that the lapse $n$ can be expressed on $\pr \Si_t$ as follows.
\begin{lemma} It holds on $\pr \Si_t$ that
\begin{align} \label{EQNboundaryIdentity}
n = \nu^{-1} \Om^{-1}.
\end{align}
\end{lemma}

\begin{proof} Indeed, by \eqref{eq:subt} and Lemma \ref{lemma:TNrelations},
\begin{align*}
\Om = L(v) = L(t) = \nu^{-1} T(t) = \nu^{-1} n^{-1},
\end{align*}
where we used that $t=v$ on $\HH$ and $T(t)= n^{-1}$.
\end{proof}

The next lemma follows by the definitions in Sections \ref{SECWeylTensors} and \ref{SECfoliationOfNullHypersurfaces}, and Lemma \ref{lemma:TNrelations}; the proof is left to the reader.
\begin{lemma} \label{LEMQLTTTrelations} The following identity holds on $\HH$,
\begin{align*} \begin{aligned}
Q(\Rbf)_{LTTT} =& \frac{1}{4}\nu^3 \vert \a \vert^2 + \frac{3}{2}\nu \vert \be \vert^2 + \frac{3 }{2}\nu^{-1}(\rho^2 + \si^2)+ \frac{1}{2}\nu^{-3} \vert \betab \vert^2.
\end{aligned}
\end{align*}
\end{lemma}

\subsection{Integration on $\HH$ and norms} In this section we define integration on null hypersurfaces $\HH$ and norms on $\Si$ and $\HH$.

\begin{definition}[Integration on $\HH$] \label{DEFintegrationH} Let $(S_v)_{1\leq v \leq 2}$ be a foliation on a null hypersurface $\HH$. For scalar functions $f$ on $\HH$, let
\begin{align*}
\int\limits_\HH f := \int\limits_{1}^{2} \left(\, \int\limits_{S_v} \Om^{-1} f d\mu_\gd \right)dv,
\end{align*}
where the integral over $S_v$ is with respect to the induced metric $\gd$ and $\Om := L(v)$ denotes the null lapse of $(S_v)$. 
\end{definition}

\begin{definition}[Norms on $\HH$] Let $(S_v)_{1\leq v \leq 2}$ be a foliation on a null hypersurface $\HH$. Let $1 \leq p < \infty$ be a real and let $F$ be an $S_v$-tangent tensor on $\HH$. Define for integers $m\geq0$,
  \begin{align*}
    \norm{F}_{L^2(\HH)} :=& \bigg(\int_\HH |F|^2\bigg)^{1/2}, \\
    \norm{F}_{L^\infty_v L^p(S_v)} :=& \sup\limits_{1\leq v \leq 2} \norm{F}_{L^p(S_v)},\\
        \norm{F}_{L^\infty_v L^\infty(S_v)} :=& \sup\limits_{1\leq v \leq 2} \norm{F}_{L^\infty(S_v)},\\
        \norm{F}_{L^\infty_v H^{1/2}(S_v)} :=& \sup\limits_{1\leq v \leq 2} \norm{F}_{H^{1/2}(S_v)},
   \end{align*}
   where the fractional Sobolev spaces $H^{1/2}(S_v)$ are defined in Section \ref{SECcalculusOnS}. Further, for spacetime tensors $\mathbf{W}$, define
  \begin{align*}
\norm{\mathbf{W}}_{L^\infty(\HH)} := \sup\limits_{\HH} \vert \mathbf{W} \vert_{\mathbf{h}^v},
\end{align*}
where $\mathbf{h}^v$ denotes the Riemannian metric on $\HH$ associated to the foliation $(S_v)$, see \eqref{DefTensorNormHH}. 
\end{definition}

\textbf{Notation.} More generally, for ease of presentation, we omit the domain of integration when it is clear over what interval the integration takes place.

\begin{definition}[Norms on $\MM$] Let $(\MM, \g)$ be a vacuum spacetime foliated by spacelike maximal hypersurfaces $\Si_t$ given as level sets of a time function $t$ on $\MM$. For $\Si_t$-tangential tensors $F$ define
  \begin{align*}
\norm{F}_{L^\infty_t L^2(\Si_t)} := \sup\limits_{t} \Vert F \Vert_{L^2(\Si_t)},
\end{align*}
and for spacetime tensors $\mathbf{W}$ let
\begin{align*}
\norm{\mathbf{W}}_{L^\infty_t L^2(\Si_t)} :=& \sup\limits_{t} \lrpar{ \,\int\limits_{\Si_t} \vert \mathbf{W} \vert^2_{\mathbf{h}^t}}^{1/2}, \\
\Vert \mathbf{W} \Vert_{L^\infty(\Si_t)} :=& \sup\limits_{\Si_t} \vert \mathbf{W} \vert_{\mathbf{h}^t},
\end{align*}
where $\mathbf{h}^t$ denotes the Riemannian metric on $\Si_t$ associated to the foliation $(\Si_t)$, see \eqref{EQdefinitionOFhNorm}.
\end{definition}


\subsection{Initial data norms} \label{SECinitialDATAnorms} In this section, we introduce the types of regularity and initial data norms used in our main result. We have the following definition of weak regularity of $2$-spheres, see \cite{CzimekGraf1} and \cite{ShaoBesov}.

\begin{definition}[Weakly regular $2$-spheres]\label{def:weakreg} Let $0\leq N < \infty$ be an integer and $c>0$ be a real number. A Riemaniann $2$-sphere $(S,\gd)$ is a \emph{weakly regular sphere} with constants $N,c$ if 
 \begin{itemize}
  \item it can be covered by $N$ coordinate patches,
 \item there is a partition of unity $\eta$ adapted to the above coordinate patches, 
 \item there are functions $0\leq \etatt \leq 1$ which are compactly supported in the coordinate patches and equal to $1$ on the support of $\eta$,
 \item on each coordinate patch there exists an orthonormal frame $(e_1,e_2)$,
 \end{itemize}
 such that on each coordinate patch,
  \begin{align*}
   & c^{-1} \leq \sqrt{\det{\gd}} \leq c, \\
  & c^{-1} \vert \xi \vert^2 \leq \gd_{AB} \xi^A \xi^B \leq c \vert \xi \vert^2 \text{ for all } \xi \in \RRR^2, \\
   & |\pr_{x^A}\eta| + |\pr_{x^A }\pr_{x^B}\eta| + |\pr_{x^A} \etatt |\leq c, \\
    & \norm{\Nd\pr_{x^A}}_{L^2} + \norm{\Nd\etA}_{L^4} \leq c.
 \end{align*}
\end{definition}

{\bf Low regularity initial data norms.} Let $(S_v)$ be the canonical foliation on $\HH$. Assume that each $S_v$ is a weakly regular $2$-sphere with constants $N,c$. Define
\begin{align*}
\OO^{\Si}_0 :=& \norm{k}_{L^2(\Si)}+\norm{\nabla k}_{L^2(\Si)}, \\
 \RR^{\Si}_0 :=& \norm{\RRRic}_{L^2(\Si)}, \\
 \OO_0^\HH :=& \left\Vert \tr \chi - \frac{2}{v}\right\Vert_{L^\infty(\HH)}+  \left\Vert \tr \chib + \frac{2}{v} \right\Vert_{L^\infty(\HH)} +\norm{\Nd \tr \chi}_{L^\infty_v L^2(S_v)}+ \norm{\Nd \tr \chib}_{L^\infty_v L^2(S_v)}\\
 &+ \left\Vert \chih \right\Vert_{L^\infty_vL^4(S_v)} + \left\Vert \chibh \right\Vert_{L^\infty_v L^4(S_v)} +\norm{\zeta}_{L^\infty_vL^4(S_v)} + \norm{\zeta}_{L^\infty_v H^{1/2}(S_v)} \\
 &+ \Vert \Nd \Om \Vert_{L^\infty_v L^4(S_v)} + \Vert \Nd \Om \Vert_{L^\infty_v H^{1/2}(S_v)}+ \Vert \Om-1 \Vert_{L^\infty(\HH)}, \\
\RRt_0^\HH  := &\norm{\alphat}_{L^2(\HH)}+\norm{\betat}_{L^2(\HH)}+\norm{\rhot}_{L^2(\HH)}+\norm{\sigmat}_{L^2(\HH)}+\norm{\betab}_{L^2(\HH)}.
\end{align*}
Here $H^{1/2}(S_v)$ is an $L^2$-based fractional Sobolev space on $S_v$ bounding $1/2$ derivatives, see Definition \ref{DEFHSspaces}.

\begin{remark} In \cite{CzimekGraf1}, it is shown that the weak regularity of the $2$-spheres $S_v$ and the norm $\OO_0^\HH$ can be bounded by the $L^2$ curvature flux $\RRt_0^\HH$ and low regularity bounds on the geometry of the initial sphere $S_1= \Si \cap \HH$. We refer to \cite{CzimekGraf1} for more details.
\end{remark}

{\bf Higher regularity initial data norms.} As higher regularity initial data norms, we consider the following. For integers $m\geq1$, let
\begin{align*}
 \OO_m^\Si :=& \sum\limits_{\vert \a \vert \leq m+1 } \Vert \nab^{\a} k \Vert_{L^\infty(\Si)}, \\
\RR_m^\Si :=& \sum\limits_{\vert \a \vert \leq m } \Vert \D^{\a} \Rbf \Vert_{L^\infty(\Si)} + \sum\limits_{\vert \a \vert \leq m } \Vert \nab^{\a} \RRRic \Vert_{L^\infty(\Si)},\\
\OO_m^\HH :=&\sum\limits_{\vert \a \vert \leq m+1}\norm{\Nd^{\a} \left(\tr \chi -\frac{2}{v} \right)}_{L^\infty(\HH)} + \norm{\Nd^{\a} \left( \trchib + \frac{2}{v} \right)}_{L^\infty(\HH)} \\
&+ \norm{\Nd^{\a}\Nd \trchi}_{L^\infty(\HH)}+ \norm{\Nd^{\a}\Nd \trchib}_{L^\infty(\HH)}  + \norm{\Nd^{\a} \chih}_{L^\infty(\HH)} + \norm{\Nd^{\a} \chibh}_{L^\infty(\HH)} \\
&+ \Vert \Nd^{\a} \zeta \Vert_{L^\infty(\HH)} + \norm{\Nd^{\a} \Nd \zeta}_{L^\infty(\HH)}+\sum\limits_{\vert \a \vert + \vert \be \vert \leq m+2 } \Vert \Nd^{\a} \Om \Vert_{L^\infty(\HH)} + \Vert \Nd^{\be} L^{\a}(\Om) \Vert_{L^\infty(\HH)}, \\
\RR_m^\HH :=& \sum\limits_{\vert \a \vert \leq m} \Vert \D^{\a} \Rbf \Vert_{L^\infty(\HH)}.
\end{align*}

\begin{remark} \label{RemarkHigherSharpNot}
For ease of presentation of the proof of the higher regularity estimates, we choose $L^\infty$-based norms instead of $L^2$-based norms for higher regularity initial data. As consequence, the \emph{higher} regularity estimates of this paper are not sharp. However, they are sufficient for the proof of the main theorem, see Section \ref{sec:ProofOfMainTheorem}.
\end{remark}

\begin{remark} In \cite{CzimekGraf1}, it is shown that for integers $m\geq1$ the norm $\OO_m^\HH$ can be bounded by higher regularity curvature fluxes and bounds on the geometry of the initial sphere $S_1= \Si \cap \HH$. We refer to \cite{CzimekGraf1} for more details.
\end{remark}

\subsection{Main result}  \label{sec:StatementMainResultandNORMS} To state our precise main result, we first introduce the next definition.

\begin{definition}[Weakly regular balls] \label{DEFweakRegdisks} A Riemannian $3$-manifold with boundary $\Si$ is a \emph{weakly regular ball with constant $0<\CMD<1/2$} if there is a coordinate chart $\phi: \ol{B(0,1)} \to \Si$ such that on $\ol{B(0,1)}$,
\begin{align*}
&(1-\CMD) \vert \xi \vert^2 \leq g_{ij} \xi^i \xi^j \leq (1+\CMD) \vert \xi \vert^2 \text{ for all } \xi \in \RRR^2, \\
&\Vert \pr g_{ij} \Vert_{L^2(\ol{B(0,1)})}+ \Vert \pr^2 g_{ij} \Vert_{L^2(\ol{B(0,1)})} \leq \CMD
\end{align*}
and for all integers $m\geq1$ the metric components $g_{ij}$ can be estimated in higher regularity by the Ricci tensor as follows,
\begin{align*}
\Vert g_{ij}-e_{ij} \Vert_{H^{m+2}(\ol{B(0,1)})} \les \sum\limits_{\vert \a \vert \leq m} \Vert \nab^{\a} \RRRic \Vert_{L^2(\Si)} + \CMD,
\end{align*}
where $e_{ij}$ denotes the standard Cartesian components of the Euclidean metric.
\end{definition} 
Here, for integers $m\geq0$, $H^m(\ol{B(0,1)})$ denotes the standard $L^2$-based Sobolev spaces on $\ol{B(0,1)}$ bounding $m$ derivatives.\\

The following is the main result of this paper.

\begin{theorem}[Main theorem, version 2] \label{MixedTheoremVersion2}
Let there be given smooth initial data for the spacelike-characteristic Cauchy problem on $\Si$ and $\HH$ and 
let $(S_v)_{1\leq v \leq 2}$ denote the canonical foliation on $\HH$. Assume that the $2$-spheres $S_v$ are uniformly weakly regular with constants $N,c$ and that for some real $\varep>0$,
\begin{align}
\OO^{\Si}_0 + \RR^{\Si}_0 + \OO^\HH_0+\RRt^\HH_0 \leq \varep, \,\,\, 1/4 \leq r_{vol}(\Si,1/2) \leq 8, \,\,\, 2\pi \leq \mathrm{vol}_g(\Si) \leq 8 \pi. \label{EQmaintheoremsmallnessAssumptions}
\end{align}
Then the following holds.
\begin{enumerate}
\item {\bf $L^2$-regularity.} Let $0<\CMD < 1/2$ be a real number. There is a universal constant $\varep_0>0$ such that if $0< \varep< \varep_0$, then the maximal smooth globally hyperbolic development $(\MM,\g)$ of the given initial data contains a foliation $(\Si_t)_{1\leq t\leq 2}$ of maximal spacelike hypersurfaces defined as level sets of a time function $t$ with $\Si_1 = \Si$ and
\begin{align*}
&\pr \Si_{t} = S_{t} \text{ for } 1 \leq t \leq 2,\\
&\Si_t \text{ is a weakly regular ball with constant $\CMD$},
\end{align*}
and such that for $1 \leq t \leq 2$, 
\begin{align} \begin{aligned}
\Vert \RRRic \Vert_{L^\infty_t L^2(\Si_t)}+\Vert k \Vert_{L^\infty_t L^2(\Si_t)}+ \Vert \nab k \Vert_{L^\infty_t L^2(\Si_t)}+\Vert \D_T k \Vert_{L^\infty_t L^2(\Si_t)}  \lesssim& \, \varep, \\
 \Vert n-1 \Vert_{L^\infty_t L^\infty(\Si_t)}+\Vert \nab n \Vert_{L^\infty_t L^2(\Si_t)} + \Vert \nab^2 n \Vert_{L^\infty_t L^2(\Si)} \les& \, \varep, \\
1/8 \leq \inf\limits_{1\leq t \leq 2} r_{vol}(\Si_{t},1/2) \leq& \, 16,\\
\pi/2 \leq  \mathrm{vol}_g(\Si_t) \leq& \,32 \pi.
\end{aligned} \label{eqL2estimatesprecise} \end{align}
\item {\bf Propagation of smoothness.} Smoothness of the initial data is propagated into the spacetime, and the spacetime is smooth up to $\Si_{2}= \{t=2\}$. More specifically, for integers $m\geq1$, for $1\leq t \leq 2$,
\begin{align} \begin{aligned}
 \sum\limits_{\vert \a \vert \leq m } \Vert \D^{\a} \Rbf \Vert_{L^\infty_t L^2(\Si_t)} \leq& C\left(\OO_m^{\Si},\RR_m^\Si,\OO_m^{\HH},\RRt_m^\HH, m \right), \\
\sum\limits_{\vert \a \vert \leq m+1 } \Vert \D^{\a} \left( {}^{(T)}\pi\right) \Vert_{L^\infty_t L^2(\Si_t)}  \leq& C\left(\OO_m^{\Si},\RR_m^\Si,\OO_m^{\HH},\RRt_m^\HH, m \right),
\end{aligned} \label{eqHigherestimatesprecise} \end{align}
where ${}^{(T)}\pi := \Lie_T \g$ denotes the deformation tensor of $T$.
\end{enumerate}
\end{theorem}

\emph{Remarks.}
\begin{enumerate}

\item In the proof of Theorem \ref{MixedTheoremVersion2}, we derive $L^2$-bounds for $\Rbf$ using the Bel-Robinson tensor $Q(\Rbf)$, see Proposition \ref{PropEllipticEstimatesFORtriangle}, which in turns requires a trilinear estimate for the corresponding error term. It is due to this trilinear estimate that we need to invoke the bounded $L^2$ curvature theorem, see Theorem \ref{THMsmalldataL2details}. We note that for the proof of the higher regularity estimates $m\geq2$ of Theorem \ref{MixedTheoremVersion2}, the corresponding error term can be bounded by a classical Gr\"onwall argument. 

\item In Theorem \ref{MixedTheoremVersion2}, at the level of $L^2$-regularity, each hypersurface $\Si_t$ is by construction a weakly regular ball with constant $\CMD$ and hence admits global coordinates such that $g_{ij} \in H^2(\Si_t)$. However, because the existence of each of these global coordinates follows by an application of Cheeger-Gromov theory to $\Si_t$ (see Theorem \ref{prop:CGestimation1}), we have no control of regularity of the components $g_{ij}$ in the $t$-direction.

\item In \cite{CzimekGraf1}, it is shown that the weak regularity of the $2$-spheres $S_v$ and the norm $\OO_0^\HH$ can be bounded by the $L^2$ curvature flux $\RRt_0^\HH$ and low regularity bounds on the geometry of the initial sphere $S_1= \Si \cap \HH$. Analogously, the norms $\OO_m^\HH$ can be bounded by higher regularity curvature fluxes and bounds on the geometry of the initial sphere $S_1$. We refer to \cite{CzimekGraf1} for more details.

\item The regularity assumptions \eqref{EQmaintheoremsmallnessAssumptions} on the canonical foliation (succesfully established in \cite{CzimekGraf1} at the level of bounded $L^2$ curvature) are crucial for the proof of Theorem \ref{MixedTheoremVersion2}. However, the exact definition of the canonical foliation is not used in this paper. Thus, any other foliation with similar regularity properties could be used to prove Theorem \ref{MixedTheoremVersion2}. 

\item The higher regularity estimates proved in Theorem \ref{MixedTheoremVersion2} are not sharp, see also the previous Remark \ref{RemarkHigherSharpNot}. Nevertheless, they are sufficient for proof of Theorem \ref{MixedTheoremVersion2}, see below.

\end{enumerate}

\subsection{Proof of the main theorem} \label{sec:ProofOfMainTheorem} The proof of Theorem \ref{MixedTheoremVersion2} goes by a bootstrapping argument which we set up and prove in this section. Let $T \in [1,2]$ be defined as
\begin{align*}
T := \sup_{t^\ast \in [ 1,2] } \left\{ 
\begin{tabular}{l}
  \text{There is a time function } $1 \leq t\leq t^\ast$ \text{ as in  Theorem } \ref{MixedTheoremVersion2} \\
  \text{such that \eqref{eqL2estimatesprecise} and \eqref{eqHigherestimatesprecise} hold}
  \end{tabular}
 \right\}.
\end{align*}
In the following, we show that $T=2$ for $\varep>0$ sufficiently small. \\


{\bf Step 1.} \emph{It holds that $T>1$.} Indeed, this follows from the next local existence and continuation result. Its proof is given in Section \ref{Section:LocalExistence}.

\begin{proposition}[Classical local existence and continuation] \label{thm:LocalExistenceMixedMaximal} Consider smooth initial data for the spacelike-characteristic Cauchy problem on $\Si$ and $\HH$ together with a smooth canonical foliation $(S_v)_{v\geq1}$ on $\HH$. Then for a small real number $\tau>0$, the maximal globally hyperbolic development $(\MM,\g)$ contains a foliation by smooth spacelike maximal hypersurfaces $(\Si_t)_{1\leq t \leq 1+\tau}$ given as level sets of a smooth time function $t$ such that $\Si_1 = \Si$ and for each $1\leq t \leq 1+\tau$, 
$$\pr \Si_t = S_t.$$
Moreover, the foliation $(\Si_t)$ can be locally continued in a smooth fashion as long as the foliation $(\Si_t)$ and the spacetime remains smooth.
\end{proposition}

\emph{Remarks.}
\begin{itemize}
\item The existence of a local maximal foliation in Proposition \ref{thm:LocalExistenceMixedMaximal} follows by a classical perturbation argument of Bruhat \cite{PerturbationBruhat}, see Theorem \ref{theorem:BruhatPerturbation}.
\item We could explicitly formulate Proposition \ref{thm:LocalExistenceMixedMaximal} in terms of function spaces of finite regularity, but for ease of presentation we choose the smooth class, that is, $C^k$ for every integer $k\geq0$.
\end{itemize}


{\bf Step 2.} \emph{Assuming that a set of bootstrap assumptions holds up to $1<t^\ast_0 < 2$, we show that we can improve them for $\varep>0$ sufficiently small.}  Indeed, the next proposition is proved in Section \ref{SectionBAImprovementMAIN}.
\begin{proposition}[Improvement of bootstrap assumptions] \label{prop:BAimprovement} Let $(\MM,\g)$ be a vacuum spacetime with past boundary consisting of a compact spacelike maximal hypersurface with boundary $\Si \simeq \overline{B(0,1)} \subset \RRR^3$ and the outgoing null hypersurface $\HH$ emanating from $\pr \Si$. Assume that $\HH$ is foliated by the canonical foliation $(S_v)_{1\leq v \leq 2}$. Let $1< t^\ast_0 \leq 2$ be a real number, and let $t$ be a time function on $\MM$ such that its level sets $(\Si_t)_{1\leq t \leq t^\ast_0}$ are spacelike maximal hypersurfaces with $\Si_1 = \Si$, satisfying for each $1 \leq t \leq t^\ast_0$,
\begin{align*}
\pr \Si_t = S_{t}.
\end{align*}
Assume that for some small $\varep>0$,
\begin{align*}
\RR_0^\Si +\OO_0^\Si + \OO_0^\HH+ \RR_0^\HH  \leq \varep,
\end{align*}
and for some fixed, large $D>0$, for $1\leq t \leq t^\ast_0$, 
\begin{align*} \begin{aligned}
\Vert \RRRic \Vert_{L^\infty_t L^2(\Si_t)} \leq& \,D \varep, \\
\Vert k \Vert_{L^\infty_t L^2(\Si_t)} + \Vert \nab k \Vert_{L^\infty_t L^2(\Si_t)} + \Vert k \Vert_{L^\infty_t L^2(S_t)} \leq& \,D \varep, \\
\Vert \nu-1 \Vert_{L^\infty_t L^\infty(S_t)} + \Vert \Nd \nu \Vert_{L^\infty_t L^4(S_t)} + \Vert \Nd \nu \Vert_{L^\infty_t H^{1/2}(S_t)} \leq& \, D \varep, \\
1/8 \leq r_{vol}(\Si_t, 1/2) \leq& \,16, \\
\pi/2 \leq \mathrm{vol}_g(\Si_{t}) \leq& \, 32 \pi.
\end{aligned} 
\end{align*}
There exists a universal constant $\varep_0>0$ such that if $0<\varep<\varep_0$, then for $1\leq t \leq t^\ast_0$,
\begin{align*}
\Vert \RRRic \Vert_{L^\infty_t L^2(\Si_t)} \les& \, D'\varep, \\
\Vert k \Vert_{L^\infty_t L^2(\Si_t)} + \Vert \nab k \Vert_{L^\infty_t L^2(\Si_t)} +\Vert \D_T k \Vert_{L^\infty_t L^2(\Si_t)}+ \Vert k \Vert_{L^\infty_t L^2(\pr \Si_t)} \les& \,D'\varep, \\
\Vert \nu-1 \Vert_{L^\infty_t L^\infty(S_t)} + \Vert \Nd \nu \Vert_{L^\infty_t L^4(S_t)} + \Vert \Nd \nu \Vert_{L^\infty_t H^{1/2}(S_t)} \les& \,D'\varep, \\
\Vert n-1 \Vert_{L^\infty_tL^\infty(\Si_t)} + \Vert \nab n \Vert_{L^\infty_tL^2(\Si_t)}  + \Vert \nab^2 n \Vert_{L^\infty_tL^2(\Si_t)} \les& \,D'\varep, \\
1/8 <  r_{vol}(\Si_t, 1/2) <& \,16, \\
\pi/2 < \mathrm{vol}_g(\Si_{t}) <& \, 32 \pi,
\end{align*}
for a constant $0<D'<D$.
\end{proposition}


{\bf Step 3.} The following \emph{higher regularity estimates} are proved in Section \ref{SectionHigherRegularity}.

\begin{proposition}[Higher regularity estimates] \label{thm:HigherRegularitySpacetimeEstimates} Let $(\MM,\g)$ be a vacuum spacetime whose past boundary consists of a compact spacelike maximal hypersurface with boundary $\Si\simeq \ol{B(0,1)}$ and the outgoing null hypersurface $\HH$ of $\pr \Si$. Let $(S_v)_{1\leq v \leq 2}$ be the canonical foliation on $\HH$. Let $1 < t^\ast_0 \leq 2$  and assume that there is a time function $1\leq t \leq t^\ast_0$ in $\MM$ such that its level sets $\Si_t$ are maximal spacelike hypersurfaces with $\Si_1=\Si$ and such that for $1\leq t \leq t^\ast_0$ and a real number $0<\CMD<1/2$,
\begin{align*}
&\pr \Si_t = S_t,\\
&\Si_t \text{ is a weakly regular ball with constant $\CMD$},
\end{align*}
and assume moreover that for some $\varep>0$, for $1\leq t \leq t^\ast_0$,
\begin{align*} \begin{aligned}
\Vert \RRRic \Vert_{L^\infty_tL^2(\Si_t)} \les& \,\varep, \\
\Vert k \Vert_{L^\infty_tL^2(\Si_t)} + \Vert \nab k \Vert_{L^\infty_tL^2(\Si_t)}+ \Vert \D_T k \Vert_{L^\infty_tL^2(\Si_t)} + \Vert k \Vert_{L^\infty_tL^2(S_t)} \les& \,\varep, \\
 \Vert \nu -1 \Vert_{L^\infty_t L^\infty(S_t)}+\Vert \Nd \nu \Vert_{L^\infty_t L^4(S_t)}+\Vert \Nd \nu \Vert_{L^\infty_t H^{1/2}(S_t)} \les&\,\varep,\\
\Vert n-1 \Vert_{L^\infty_t L^\infty(\Si_t)} + \Vert \nab n \Vert_{L^\infty_t L^2(\Si_t)} + \Vert \nab^2 n \Vert_{L^\infty_tL^2(\Si_t)} \les& \,\varep.
\end{aligned}
\end{align*}

For $\varep>0$ and $\CMD>0$ sufficiently small, it holds that for all integers $m\geq1$, on $1\leq t \leq t^\ast_0$,
\begin{align*}
 \sum\limits_{\vert \a \vert \leq m } \Vert \D^{\a} \Rbf \Vert_{L^\infty_t L^2(\Si_t)} \leq& C\left(\OO_m^{\Si},\RR_m^\Si,\OO_m^{\HH},\RRt_m^\HH,m,\CMD \right), \\
\sum\limits_{\vert \a \vert \leq m+1 } \Vert \D^{\a}\left( {}^{(T)} \pi \right) \Vert_{L^\infty_t L^2(\Si_t)}  \leq& C\left(\OO_m^{\Si},\RR_m^\Si,\OO_m^{\HH},\RRt_m^\HH, m, \CMD \right),
\end{align*}
where ${}^{(T)} \pi := \Lie_T \g$ denotes the deformation tensor of $T$.
\end{proposition}

\emph{Remarks.}
\begin{enumerate}
\item The smallness of $\CMD>0$ and $\varep>0$ is only used in the proof of the case $m=1$.
\item The estimates for $m=1$ require a trilinear estimate based on the null structure of the Einstein vacuum equations, see Section~\ref{SEChigherEstimatesM1}. In Appendix \ref{SECtrilinearEstimateM1} we reduce this trilinear estimate to the $(m=1)$-estimates of \cite{KRS}.
\item The estimates for $m\geq2$ are based on a classical Gr\"onwall argument together with the estimates for $m=1$; see Section \ref{SECm2estimatesOUTLINE}. 
\end{enumerate}


{\bf Step 4.} By the higher regularity estimates of Proposition \ref{thm:HigherRegularitySpacetimeEstimates}, it follows that on each hypersurface $\Si_t$, the induced initial data is smooth. Thus by Proposition \ref{thm:LocalExistenceMixedMaximal} and Proposition \ref{prop:BAimprovement}, we can \emph{continue the spacetime such that \eqref{eqL2estimatesprecise} and \eqref{eqHigherestimatesprecise} hold beyond $T$ for each $T<2$}, which yields a contradiction to the maximality of $T$. This concludes the proof of Theorem \ref{MixedTheoremVersion2}.

\subsection{Organisation of the paper}

The paper is organised as follows.
\begin{itemize}
\item In Section \ref{Calculus}, we recall calculus prerequisites, elliptic estimates and energy estimates.
\item In Section \ref{SectionBAImprovementMAIN}, we improve the bootstrap assumptions, see Proposition \ref{prop:BAimprovement}.
\item In Section \ref{SectionHigherRegularity}, we show higher regularity estimates, see Proposition \ref{thm:HigherRegularitySpacetimeEstimates}.
\item In Section \ref{Section:LocalExistence}, we prove classical local existence and smooth continuation of spacetimes with maximal foliations, see Proposition \ref{thm:LocalExistenceMixedMaximal}.
\item In Section \ref{Section:ProofOfCG}, we prove existence of global coordinates on $\Si$, see Theorem \ref{prop:CGestimation1}.
\item In Appendix \ref{sec:AppendixProofOfLowRegEllipticEstimateForK}, we prove global elliptic estimates for Hodge systems on $\Si$.
\item In Appendix \ref{SECproofTRACEEST}, we prove the trace estimate of Lemma \ref{PROPtraceEstimate}.
\item In Appendix \ref{secComparisonAppendix}, we prove comparison estimates between two maximal foliations on a spacetime $\MM$.
\item In Appendix \ref{SECtrilinearEstimateM1}, we discuss the trilinear estimate for the higher regularity estimates $m=1$, see Proposition \ref{PropTrilinearM1}.
\end{itemize}


\section{Calculus inequalities and prerequisite results} \label{Calculus}

\subsection{Calculus on $S$} \label{SECcalculusOnS} In this section, we recapitulate calculus prerequisites on Riemannian $2$-spheres $(S,\gd)$ that satisfy the weak regularity condition of Definition \ref{def:weakreg}. The next lemma is proved for example in \cite{ShaoBesov}.
\begin{lemma}[Sobolev inequalities on $S$] \label{lemma:CalculusOnGat1} Let $(S,\gd)$ be a weakly regular Riemannian $2$-sphere with constants $N,c$. Then it holds that for each tensor $F$ on $S$,
\begin{align*}
\Vert F \Vert_{L^\infty(S)} \lesssim& \Vert \Nd F \Vert_{L^4(S)} + \Vert F \Vert_{L^2(S)},
\end{align*}
where the constant depends only on $N,c$.
\end{lemma}

We introduce the following fractional Sobolev spaces on $S$.
\begin{definition}[Fractional Sobolev spaces on $S$] \label{DEFHSspaces} Let $(S,\gd)$ be a Riemannian $2$-sphere and let $s \in \RRR$. For tensors $F$ on $S$, define the norm
\begin{align*}
\norm{F}_{H^s(S)} := \Vert (1-\Ld)^{s/2} F \Vert_{L^2(S)},
\end{align*}
where the fractional Laplace-Beltrami operator is defined by standard spectral decomposition, see \cite{ShaoBesov}.  \end{definition}

The next are straight-forward properties of fractional Sobolev spaces (see for example Section 2 and Theorem 3.6 in \cite{ShaoBesov}, and Lemma 3.5 and Appendix B in \cite{CzimekGraf1}).
 \begin{lemma}[Properties of fractional Sobolev spaces] \label{LEMproductEstimatesBesov} Let $(S,\gd)$ be a weakly regular Riemannian $2$-sphere with constants $N,c$. Let $F_1$ and $F_2$ be two tensors on $S$. Then it holds that
\begin{align*}
\Vert F_1 F_2 \Vert_{H^{1/2}(S)} \les \lrpar{\Vert F_1 \Vert_{L^\infty(S)}+ \Vert \Nd F_1 \Vert_{L^2(S)}} \Vert F_2 \Vert_{H^{1/2}(S)}.
\end{align*}
and further for tensors $F$ on $S$,
\begin{align*}
\Vert \Nd F \Vert_{H^{-1/2}(S)} \les \Vert F \Vert_{H^{1/2}(S)},
\end{align*}
where the constants depend only on $N,c$.
\end{lemma}

\subsection{Calculus on $\Si$} \label{sec:CalculusOnSi} In this section, we recall calculus prerequisites on compact Riemannian $3$-manifolds with boundary $\Si \simeq \ol{B(0,1)} \subset \RRR^3$.\\

The following standard Sobolev inequalities follow for example from Section 3 in \cite{J3}.
\begin{lemma}[Sobolev inequalities on $\Si$] \label{LEMsobolevEmbeddingSigmaT} Let $F$ be a tensor on $\Si \simeq \overline{B(0,1)} \subset \RRR^3$ and let $g$ be a Riemannian metric on $\Si$ such that
\begin{align} \label{EQC0reg1def}
\frac{1}{4} \vert \xi \vert \leq g_{ij}\xi^i \xi^j \leq 2 \vert \xi \vert^2 \text{ for all } \xi \in \RRR^2.
\end{align}
Then
\begin{align*}
\Vert F \Vert_{L^\infty(\Si)} \les& \Vert F \Vert_{L^2(\Si)} + \Vert \nab F \Vert_{L^2(\Si)} + \Vert \nab^2 F \Vert_{L^2(\Si)}, \\
\Vert F \Vert_{L^6(\Si)} \les& \Vert F \Vert_{L^2(\Si)} + \Vert \nab F \Vert_{L^2(\Si)}.
\end{align*}
\end{lemma}
\begin{remark} In Lemma \ref{LEMsobolevEmbeddingSigmaT} we distinguish between the $C^0$-regularity \eqref{EQC0reg1def} and the stronger $H^2$-regularity of weak balls because the spacelike maximal hypersurfaces of the bounded $L^2$ curvature theorem admit the former but not the latter, see Theorem \ref{THMsmalldataL2details}. The hypersurfaces constructed in the main Theorem \ref{MixedTheoremVersion2} are weak balls.
\end{remark}

The next well-known trace estimates follow from the theory of function spaces in \cite{Adams} and \cite{ShaoBesov}. A proof is provided in Appendix \ref{SECproofTRACEEST}.
\begin{lemma}[Trace estimate] \label{Lemma:TraceEstimateH1toL4bdry} \label{PROPtraceEstimate} Let $F$ be a tensor on $\Si \simeq \overline{B(0,1)} \subset \RRR^3$ and let $g$ be a Riemannian metric on $\Si$ such that
\begin{align*}
\frac{1}{4} \vert \xi \vert \leq g_{ij}\xi^i \xi^j \leq 2 \vert \xi \vert^2 \text{ for all } \xi \in \RRR^2.
\end{align*}
Then it holds that
\begin{align} \begin{aligned}
\Vert F \Vert_{L^2(\pr \Si)} \lesssim& \Vert F \Vert_{L^2(\Si)} + \Vert \nab F \Vert_{L^2(\Si)}, \\
\Vert F \Vert_{L^4(\pr \Si)} \lesssim& \Vert F \Vert_{L^2(\Si)} + \Vert \nab F \Vert_{L^2(\Si)}.
\end{aligned} \label{EQtraceEstimate} \end{align}
Moreover, if $(\Si,g)$ is a weakly regular ball with constant $0<\CMD<1/2$, then
\begin{align} \label{EQtraceEstimate12prereq}
\Vert F \Vert_{H^{1/2}(\pr \Si)} \lesssim& \Vert F \Vert_{L^2(\Si)} + \Vert \nab F \Vert_{L^2(\Si)},
\end{align}
and for integers $m\geq1$,
\begin{align} \label{EQhigherregtraceEstimate}
\sum\limits_{\vert \a \vert \leq m} \Vert \Nd^{\a}F \Vert_{H^{1/2}(\pr \Si)} \les& \sum\limits_{\vert \a \vert \leq m+1} \Vert \nab^{\a} F \Vert_{L^2(\Si)}+ \sum\limits_{\vert \a \vert \leq m} \Vert \nab^{\a} \RRRic \Vert_{L^2(\Si)} + C_m\CMD.
\end{align}
\end{lemma}

The next lemma shows that in spherical coordinates the metric components are estimated away from $r=0$ by the constant $\CMD$ of the weakly regular ball, see Lemma 2.22 in \cite{Czimek1} for a proof.
\begin{lemma}[Estimates for metric components in spherical coordinates] \label{LemmaComparisonEstimates} Let $(\Si,g)$ be a weakly regular ball with constant $0<\CMD<1/2$. Then
\begin{align} \begin{aligned}
\Vert a-1 \Vert_{H^2(A(1/2,1))} + \Vert b \Vert_{H^2(A(1/2,1))} + \Vert \gd-\ga \Vert_{H^2(A(1/2,1))} \les& \CMD, \\
 \norm{ \tr \Th - \frac{2}{r}}_{H^1(A(1/2,1))} + \norm{\widehat{\Theta}}_{H^1(A(1/2,1))}+\norm{\Nd a}_{H^1(A(1/2,1))} \les& \CMD,
\end{aligned} \label{EQCoeffSmallSpherical} \end{align}
where $\ga$ denotes the standard round metric on $\SS_r$, and the notation for $a, b, \gd$ and $A(r_1, r_2)$ was introduced in Section \ref{SECfoliationSpheres}.
\end{lemma}


\subsubsection{Global elliptic estimates for the Laplace-Beltrami operator on $\Si$} \label{SECellipticEstimatesLaplacian} In this section we cite global elliptic estimates for the Laplace-Beltrami operator. They are a straight-forward generalisation of the global elliptic estimates of Appendix A in \cite{Czimek21}.

\begin{proposition} \label{PropEllipticEstimatesFORtriangle} Let $(\Si,g)$ be a weakly regular ball with constant $0<\CMD<1/2$. Then for any scalar function $f$ on $\Si$,
\begin{align*}
\sum\limits_{\vert \a \vert \leq 2} \Vert \nab^{\a} f \Vert_{L^2(\Si)} \les& \Vert \triangle f \Vert_{L^2(\Si)} + \Vert \Nd f \Vert_{H^{1/2}(\pr\Si)} + \Vert f \Vert_{L^2(\pr \Si)}, \\
\sum\limits_{\vert \a \vert \leq 3} \Vert \nab^{\a} f \Vert_{L^2(\Si)} \les& \Vert \nab\triangle f \Vert_{L^2(\Si)} +\Vert \triangle f \Vert_{L^2(\Si)} \\
&+\Vert \Nd^2 f \Vert_{H^{1/2}(\pr\Si)}+ \Vert \Nd f \Vert_{H^{1/2}(\pr\Si)} + \Vert f \Vert_{L^2(\pr \Si)}.
\end{align*}
Furthermore, for integers $m\geq1$, 
\begin{align*}
&\sum\limits_{\vert \a \vert \leq m+2} \Vert \nab^{\a} f \Vert_{L^2(\Si)} \\
\les& \sum\limits_{\vert \a \vert \leq m}\Vert \nab^{\a}\triangle f \Vert_{L^2(\Si)} + \sum\limits_{\vert \a \vert \leq m+1} \Vert \Nd^{\a} f \Vert_{H^{1/2}(\pr\Si)} \\
&+ C_m\lrpar{\sum\limits_{\vert \a \vert\leq m} \Vert \nab^{\a} \RRRic \Vert_{L^2(\Si)}+\CMD} \lrpar{\Vert \triangle f \Vert_{L^2(\Si)} + \Vert \Nd f \Vert_{H^{1/2}(\pr\Si)} + \Vert f \Vert_{L^2(\pr \Si)}}.
\end{align*}
\end{proposition}

\subsection{Energy estimates for the curvature tensor on $\MM$} \label{sec:BelRobinsonCalculus} The following classical energy estimate for Weyl tensors is proved in \cite{ChrKl93}, see the introduction and Lemma 8.1.1 therein. 

\begin{proposition}[Energy estimate for Weyl tensors] \label{PROPenergyEstimateWeyl} Let $(\MM,\g)$ be a vacuum spacetime bounded by two disjoint maximal spacelike hypersurfaces $\Si_1$ and $\Si_2$ and an outgoing null hypersurface $\HH$, and assume that $\MM$ is foliated by the spacelike level sets $(\Si_t)_{1\leq t \leq 2}$ of a time function $t$ such that $\{t=1\} = \Si_1$ and $\{t=2\} = \Si_2$. Let $T$ denote the timelike unit normal to $\Si_t$. Let further $\mathbf{W}$ be a Weyl tensor on $\MM$. Then it holds that
\begin{align} \begin{aligned}
\int\limits_{\Si_2} Q(\mathbf{W})_{TTTT}=& \int\limits_{\Si_1} Q(\mathbf{W})_{TTTT} + \int\limits_{\HH} Q(\mathbf{W})_{LTTT} \\
&- \int\limits_{\MM} \D^\mu Q(\mathbf{W})_{\mu TTT} - \int\limits_{\MM} \frac{3}{2} Q(\mathbf{W})_{\a \be TT} {}^{(T)}\pi^{\a\be},
\end{aligned} \label{lemma:IntegralidentityWeyl1} \end{align}
where ${}^{(T)}\pi := \Lie_T \g$ denotes the deformation tensor of $T$ and the integral over $\HH$ is defined in Definition \ref{DEFintegrationH}.
\end{proposition}

\subsection{An extension procedure for the constraint equations} \label{SECliteratureResults1} In this section, we cite in detail the exact statement of the extension procedure for the constraint equations \cite{Czimek1} which is used as \emph{black box} in this paper.

\begin{theorem}[An extension procedure for the constraint equations, \cite{Czimek1}] \label{THMextensionConstraintsCZ1} Let $(\bar g, \bar k)$ be initial data for the Einstein equations on a maximal hypersurface $\Si = B(0,1) \subset \RRR^3$. There exists a universal constant $\varep >0$ such that if
\begin{align*} 
\Vert \bar g- e \Vert_{H^2(B(0,1))} + \Vert  \bar k \Vert_{H^{1}(B(0,1))} < \varep,
\end{align*}
where $e$ denotes the Euclidean metric, then $(\bar g, \bar k)$ can be smoothly extended to asymptotically flat, maximal initial data $(g,k)$ on $\RRR^3$ with
 $$(g, k) \vert_{B(0,1)} = ( \bar g, \bar k),$$ 
which is bounded by
\begin{align*}
\Vert g- e \Vert_{H^2_{-1/2}(\RRR^3)} + \Vert  k \Vert_{H^{1}_{-3/2}(\RRR^3)} \lesssim \Vert \bar g-e \Vert_{H^2(B(0,1))} + \Vert \bar k \Vert_{H^{1}(B(0,1))}.
\end{align*}
Moreover, for integers $m\geq1$,
\begin{align*} 
\Vert g- e \Vert_{H^{m+2}_{-1/2}(\RRR^3)} + \Vert  k \Vert_{H^{m+1}_{-3/2}(\RRR^3)} \leq C \left( \Vert \bar g-e \Vert_{H^{m+2}(B(0,1))} + \Vert \bar k \Vert_{H^{m+1}(B(0,1))} \right),
\end{align*} 
where the constant $C>0$ depends only on $m$.
\end{theorem}

\begin{remark} In Theorem \ref{THMextensionConstraintsCZ1}, for integers $m\geq0$, $H^{m+2}_{-1/2}(\RRR^3)$ and $H^{m+1}_{-3/2}(\RRR^3)$ denote weighted $L^2$-based Sobolev spaces bounding $m$ coordinate derivatives and measuring asymptotic flatness, see \cite{Czimek1}. In particular, for $\varep>0$ sufficiently small, the global maximal initial data $(g,k)$ on $\RRR^3$ satisfies the assumptions of the bounded $L^2$ curvature theorem, see Theorem \ref{THMsmalldataL2details} below. \end{remark}

\subsection{The bounded $L^2$ curvature theorem} \label{SECliteratureResults2}  In this section, we cite in detail the bounded $L^2$ curvature theorem \cite{KRS} which is used as \emph{black box} in this paper. The theorem as stated below is a paraphrase of Theorems 2.4 and 2.5 in \cite{J1} and Theorem 2.18 in \cite{J3}.

\begin{theorem}[The bounded $L^2$ curvature theorem, version 2] \label{THMsmalldataL2details} Let $(\Si,g,k)$ be asymptotically flat maximal initial data for the Einstein vacuum equations such that $\Si \simeq \RRR^3$. Assume moreover there are global coordinates in which
\begin{align*}
\half e_{ij} \leq g_{ij} \leq \frac{3}{2} e_{ij},
\end{align*}
and that for some $\varep>0$,
\begin{align*}
\Vert \RRRic \Vert_{L^2(\Si)} \leq \varep, \,\,\, \Vert k \Vert_{L^2(\Si)} + \Vert \nab k \Vert_{L^2(\Si)} \leq \varep, \,\,\, r_{vol}(\Si,1/2) >1/4.
\end{align*}
Then.
\begin{enumerate}
\item \textbf{$L^2$-regularity.} There is a universal constant $\varep_0>0$ such that if $0< \varep< \varep_0$, then the maximal globally hyperbolic development $(\MM,\g)$ of the initial data $(\Si,g,k)$ contains a foliation of maximal spacelike hypersurfaces $(\Si_t)_{0\leq t \leq 1}$ with $\Si_0=\Si$ such that on each $\Si_t$,
\begin{align*}
\frac{1}{4} e_{ij} \leq g_{ij} \leq& \,2 e_{ij}, \\
\Vert \RRRic \Vert_{L^\infty_t L^2(\Si_t)} \les&\, \varep, \\
\Vert k \Vert_{L^\infty_tL^2(\Si_t)} + \Vert \nab k \Vert_{L^\infty_tL^2(\Si_t)}+ \Vert \D_T k \Vert_{L^\infty_tL^2(\Si_t)} +\Vert {\mathbf{A}} \Vert_{L^\infty_\ttt L^4(\Sitt_\ttt)} \les& \, \varep,\\
\Vert n-1 \Vert_{L^\infty(\MM)} + \Vert \nab n \Vert_{L^\infty(\MM)} + \Vert \nab^2 n \Vert_{L^\infty_tL^2(\Si_t)}+ \Vert \nab T( n) \Vert_{L^\infty_tL^2(\Si_t)} \les& \,\varep, \\
 r_{vol}(\Si_t,1/2) \geq& \, 1/8.
\end{align*}
Moreover, for each $\om \in \mathbb{S}^2$, there is a foliation $(\HH_{{}^\om u})_{{}^\om u \in \RRR}$ of $\MM$ by weakly regular (see remarks below) null hyperplanes $\HH_{{}^\om u}$ given as level sets of an optical function ${}^\om u$ such that
\begin{align*}
\sup\limits_{\om \in \SSS^2} \Vert \R \cdot L \Vert_{L^\infty_u L^2(\HH_{{}^\om u})} \lesssim \varep,
\end{align*}
where $L$ denotes the $\HH_{{}^\om u}$-tangential null vectorfield with $\g(T,L)=-1$. In addition, the following trilinear estimate holds,
\begin{align} \label{EQtrilinearStatement1}
\left\vert \, \int\limits_{\MM} Q(\R)_{ij TT} k^{ij} \right\vert \les& \, \varep\Vert \R \Vert_{L^\infty_{t} L^2(\Si_t)}^2 + \varep \Vert \R \Vert_{L^2(\MM)} \sup\limits_{\om \in \SSS^2} \Vert \R \cdot L \Vert_{L^\infty_{{}^\om u} L^2(\HH_{{}^\om u} )}.
\end{align}
\item \textbf{Higher regularity.} For integers $m\geq1$, it holds that
\begin{align*}
\sum\limits_{\vert \a \vert \leq m} \Vert \D^{\a} \R \Vert_{L^\infty_t L^2(\Si_t)} \les& \sum\limits_{\vert \a \vert \leq m} \Vert \nab^{\a} \RRRic \Vert_{L^2(\Si)} + \Vert \nab^{\a}\nab k \Vert_{L^2(\Si)},\\
\sum\limits_{\vert \a \vert \leq m+1} \Vert \D^{\a} \left( {}^{(T)}\pi  \right) \Vert_{L^\infty_t L^2(\Si_t)} \les& \sum\limits_{\vert \a \vert \leq m} \Vert \nab^{\a} \RRRic \Vert_{L^2(\Si)} + \Vert \nab^{\a}\nab k \Vert_{L^2(\Si)},
\end{align*}
where ${}^{(T)}\pi := \Lie_T \g$ denotes the deformation tensor of $T$.
\end{enumerate}
\end{theorem}

\emph{Remarks.}
\begin{enumerate}
\item Theorem \ref{thm:smallboundedL2thm} is the \emph{small data version} of the bounded $L^2$ curvature theorem. A corresponding large data version is obtained in \cite{KRS} by a rescaling procedure.
\item We refer to Definition 5.3 in \cite{KRS} for a definition of weakly regular null hypersurfaces. For the purposes of this paper, it suffices to note that weak regularity is sufficient for an application of Stokes' theorem as in Proposition \ref{PROPenergyEstimateWeyl}.
\item In Appendix \ref{SECtrilinearEstimateM1}, we give more details about the Yang-Mills formalism and wave parametrix construction of \cite{KRS} and use the estimates of that paper to prove a trilinear estimate, see in particular \eqref{EQparametrixEstimate0} and \eqref{EQparametrixEstimate2}.
\end{enumerate}

\section{Low regularity estimates} \label{SectionBAImprovementMAIN}

In this section we prove Proposition \ref{prop:BAimprovement}. Let $\varep>0$ be a real. Assume that
\begin{align} \label{eq:SmallData}
\OO_0^\Si +\RR_0^\Si +  \OO_0^\HH+ \RR_0^\HH  \leq \varep,
\end{align}
and further, for a large, fixed constant $D >0$, assume that for $1\leq t \leq t_0^\ast$ (where $t_0^\ast <2$),
\begin{align} \begin{aligned}
\Vert \RRRic \Vert_{L^\infty_t L^2(\Si_t)} \leq& \,D \varep, \\
\Vert k \Vert_{L^\infty_t L^2(\Si_t)} + \Vert \nab k \Vert_{L^\infty_t L^2(\Si_t)} + \Vert k \Vert_{L^\infty_t L^2(S_t)} \leq& \, D \varep, \\
\Vert \nu-1 \Vert_{L^\infty_t L^\infty(S_t)} + \Vert \Nd \nu \Vert_{L^\infty_t L^4(S_t)} + \Vert \nu^{-1} \Nd \nu \Vert_{L^\infty_t H^{1/2}(S_t)} \leq& \, D \varep, \\
1/4 \leq r_{vol}(\Si_t, 1/2) \leq& \, 8, \\
\pi/2 \leq \mathrm{vol}_g(\Si_{t}) \leq& \, 32 \pi.
\end{aligned}\label{eq:BAspacetimeEstimates01}
 \end{align}

In the following, we prove that for $\varep>0$ sufficiently small, for $1 \leq t \leq t_0^\ast$,
\begin{align*} \begin{aligned}
\Vert \RRRic \Vert_{L^\infty_t L^2(\Si_t)} \les& \,D' \varep, \\
\Vert k \Vert_{L^\infty_t L^2(\Si_t)} + \Vert \nab k \Vert_{L^\infty_t L^2(\Si_t)} + \Vert k \Vert_{L^\infty_t L^2(S_t)} \les& \,D' \varep, \\
\Vert \nu-1 \Vert_{L^\infty_tL^\infty(S_t)} + \Vert \Nd \nu \Vert_{L^\infty_t L^4(S_t)} + \Vert \nu^{-1} \Nd \nu \Vert_{L^\infty_t H^{1/2}(S_t)} \les& \,D' \varep, \\
1/4 <  r_{vol}(\Si_t, 1/2) <& \, 8, \\
\pi/2 < \mathrm{vol}_g(\Si_{t}) <& \, 32 \pi.
\end{aligned}
 \end{align*}
for a constant $0<D'<D$, and furthermore,
\begin{align*}
\Vert n-1 \Vert_{L^\infty_tL^\infty(\Si_t)} + \Vert \nab n \Vert_{L^\infty_tL^2(\Si_t)}  + \Vert \nab^2 n \Vert_{L^\infty_tL^2(\Si_t)} \les& \, \varep.
\end{align*}

\subsection{Overview of the proof of Proposition \ref{prop:BAimprovement}} In the following, we outline the main steps of the proof of Proposition \ref{prop:BAimprovement}. An important tool applied in the proof is the next theorem about the existence of global coordinates. Its proof using the Cheeger-Gromov theory of manifold convergence is given in Section \ref{Section:ProofOfCG}.

\begin{theorem}[Existence of global regular coordinates] \label{prop:CGestimation1} Let $(M,g)$ be a compact Riemannian $3$-manifold with boundary such that $M \simeq \overline{B(0,1)} \subset \RRR^3$, and assume that for two reals $\varep>0$ and $0< V < \infty$,
\begin{align*} \begin{aligned}
\Vert \RRRic \Vert_{L^2(M)} \leq \varep, \,\, \Vert \tr \Theta - 2 \Vert_{L^4(\pr M)}+ \Vert \widehat{\Theta}\Vert_{L^4(\pr M)} \leq \varep, \,\,  r_{vol}(M,1/2) \geq 1/4, \,\, \mathrm{vol}_g (M) \leq V,
\end{aligned} \end{align*}
where $\Th$ denotes the second fundamental form of $\pr M \subset M$. Then for every real number $0<\CMD<1/2$, there is an $\varep_0>0$ such that if $0< \varep < \varep_0$, then 
\begin{align*}
(M,g) \text{ is a weakly regular ball with constant }\CMD,
\end{align*}
that is,
\begin{enumerate}
\item {\bf $H^2$-regularity.} There is a coordinate chart $\phi: \overline{B(0,1)} \to M$ such that
\begin{align*}
&\Vert g_{ij} -e_{ij} \Vert_{H^2({B(0,1)})} \lesssim \CMD, \\
&(1-\CMD) \vert \xi \vert^2 \leq g_{ij} \xi^i \xi^j \leq (1+\CMD) \vert \xi \vert^2 \text{ for all } \xi \in \RRR^2. 
\end{align*}

\item {\bf Higher regularity.} For integers $m\geq1$, the following estimate for the coordinate components $g_{ij}$ holds,
\begin{align*}
\Vert g_{ij} -e_{ij} \Vert_{H^{m+2}({B(0,1)})} \les C_{V}\left( \sum\limits_{\vert \a \vert \leq m}\Vert \nab^{\a}\RRRic \Vert_{L^2(M)} +C_{m,V} \right).
\end{align*}
\end{enumerate}
\end{theorem}

We are now in position to give an overview of the proof of Proposition \ref{prop:BAimprovement}. It suffices to improve the bootstrap assumptions on $\Si_{t^\ast}$ for a fixed real $1\leq t^\ast \leq t^\ast_0$.

\begin{enumerate}

\item Let $0<\CMD<1/2$ be a small constant to be determined below. By the bootstrap assumptions \eqref{eq:BAspacetimeEstimates01} 
together with Theorem \ref{prop:CGestimation1}, we deduce that for $\varep>0$ sufficiently small, $\Si_{t^\ast}$ is a weakly regular ball with constant $0<\CMD<1/2$. For $\CMD >0$ sufficiently small, this directly improves the bootstrap assumptions on $\mathrm{vol}_g(\Si_{t^\ast})$ and $r_{vol}(\Si_{t^\ast}, 1/2)$.

\item For $0<\CMD<1/2$ and $\varep>0$ sufficiently small, the \emph{extension procedure for the constraint equations} (see Theorem \ref{THMextensionConstraintsCZ1}) can be applied to $\Si_{t^\ast}$. This yields an extension of the maximal initial data $(\Si_{t^\ast},g,k)$ to an asymptotically flat initial data set of size bounded by $\CMD$.

\item For $\CMD>0$ sufficiently small, we can subsequently apply \emph{backwards} the bounded $L^2$ curvature theorem (see Theorem \ref{thm:smallboundedL2thm}) to the above extended initial data set. This yields a foliation of the past of $\Si_{t^\ast}$ in $\MM$ by maximal hypersurfaces $(\wt{\Si}_{\ttt})_{0 \leq \ttt \leq t^\ast}$ with controlled foliation geometry. In particular, this foliation admits $\wt{\nab} \wt{n} \in L^\infty(\MM_{t^\ast})$ and a trilinear estimate, see Theorem \ref{THMsmalldataL2details}.

\item Using the $\CMD$-control of the foliation $\wt{\Si}_{\ttt}$ we can estimate $\Rbf$ using the Bel-Robinson tensor (see Proposition \ref{PropEllipticEstimatesFORtriangle}), relating the curvature flux through $\Si_{t^\ast}$ with the curvature fluxes through $\HH$ and $\Si$ which are in turn bounded by the $\varep$-small initial data norms. This improves the bootstrap assumption on the curvature flux on $\Si_{t^\ast}$. It is in this step that the control of $\wt{\nab}\wt{n} \in L^\infty(\MM_{t^\ast})$ and the trilinear estimate for the $(\wt{\Si}_{\ttt})_{0 \leq \ttt \leq t^\ast}$-foliation are needed.

\item The second fundamental form $k$ on $\Si_{t^\ast}$ satisfies a Hodge system and thus global elliptic estimates (see Corollary \ref{integralIDKEST}) improve the bounds on $\nab k \in L^2(\Si_{t^\ast})$. Here we use that the source terms for the Hodge system of $k$ depend only on already improved curvature terms. Moreover, in this step it is crucial to analyse the boundary integrals appearing in the global elliptic estimates for $k$. Indeed, they admit a special structure which allows to split them up into one part which \emph{bounds} the slope $\nu$ between $\HH$ and $\Si_{t^\ast}$ and the $L^2(\pr \Si_{t^\ast})$-norm of $k$, and another part which can be estimated by the assumed $\varep$-smallness of the null connection coefficients of the canonical foliation on $\HH$. 

\item The $L^2(\Si_{t^\ast})$-norm of $k$ is estimated using the previously improved $L^2(\pr \Si_{t^\ast})$-norm of $k$ and $L^2(\Si_{t^\ast})$-norm of $\nab k$. The estimate for $\D_T k$ in $L^2(\Si_{t^\ast})$ follows by the second variation equation \eqref{eq:sndvar}.

\item The bootstrap assumptions for $\nu$ on $\pr \Si_{t^\ast}$ are improved by using the slope equation \eqref{eq:slope} and the bounds mentioned in (5) together with the assumed $\varep$-smallness of the null connection coefficients of the canonical foliation.

\item The foliation lapse $n$ of the foliation $\Si_t$ is improved by global elliptic estimates applied to the maximal lapse equation, using that $k$ on $\Si_{t^\ast}$ and the boundary value $n= \Om^{-1} \nu^{-1}$ on $\pr \Si_{t^\ast}$ are improved in the previous steps.

\end{enumerate}

We remark that to compare the curvature fluxes through $\HH$ and $\Si$ with the initial data norms on $\HH$ and $\Si$, a comparison argument between the two maximal foliations $(\Si_t)_{1 \leq t \leq t^\ast}$ and $(\wt{\Si})_{0\leq \ttt \leq t^\ast}$ is needed, see Lemma \ref{lem:comparisonfoliationMM} and its proof in Appendix \ref{secComparisonAppendix}. This comparison argument requires the control $n -1 \in L^\infty(\MM_{t^\ast})$, and hence in the proof of Proposition \ref{prop:BAimprovement} below, we use the bootstrap assumptions to bound $n$ of size $D\varep$ before improving them.

\subsection{First consequences of the bootstrap assumptions} \label{sec:FirstConseqSPACETIME} We first remark that by the smallness assumption \eqref{eq:SmallData},
\begin{align*}
\left\Vert \tr \chib + \frac{2}{t} \right\Vert_{L^\infty_t L^\infty(S_t)}+ \left\Vert \tr \chi - \frac{2}{t} \right\Vert_{L^\infty_t L^\infty(S_t)} \lesssim& \, \varep, \\
\norm{\chih}_{L^\infty_tL^4(S_t)} + \norm{\chibh}_{L^\infty_t L^4(S_t)} \les& \, \varep.
\end{align*}

Therefore on the one hand, using that by Lemma \ref{lemma:TNrelations}
\begin{align*}
\tr \Th - \frac{2}{t} =& \half \nu \trchi - \half \nu^{-1} \trchib -\frac{2}{t} \\
=& \half (\nu-1) \trchi + \half \lrpar{\trchi -\frac{2}{t}} - \half (\nu^{-1}-1) \trchib - \half \lrpar{\trchib+\frac{2}{t}},
\end{align*}
we have that by \eqref{eq:SmallData} and \eqref{eq:BAspacetimeEstimates01},
 \begin{align}
\left\Vert \tr \Th - \frac{2}{t} \right\Vert_{L^\infty_t L^\infty(S_t)} \les&\, D \varep.\label{lem:FirstConsequencesSpacetime1}
\end{align}
On the other hand, using that by Lemma \ref{lemma:TNrelations}
\begin{align*}
\widehat{\Theta} = \half \nu \chih - \half \nu^{-1} \chibh,
\end{align*}
we get that by \eqref{eq:SmallData} and \eqref{eq:BAspacetimeEstimates01},
 \begin{align}
\left\Vert \widehat{\Theta} \right\Vert_{L^\infty_t L^4(S_t)} \les& \,\varep.\label{lem:FirstConsequencesSpacetime2}
\end{align}

\subsection{Weak regularity of $\Si_{t^\ast}$} \label{SECweakreguM0} From \eqref{eq:BAspacetimeEstimates01}, \eqref{lem:FirstConsequencesSpacetime1} and \eqref{lem:FirstConsequencesSpacetime2}, we have for $\varep>0$ sufficiently small that
\begin{align*} \begin{aligned}
\Vert \RRRic \Vert_{L^2(\Si_{t^\ast})} +  \Vert k \Vert_{L^2(\Si_{t^\ast})} + \Vert \nab k \Vert_{L^2(\Si_{t^\ast})} \les& D \varep, \\
\left\Vert \tr \Theta - \frac{2}{t^\ast} \right\Vert_{L^4(S_{t^\ast})}+ \norm{\widehat{\Theta}}_{L^4(S_{t^\ast})} \lesssim& D \varep, \\
r_{vol}(\Si_{t^\ast},1/2) \geq 1/4, \,\,\, \mathrm{vol}_g(\Si_{t^\ast}) \leq& 32\pi.
\end{aligned} 
\end{align*}
Therefore by Theorem \ref{prop:CGestimation1}, for any real $\CMD>0$, there is $\varep_0>0$ such that if $0<\varep<\varep_0$, then $\Si_{t^\ast}$ is a weakly regular ball with constant $\CMD$. Moreover, we can pick $\CMD>0$ and $\varep>0$ sufficiently small such that
\begin{align}\label{EqBounds2}
\max\limits_{i,j=1,2,3} \Big( \Vert g_{ij}- e_{ij} \Vert_{H^2(\Si_{t^\ast})} + \Vert k_{ij} \Vert_{H^1(\Si_{t^\ast})} \Big) \lesssim \CMD,
\end{align}
and further,
\begin{align*}
1/4< r_{vol}(\Si_{t^\ast},1/2) < 8, \,\, \pi/2 < \mathrm{vol}_g(\Si_{t^\ast}) < 32 \pi,
\end{align*}
which improves the bootstrap assumptions on $r_{vol}(\Si_{t^\ast},1/2)$ and $\mathrm{vol}_g(\Si_{t^\ast})$.

\subsection{Estimates for $n$ on $\Si_{t^\ast}$} \label{SECestimatesForLapseN} The lapse function $n$ is by  \eqref{eq:subt}, \eqref{eqRelationTandDt}, \eqref{EQNboundaryIdentity} and \eqref{eq:Deltan} a solution to the following elliptic boundary value problem,
\begin{align} \begin{aligned} 
\Delta n =& \,n \vert k\vert_g^2 &\text{ on } \Si_{t^\ast}, \\
n=& \,\nu^{-1} \Om^{-1} &\text{ on } \pr \Si_{t^\ast}.
\end{aligned} \label{eq:ellEQforN}
\end{align}

In this section, we prove that for $\varep>0$ sufficiently small,
\begin{align} \label{EQnEstimatesSummary}
\Vert n-1 \Vert_{L^\infty(\Si_{t^\ast})} + \Vert \nab n \Vert_{L^2(\Si_{t^\ast})} + \Vert \nab^2 n \Vert_{L^2(\Si_{t^\ast})} \les D\varep.
\end{align}

\begin{remark} In accordance with the continuity argument, we do not have any bootstrap assumptions on $n$ in \eqref{eq:BAspacetimeEstimates01}. \end{remark}

On the one hand, by \eqref{eq:SmallData} and \eqref{eq:BAspacetimeEstimates01}, the boundary value $n=\nu^{-1} \Om^{-1}$ satisfies for $\varep>0$ small
\begin{align*} 
\Vert n-1 \Vert_{L^\infty(\pr \Si_{t^\ast})} \les& \, \Vert \nu-1 \Vert_{L^\infty(\pr \Si_{t^\ast})} + \Vert \Om-1 \Vert_{L^\infty(\pr \Si_{t^\ast})} \\
\les& \, D\varep,\\
\Vert \Nd n \Vert_{L^2(\pr \Si_{t^\ast})} \les& \, \Vert \Nd \nu \Vert_{L^2(\pr \Si_{t^\ast})} + \Vert \Nd \Om \Vert_{L^2(\pr \Si_{t^\ast})} \\
\les& \, D\varep.
\end{align*}
Using further \eqref{eq:slope}, \eqref{eq:SmallData}, \eqref{eq:BAspacetimeEstimates01} and Lemmas \ref{LEMproductEstimatesBesov} and \ref{PROPtraceEstimate}, we get
\begin{align*} 
&\Vert \Nd n \Vert_{H^{1/2}(\pr \Si_{t^\ast})} \\
=& \left\Vert \frac{1}{\nu \Om} \left( \nu^{-1} \Nd \nu \right) + \frac{1}{\nu \Om^2} \Nd \Om \right\Vert_{H^{1/2}(\pr \Si_{t^\ast})} \\
\les&\lrpar{\norm{\Nd \left( \frac{1}{\Om \nu} \right)}_{L^2(\pr \Si_{t^\ast})} + \norm{\frac{1}{\Om \nu}}_{L^\infty(\pr\Si_{t^\ast})}+\norm{\Nd \left( \frac{1}{\Om^2 \nu} \right)}_{L^2(\pr \Si_{t^\ast})} + \norm{\frac{1}{\Om^2 \nu}}_{L^\infty(\pr\Si_{t^\ast})}} \\
&\qquad \cdot \lrpar{\norm{\nu^{-1} \Nd \nu}_{H^{1/2}(\pr \Si_{t^\ast})} + \norm{\Nd \Om}_{H^{1/2}(\pr \Si_{t^\ast})}} \\
\les& \norm{\nu^{-1} \Nd \nu}_{H^{1/2}(\pr \Si_{t^\ast})} + \norm{\Nd \Om}_{H^{1/2}(\pr \Si_{t^\ast})} \\
\les& D\varep.
\end{align*}

On the other hand, by \eqref{eq:BAspacetimeEstimates01}, \eqref{eq:ellEQforN}, Lemma \ref{LEMsobolevEmbeddingSigmaT} and H\"older's inequality, we have that
\begin{align*}
\Vert \triangle n \Vert_{L^2(\Si_{t^\ast})} =& \,\Vert n \vert k\vert^2 \Vert_{L^2(\Si_{t^\ast})} \\
\les&\, (1+ \Vert n-1 \Vert_{L^6(\Si_{t^\ast})}) \Vert k \Vert_{L^6(\Si_{t^\ast})}^2 \\
\les&\, (1+ \Vert n-1 \Vert_{L^2(\Si_{t^\ast})} + \Vert \nab n \Vert_{L^2(\Si_{t^\ast})} ) \left(\Vert k \Vert_{L^2(\Si_{t^\ast})} + \Vert \nab k \Vert_{L^2(\Si_{t^\ast})} \right)^2 \\
\les&\, (1+ \Vert n-1 \Vert_{L^2(\Si_{t^\ast})} + \Vert \nab n \Vert_{L^2(\Si_{t^\ast})} ) (D\varep)^2.
\end{align*}

Therefore by the elliptic estimates of Proposition \ref{PropEllipticEstimatesFORtriangle} together with the above, we get, for $\CMD>0$ and $\varep>0$ sufficiently small,
\begin{align*}
&\Vert n-1 \Vert_{L^2(\Si_{t^\ast})} + \Vert \nab n \Vert_{L^2(\Si_{t^\ast})} + \Vert \nab^2 n \Vert_{L^2(\Si_{t^\ast})}\\
 \les&\,\Vert \triangle n \Vert_{L^2(\Si_{t^\ast})} + \Vert \Nd n \Vert_{H^{1/2}(\pr \Si_{t^\ast})} \\
\les&\,  (1+ \Vert n-1 \Vert_{L^2(\Si_{t^\ast})} + \Vert \nab n \Vert_{L^2(\Si_{t^\ast})} ) (D\varep)^2 + D\varep \\
\les&\, D\varep,
\end{align*}
where we used the smallness of $\varep>0$ to absorb the term in the left-hand side. The estimate \eqref{EQnEstimatesSummary} follows then by the Sobolev inequality of Lemma \ref{LEMsobolevEmbeddingSigmaT}.

\subsection{Construction of a background foliation of $\MM_{t^\ast}$} \label{sec:BAimprovBackgroundFoliation} In this section, we apply \emph{backwards} the bounded $L^2$ curvature theorem to $\Si_{t^\ast}$ backwards to construct a background foliation of the past of $\Si_{t^\ast}$ in $\MM$, denoted by $\MM_{t^\ast}$. \\

By Theorem \ref{THMextensionConstraintsCZ1} and \eqref{EqBounds2}, for $\CMD>0$ sufficiently small, $(\Si_{t^\ast},g,k)$ can be extended to an asymptotically flat maximal initial data set on $\RRR^3$ which satisfies the assumptions of Theorem \ref{THMsmalldataL2details}. Subsequently, Theorem \ref{THMsmalldataL2details} yields that the following holds for $\CMD>0$ sufficiently small.

\begin{enumerate}
\item The spacetime region $\MM_{t^\ast}$ is foliated by spacelike maximal hypersurfaces $(\wt{\Si}_{\ttt})_{0 \leq \ttt \leq t^\ast}$ given as level sets of a time function $\ttt$ with $\Si_{t^\ast} = \wt{\Si}_{t^\ast} \cap \MM_{t^\ast}$ and satisfying for $0 \leq \ttt \leq t^\ast$,
\begin{align} \begin{aligned}
\Vert \wt{\RRRic} \Vert_{L^\infty_{\ttt} L^2(\wt{\Si}_{\ttt})} \les& \,\CMD, & \Vert \ktt \Vert_{L^\infty_\ttt L^2(\wt{\Si}_{\ttt})} + \Vert \wt{\nab} \ktt \Vert_{L^\infty_\ttt L^2(\wt{\Si}_{\ttt})}+ \Vert \wt{\mathbf{A}} \Vert_{L^\infty_\ttt L^4(\Sitt_\ttt)}\les&\, \CMD, \\
\Vert \R \Vert_{L^\infty_{\tilde t} L^2(\wt{\Si}_{\tilde t})}\les& \,\CMD, & \Vert \wt{n} -1 \Vert_{L^\infty_\ttt L^\infty(\wt{\Si}_{\ttt})} +\Vert \wt{\nab} \ntt \Vert_{L^\infty_{\ttt} L^\infty(\wt{\Si}_{\ttt})} \les& \,\CMD. 
\end{aligned} \label{EqBounds11}
\end{align}
Let $\Ttt$ denote the future-pointing time-like unit normal to $\wt{\Si}_\ttt$.
\item For each $\om \in \mathbb{S}^2$, the spacetime $\MM_{t^\ast}$ is foliated by a family of null hyperplanes $(\HH_{{}^\om u})_{{}^\om u \in \RRR}$ given as level sets of an optical function ${}^\om u$ satisfying
\begin{align*}
\sup\limits_{\om \in \mathbb{S}^2} \Vert \R \cdot \Ltt \Vert_{L^\infty_{{}^\om u} L^2(\HH_{{}^\om u} \cap \MM_{t^\ast})} \les \CMD,
\end{align*}
where $\Ltt$ is the unique $\HH_{{}^\om u}$-tangent null vectorfield with $\g(\Ltt,\Ttt)=-1$.
\item Define the \emph{angle} $\nutt$ between $T$ and $\Ttt$ by
\begin{align} \label{EQangleDefinition}
\nutt := -\g(T, \Ttt).
\end{align}
The proof of the next lemma is provided in Appendix \ref{secComparisonAppendix}.
\begin{lemma}[Comparison of maximal foliations on $\MM_{t^\ast}$]\label{lem:comparisonfoliationMM} For $\varep>0$ and $\CMD>0$ sufficiently small, it holds that with respect to the foliation $(\Si)_{1\leq t\leq t^\ast}$,
\begin{align} \label{EQlemmaComparisonTwoEstimates}
\norm{\nutt-1}_{L^\infty(\MM_{t^\ast})} \les \CMD, \,\, \Vert \wt{k} \Vert_{L^\infty_t L^4(\Si_t)} \les \CMD,
\end{align}
where $\wt{k}$ denotes the second fundamental form of $\wt{\Si}_\ttt$.
\end{lemma}
\end{enumerate}

\subsection{Energy estimates for curvature tensor on $\Si_{t^\ast}$} \label{sec:CurvatureEstimatesForBackground} \label{SectionImprovement1}

In this section, we prove that
\begin{align} \label{eq:trilinearBootstrapAssumption}
\Vert \R \Vert_{L^\infty_{\tilde t} L^2(\wt{\Si}_{\tilde t} \cap \MM_{t^\ast})} \les \varep.
\end{align}
Using that $\Si_{t^\ast} = \wt{\Si}_{t^\ast} \cap \MM_{t^\ast}$, \eqref{eq:trilinearBootstrapAssumption} implies in particular that
\begin{align} \label{eq:improvedL2curvature}
\Vert \R \Vert_{L^2(\Si_{t^\ast})} \les \varep.
\end{align}

We turn to the proof of \eqref{eq:trilinearBootstrapAssumption}. By a standard application of \eqref{lemma:IntegralidentityWeyl1} with $W= \R$ and multiplier field $\Ttt$, we have
\begin{align} \begin{aligned} 
&\Vert \R \Vert^2_{L^\infty_{\tilde t} L^2(\wt{\Si}_{\tilde t} \cap \MM_{t^\ast})} + \sup\limits_{\om \in \SSS^2} \Vert \R \cdot \Ltt \Vert^2_{L^\infty_{{}^\om u} L^2(\HH_{{}^\om u} \cap \MM_{t^\ast})} \\
\lesssim&  \int\limits_{\Si_1} Q(\R)_{\tilde T \tilde T \tilde T T} + \int\limits_{\HH} Q(\R)_{\Ttt \Ttt \Ttt L} + \underbrace{\left\vert  \, \int\limits_{\MM_{t^\ast}} Q(\Rbf)_{\a \be \Ttt \Ttt} \wt{\pi}^{\a \be} \right\vert}_{=:\mathcal{E}},
\end{aligned} \label{eq:CurvatureEst1} \end{align}
where the integral over $\HH$ is defined in Definition \ref{DEFintegrationH}.\\

We bound the error term $\mathcal{E}$ on the right-hand side of \eqref{eq:CurvatureEst1} as follows. By \eqref{eq:deformationTensorRelations} the components of $\wt{\pi}:= \Lieh_\Ttt \g$ are
$$\wt{\pi}_{\Ttt \Ttt} =0, \,\, \wt{\pi}_{\Ttt j} = \tilde{n}^{-1} \nab_j \tilde n, \,\, \wt{\pi}_{ab} = -2 \tilde{k}_{ab}.$$
Hence, by \eqref{eq:BAspacetimeEstimates01} and \eqref{EqBounds11},
\begin{align} \begin{aligned}
\EE =& \left\vert  \, \int\limits_{\MM_{t^\ast}} Q(\Rbf)_{\a \be \Ttt \Ttt} \wt{\pi}^{\a \be}  \right\vert \\
\les& \left\vert  \, \int\limits_{\MM_{t^\ast}} Q(\Rbf)_{ab \Ttt \Ttt} \wt{k}^{ab} \right\vert + \left\vert  \, \int\limits_{\MM_{t^\ast}} \wt{n}^{-1} Q(\Rbf)_{\Ttt j \Ttt \Ttt} \wt{\nab}^j \wt{n} \right\vert \\
\les& \left\vert  \, \int\limits_{\MM_{t^\ast}} Q(\Rbf)_{ab \Ttt \Ttt} \wt{k}^{ab} \right\vert +\Big(1+ \Vert \wt{n} -1 \Vert_{L^\infty(\MM_{t^\ast})} \Big) \Vert \wt\nab \wt{n} \Vert_{L^\infty(\MM_{t^\ast})} \Vert \R \Vert^2_{L^\infty_t L^2(\Si_t)} \\
\les&  \left\vert  \, \int\limits_{\MM_{t^\ast}} Q(\Rbf)_{ab \Ttt \Ttt} \wt{k}^{ab} \right\vert + \CMD \Vert \R \Vert^2_{L^\infty_t L^2(\Si_t)}.
\end{aligned} \label{eq:Qestimate1BAimprovbackground} \end{align}

The first term on the right-hand side of \eqref{eq:Qestimate1BAimprovbackground} is estimated by a localisation of the trilinear estimate \eqref{EQtrilinearStatement1} of Theorem \ref{THMsmalldataL2details}. Indeed, a direct inspection of its proof on page 112 in \cite{KRS} yields that the following estimate holds on $\MM_{t^\ast}$,
\begin{align*}
\left\vert  \, \int\limits_{\MM_{t^\ast}} Q(\Rbf)_{ab \Ttt \Ttt} \ktt^{ab} \right\vert \lesssim& \CMD \Vert \R \Vert_{L^\infty_{\tilde t} L^2(\wt{\Si}_{\tilde t} \cap \MM_{t^\ast})}^2 + \CMD \Vert \R \Vert_{L^2(\MM_{t^\ast})} \sup\limits_{\om \in \SSS^2} \Vert \R \cdot \Ltt \Vert_{L^\infty_{{}^\om u} L^2(\HH_{{}^\om u} \cap \MM_{t^\ast})} \\
\les& \CMD \Vert \R \Vert_{L^\infty_{\tilde t} L^2(\wt{\Si}_{\tilde t} \cap \MM_{t^\ast})}^2 + \CMD \Vert \R \Vert_{L^\infty_\ttt L^2(\wt{\Si}_\ttt)} \sup\limits_{\om \in \SSS^2} \Vert \R \cdot \Ltt \Vert_{L^\infty_{{}^\om u} L^2(\HH_{{}^\om u} \cap \MM_{t^\ast})} \\
\les& \CMD \Vert \R \Vert_{L^\infty_{\tilde t} L^2(\wt{\Si}_{\tilde t} \cap \MM_{t^\ast})}^2 + \CMD \lrpar{\, \sup\limits_{\om \in \SSS^2} \Vert \R \cdot \Ltt \Vert_{L^\infty_u L^2(\HH_{{}^\om u} \cap \MM_{t^\ast})}}^2.
\end{align*}

Plugging the above and \eqref{eq:Qestimate1BAimprovbackground} into \eqref{eq:CurvatureEst1}, we get
\begin{align} \begin{aligned} 
&\Vert \R \Vert^2_{L^\infty_{\tilde t} L^2(\wt{\Si}_{\tilde t} \cap \MM_{t^\ast})} + \sup\limits_{\om \in \SSS^2} \Vert \R \cdot \Ltt \Vert^2_{L^\infty_{{}^\om u} L^2(\HH_{{}^\om u} \cap \MM_{t^\ast})} \\
\lesssim& \int\limits_{\Si_1} Q(\R)_{\tilde T \tilde T \tilde T T} +\int\limits_{\HH} Q(\R)_{\Ttt \Ttt \Ttt L} \\
&+\CMD \Vert \R \Vert_{L^\infty_{\tilde t} L^2(\wt{\Si}_{\tilde t} \cap \MM_{t^\ast})}^2 + \CMD \lrpar{\, \sup\limits_{\om \in \SSS^2} \Vert \R \cdot \Ltt \Vert_{L^\infty_{{}^\om u} L^2(\HH_{{}^\om u} \cap \MM_{t^\ast})}}^2 \\
\les& \underbrace{\int\limits_{\Si_1} Q(\R)_{\tilde T \tilde T \tilde T T}}_{:= \II_1} + \underbrace{\int\limits_{\HH} Q(\R)_{\Ttt \Ttt \Ttt L}}_{:=\II_2},
\end{aligned} \label{EQcurvatureEststep2}\end{align}
where we used the smallness of $\CMD>0$ to absorb the term on the left-hand side. It remains to bound $\II_1$ and $\II_2$ on the right-hand side of \eqref{EQcurvatureEststep2}. \\

{\bf Estimation of $\mathcal{I}_1$.} Let $(e_i)_{i=1,2,3}$ be an orthonormal frame of $\Si_1$. Decompose $\Ttt$ with respect to this frame into
\begin{align} \label{eqDECOMPcomp1}
\Ttt = \nutt T + C^1e_1 +C^2e_2 + C^3e_3,
\end{align}
and denote $C^0:= \nutt-1$. By \eqref{EQlemmaComparisonTwoEstimates} and the fact that $(T,e_1,e_2,e_3)$ is an orthonormal frame, we deduce that for $i=1,2,3$,
  \begin{align}\label{est:compariC}
    \norm{C^i}_{L^\infty(\Si_1)} \les \sqrt{\CMD}.
  \end{align}

Using \eqref{eqDECOMPcomp1}, we get that
\begin{align*}
\int_{\Si_1} Q_{\Ttt \Ttt \Ttt T} = &  \int_{\Si_1} Q_{TTTT} + \int_{\Si_1} C^\mu Q_{\mu TTT}  + \int_{\Si_1} C^\mu C^\nu Q_{\mu \nu TT} + \int_{\Si_1} C^\mu C^\nu C^\la Q_{\mu \nu \la T}.
  \end{align*}
From Lemmas 7.3.1 and 7.3.2 in~\cite{ChrKl93}, we get that for $\mu, \nu, \la = 0,1,2,3$,
\begin{align} \label{EQrelationQothervectors1}
|Q_{\mu \nu \la T}| \les Q_{TTTT} \les \vert \Rbf \vert_{\mathbf{h}^t}^2.
\end{align}
Hence by~\eqref{est:compariC}, we deduce that for $\CMD>0$ and $\varep>0$ sufficiently small,
\begin{align} \begin{aligned} 
\int_{\Si_1} Q_{\Ttt \Ttt \Ttt T} \les &  \lrpar{1+\sqrt{\CMD}+\sqrt{\CMD}^2+\sqrt{\CMD}^3}\int_{\Si_1}Q_{TTTT}\\
 \les& \norm{\R}^2_{L^2(\Si_0)}\\
  \les& \varep^2,
\end{aligned}\label{EQI1estimateM0} \end{align}
where we used the smallness of the initial data \eqref{eq:SmallData}. \\

{\bf Estimation of $\mathcal{I}_2$.} Let $(N,e_1,e_2)$ be a local frame on $\HH$ such that $(e_1,e_2)$ is an orthonormal frame tangent to $\pr \Si_t$ and $N$ is tangent to $\Si_t$ and normal to $\pr\Si_t$. Decompose $\Ttt$ into 
\begin{align} \label{eqdecompcomp2}
\Ttt = \nutt T + C_1e_1 +C_2e_2 + C_3 N,
\end{align}
and denote $C^0 := \nutt-1$. For $i=1,2,3$, we have by \eqref{EQlemmaComparisonTwoEstimates} and since $(N,e_1,e_2)$ is orthonormal,
\begin{align}\label{est:compariC2}
    \norm{C^i}_{L^\infty(\HH)} \les \sqrt{\CMD}.
  \end{align}

Using \eqref{eqdecompcomp2}, we get
\begin{align*}
\int_\HH Q_{\Ttt \Ttt \Ttt L} = &  \int_\HH Q_{TTTL} + \int_{\HH} C^\mu Q_{\mu TTL} \\ & + \int_{\HH}C^\mu C^\nu Q_{\mu \nu TL} + \int_{\HH} C^\mu C^\nu C^\la Q_{\mu \nu \la L}.
\end{align*}
By Lemmas 7.3.1 and 7.3.2 of \cite{ChrKl93}, we have 
\begin{align*}
|Q_{\mu \nu \la L}| \les Q_{TTTL}.
\end{align*}

By the above, Lemma \ref{LEMQLTTTrelations} and \eqref{eq:SmallData}, we get for $\varep>0$ and $\CMD>0$ sufficiently small,
\begin{align*} \begin{aligned}
\int_\HH Q_{TTTL} \les & \norm{\nutt}^3_{L^\infty(\HH)} \norm{\alpha}^2_{L^2(\HH)} + \norm{\nutt}_{L^\infty(\HH)} \norm{\beta}^2_{L^2(\HH)} \\ & + \norm{\nutt^{-1}}_{L^\infty(\HH)}(\norm{\rho}^2_{L^2(\HH)}+\norm{\sigma}_{L^2(\HH)}^2)
+\norm{\nutt^{-3}}_{L^\infty(\HH)} \norm{\betab}^2_{L^2(\HH)} \\
\les & \norm{\alpha}^2_{L^2(\HH)} + \norm{\beta}^2_{L^2(\HH)} + \norm{\rho}^2_{L^2(\HH)} + \norm{\sigma}^2_{L^2(\HH)} + \norm{\betab}^2_{L^2(\HH)} \\
\les& \varep^2,
\end{aligned} \end{align*}
and therefore with \eqref{est:compariC2},
\begin{align} \label{EQII2estimateM0}
\int_\HH Q_{\Ttt \Ttt \Ttt L}\les& \varep^2.
\end{align}

Plugging \eqref{EQI1estimateM0} and \eqref{EQII2estimateM0} into \eqref{EQcurvatureEststep2}, we get that for $\CMD>0$ and $\varep>0$ sufficiently small,
\begin{align*}
\Vert \R \Vert^2_{L^\infty_{\tilde t} L^2(\wt{\Si}_{\tilde t} \cap \MM)} + \sup\limits_{\om \in \SSS^2} \Vert \R \cdot \Ltt \Vert^2_{L^\infty_{{}^\om u} L^2(\HH_{{}^\om u} \cap \MM)} \les \varep^2.
\end{align*}
This proves in particular \eqref{eq:trilinearBootstrapAssumption} and consequently \eqref{eq:improvedL2curvature}.

\subsection{Elliptic estimates for $k$ on $\Si_{t^\ast}$} \label{SectionKEllipticEstimateWithBoundaryTerm} In this section we prove the next global elliptic estimate for $k$ on $\Si_{t^\ast}$ to improve the bootstrap assumption \eqref{eq:BAspacetimeEstimates01} for $k$.

\begin{proposition}[Global elliptic estimate for $k$] \label{prop:Kestim1} It holds that
\begin{align*}
\Vert \nab k \Vert^2_{L^2(\Si_{t^\ast})} + \Vert k \Vert_{L^4(\Si_{t^\ast})}^4 + \Vert k \Vert^2_{L^2(\pr \Si_{t^\ast})}+ \Vert \Nd \nu \Vert^2_{L^2(\pr \Si_{t^\ast})} \lesssim (\sqrt{D}\varep)^2.
\end{align*}
\end{proposition}

\begin{proof} From Corollary \ref{integralIDKEST}, we have the following well-known, classical global elliptic estimate for $k$ (a proof is provided in Appendix \ref{SECellipticEstimatesKnotation}),
\begin{align}\label{eq:energyID1222}
\int\limits_{\Si_{t^\ast}} \vert \nab k \vert^2 + \frac{1}{4} \vert k \vert^4- \int\limits_{\pr \Si_{t^\ast}}  \nab_a k_{bN} k^{ba} \lesssim \int\limits_{\Si_{t^\ast}} \vert \mathbf{R} \vert_{\mathbf{h}^t}^2,
\end{align}
where $N$ denotes the outward-pointing unit normal to $\pr \Si_{t^\ast} \subset \Si_{t^\ast}$.\\

By \eqref{eq:divk}, the boundary term on the left-hand side of \eqref{eq:energyID1222} can be rewritten as
\begin{align} \begin{aligned}
- \int\limits_{\pr \Si_{t^\ast}}  \nab_a k_{bN} k^{ba} =& - \int\limits_{\pr \Si_{t^\ast}} \Big( \nab_N k_{bN} k^{bN} + \nab_C k_{bN} k^{bC} \Big) \\
=&-  \int\limits_{\pr \Si_{t^\ast}} \Big( - \nab^C k_{bC} k^{bN}+ \nab_C k_{NN} k^{NC} + \nab_C k_{AN} k^{AC} \Big)\\
=& \int\limits_{\pr \Si_{t^\ast}} \nab^C k_{AC} k^{AN} + \nab_C k_{NC} k^{NN} - \nab_C k_{NN} k^{NC} - \nab_C k_{AN} k^{AC}.
\end{aligned} \label{eq:firstdecompConstraints} \end{align}

Using the boundary decomposition of $k$ on $\pr \Si_{t^\ast}$, see Section \ref{SECspacelikefoliations},
\begin{align*}
\de := k_{NN}, \, \ep_A := k_{NA}, \, \eta_{AB} :=k_{AB},
\end{align*}
together with \eqref{eq:DivEtaNabDeRelation} and integration by parts, the right-hand side of \eqref{eq:firstdecompConstraints} equals 
\begin{align} \begin{aligned}
- \int\limits_{\pr \Si_{t^\ast}}  \nab_a k_{bN} k^{ba} =& \int\limits_{\pr \Si_{t^\ast}} 2 (\Divd \eta)_{A} \ep^A  - 2 \ep \cdot \Nd \de + 3 \Th_{AB} \ep^A \ep^B + \tr \Th \vert \ep \vert^2 \\
&+  \int\limits_{\pr \Si_{t^\ast}} \de^2 \tr \Th -2 \de \eta_{AB} \Th^{AB}  + \eta_{AB} \Th_{BC} \eta^{CA} \\
=& \int\limits_{\pr \Si_{t^\ast}} - 4 \ep \cdot \Nd \de - 2 \in_{AB} H^B_{\,\,\, N} \ep^A + 5 \Th_{AB} \ep^A \ep^B +3 \tr \Th \vert \ep \vert^2 \\
&+  \int\limits_{\pr \Si_{t^\ast}} \de^2 \tr \Th -2 \de \eta_{AB} \Th^{AB}  + \eta_{AB} \Th_{BC} \eta^{CA}.
\end{aligned} \label{eq53434} 
\end{align}

The right-hand side of \eqref{eq53434} is then rewritten as
\begin{align} \begin{aligned}
- \int\limits_{\pr \Si_{t^\ast}}  \nab_a k_{bN} k^{ba} =& - \int\limits_{\pr \Si_{t^\ast}} 4 \ep \cdot \Nd \de + 2 \in^{AB} H_{B N} \ep_A \\
&+ \int\limits_{\pr \Si_{t^\ast}}  \left( \tr \Theta - \frac{2}{{t^\ast}} \right)\left( \frac{11}{2}\vert \ep \vert^2 + 2  \vert \de \vert^2 + \frac{1}{2}  \vert \eta \vert^2 \right)\\
&+ \int\limits_{\pr \Si_{t^\ast}} 5 \widehat{\Theta}_{AB} \ep^A \ep^B - 2 \de \widehat{\Theta}_{AB} \eta^{AB} - \eta_{AC} \widehat{\Theta}^{CB} \eta_{B}^{\,\,\,\, A} \\
&+ \int\limits_{\pr \Si_{t^\ast}} \frac{11}{{t^\ast}} \vert \ep \vert^2 + \frac{4}{{t^\ast}} \vert \de \vert^2 + \frac{1}{{t^\ast}} \vert \eta \vert^2.
\end{aligned} \label{eq:boundaryestimate111}
\end{align}

On the right-hand side of \eqref{eq:boundaryestimate111}, we can use Lemmas \ref{lemma:TNrelations} and \ref{lemma:slopeEquation}, that is, the relations
\begin{align} \begin{aligned}
\ep_A=& - \nut^{-1} \Nd_A \nut + \zet_A, \\
\de=& \half \nu \tr \chi + \half \nu^{-1} \tr \chib, \\
 \Nd \de =& \Nd \left( \half \nu \tr \chi + \half \nu^{-1} \tr \chib \right) \\
 =&  \underbrace{\left( \frac{1}{t^\ast} \left( 1+ \frac{1}{\nu^2}\right) + \half \left(\tr \chi - \frac{2}{t^\ast} \right) - \half \frac{1}{\nu^2} \left( \tr \chib +\frac{2}{t^\ast} \right) \right)}_{=: F(\nu, \tr \chi, \tr \chib)} \Nd \nu \\
 &+ \half \nu \Nd \tr \chi + \half \nu^{-1} \Nd \tr \chib,
\end{aligned} \label{eq:deltanuRelation} \end{align}
to rewrite
\begin{align} \begin{aligned} 
- \int\limits_{\pr \Si_{t^\ast}} 4\ep \cdot \Nd \de =& - \int\limits_{\pr \Si_{t^\ast}}4\zet \cdot \Nd \de +\int\limits_{\pr \Si_{t^\ast}} 4\nut^{-1} \Nd_A \nut \Nd^A \de \\
=& - \int\limits_{\pr \Si_{t^\ast}}4\zet \cdot \Nd \de + \int\limits_{\pr \Si_{t^\ast}}4\nu^{-1} F(\nu,\tr \chi,\tr \chib) \vert \Nd \nu \vert^2 \\
&+  \int\limits_{\pr \Si_{t^\ast}}2 \left( \Nd \nu \cdot \Nd \tr \chi + \nu^{-2} \Nd \nu \cdot \Nd \tr \chib \right).
\end{aligned} \label{eq:bdryterm1analysis1} \end{align}

By \eqref{eq:BAspacetimeEstimates01} and \eqref{lem:FirstConsequencesSpacetime1}, it holds for $\varep>0$ sufficiently small that on $\pr \Si_{t^\ast}$,
\begin{align*}
\nu^{-1} F(\nu, \tr \chi,\tr \chib) \geq \frac{1}{8}.
\end{align*}
Hence, for $\varep>0$ sufficiently small, \eqref{eq:bdryterm1analysis1} yields
\begin{align} \label{eq:BdryTermEstimate1}
- \int\limits_{\pr \Si_{t^\ast}} 4\ep \cdot \Nd \de \geq + \int\limits_{\pr \Si_{t^\ast}} \half \vert \Nd \nu \vert^2 + \int\limits_{\pr \Si_{t^\ast}} 2 \left( \Nd \nu \cdot \Nd \tr \chi + \nu^{-2} \Nd \nu \cdot \Nd \tr \chib \right)-4 \zeta \cdot \Nd \de.
\end{align}

Plugging \eqref{eq:boundaryestimate111} and \eqref{eq:BdryTermEstimate1} into \eqref{eq:energyID1222}, we get for $\varep>0$ sufficiently small that
\begin{align} \begin{aligned}
&\int\limits_{\Si_{t^\ast}} \vert \nab k \vert^2 + \frac{1}{4} \vert k \vert^4 + \int\limits_{\pr \Si_{t^\ast}} \vert k \vert^2 + \int\limits_{\pr \Si_{t^\ast}} \vert \Nd \nu \vert^2 \\
\lesssim& \underbrace{\int\limits_{\Si_{t^\ast}} \vert \mathbf{R} \vert_{\mathbf{h}^t}^2}_{:= I_1}+ \underbrace{\left\vert \, \int\limits_{\pr \Si_{t^\ast}} \in^{AB} H_{B N} \ep_A \right\vert}_{:=I_2}  + \underbrace{\left\vert \, \int\limits_{\pr \Si_{t^\ast}}  \left( \tr \Theta - \frac{2}{{t^\ast}} \right) \vert k \vert^2 \right\vert}_{:=I_3} \\
&+ \underbrace{\left\vert \, \int\limits_{\pr \Si_{t^\ast}} \widehat{\Theta}_{AB} \ep^A \ep^B \right\vert}_{:=I_4} + \underbrace{\left\vert \, \int\limits_{\pr \Si_{t^\ast}} \de \widehat{\Theta}_{AB} \eta^{AB} \right\vert}_{:=I_5} + \underbrace{\left\vert \, \int\limits_{\pr \Si_{t^\ast}} \eta_{AC} \widehat{\Theta}^{CB} \eta_{B}^{\,\,\,\, A} \right\vert}_{:=I_6} \\
&+ \underbrace{\left\vert \, \int\limits_{\pr \Si_{t^\ast}} \zeta \cdot \Nd \de \right\vert}_{:=I_7} + \underbrace{\left\vert \, \int\limits_{\pr \Si_{t^\ast}} \Nd \nu \cdot \Nd \tr \chi \right\vert}_{:=I_8} + \underbrace{\left\vert \, \int\limits_{\pr \Si_{t^\ast}} \nut^{-2} \Nd \nu \Nd \tr \chib \right\vert}_{:=I_9}.
\end{aligned} \label{eq:EllEstKFullGeneral2} \end{align}
In the following, we bound each term $I_1$-$I_9$ from \eqref{eq:EllEstKFullGeneral2}. \\


{\bf Estimation of $I_1$.} From \eqref{eq:improvedL2curvature}, we directly have
\begin{subequations}
\begin{align} \label{eq:i1estimateIMPROVED}
I_1 := \int\limits_{\Si_{t^\ast}} \vert \mathbf{R} \vert_{\mathbf{h}^t}^2 \lesssim \varep^2.
\end{align}

{\bf Estimation of $I_2$.} The integral $I_2$ is estimated by
\begin{align} \begin{aligned}
I_2 := \left\vert \, \int\limits_{\pr \Si_{t^\ast}} \in^{AB} H_{B N} \ep_A \right\vert \lesssim& \Big(\Vert \ep \Vert_{L^2(\Si_{t^\ast})} + \Vert \nab \ep \Vert_{L^2(\Si_{t^\ast})} \Big) \Big( \Vert H \Vert_{L^2(\Si_{t^\ast})} + \Vert E \Vert_{L^2(\Si_{t^\ast})} \Big)  \\
\lesssim& D\varep \cdot \varep \\
\les& (\sqrt{D}\varep)^2.
\end{aligned} \label{eq:SecondTermBAMainK} \end{align}
Indeed, \eqref{eq:SecondTermBAMainK} follows from a standard bilinear trace theorem (see for example Lemma 4.14 in \cite{J3}). For completeness, we outline the proof of \eqref{eq:SecondTermBAMainK} here. In the following we use \eqref{EqBounds2}, spherical coordinates $(r,\th^1, \th^2)$ on $\Si_{t^\ast}$ as defined in Section \ref{SECfoliationSpheres}, and \eqref{EQCoeffSmallSpherical}. Let $\phi: \Si_{t^\ast} \to [0,1]$ be a smooth radial cut-off function such that $\phi(p)=1$ for $r(p) \geq 3t^\ast/4$ and $\phi(p) =0$ for $r(p) \leq t^\ast/2$. Then we have by the fundamental theorem of calculus, the Bianchi identity \eqref{eq:BianchiForEandH} and integration by parts on $\SS_r$,
\begin{align*}
\int\limits_{\pr \Si_{t^\ast}} \in^{AB} H_{BN} \ep_A =& \int\limits_{t^\ast/2}^{t^\ast} \pr_r \lrpar{\phi \int\limits_{\SS_r} \in^{AB} H_{BN} \ep_A} dr \\
\les& \int\limits_{t^\ast/2}^{t^\ast} \phi \lrpar{\, \int\limits_{\SS_r} \in^{AB} (\nab_N H_{BN}) \ep_A}dr + \varep (D\varep) \\
\les& \int\limits_{t^\ast/2}^{t^\ast} \phi \lrpar{\, \int\limits_{\SS_r} \in^{AB} (-\Nd_C H_{BC}) \ep_{A}} dr + \varep(D \varep) \\
\les& \int\limits_{t^\ast/2}^{t^\ast} \phi \lrpar{\, \int\limits_{\SS_r} \in^{AB} H_{BC} \Nd_C \ep_{A}} dr + \varep(D \varep) \\
\les& \Vert H \Vert_{L^2(\Si_{t^\ast})} \Vert \Nd \ep \Vert_{L^2(\Si_{t^\ast})} + \varep(D\varep) \\
\les& \varep(D \varep)\\
\les& (\sqrt{D}\varep)^2,
\end{align*} 
where we estimated the error terms by $\varep D\varep$ and used that $0<\CMD<1/2$; details are left to the reader. \\

{\bf Estimation of $I_3$.} Using \eqref{eq:BAspacetimeEstimates01} and \eqref{lem:FirstConsequencesSpacetime1}, we have
\begin{align} \begin{aligned}
I_3 := \left\vert \, \int\limits_{\pr \Si_{t^\ast}}  \left( \tr \Theta - \frac{2}{{t^\ast}} \right) \vert k \vert^2 \right\vert \leq& \left\Vert \tr \Th - \frac{2}{t^\ast} \right\Vert_{L^\infty(\pr \Si_{t^\ast})}\lrpar{ \,\, \int\limits_{\pr \Si_{t^\ast}} \vert k \vert^2} \\
\les&  (D \varep)^3.
\end{aligned} \label{eq:i3estimate} \end{align}

{\bf Estimation of $I_4,I_5$ and $I_6$.} By Lemma \ref{Lemma:TraceEstimateH1toL4bdry}, \eqref{lem:FirstConsequencesSpacetime1} and \eqref{lem:FirstConsequencesSpacetime2}, we have
\begin{align} \begin{aligned} 
I_4:= \left\vert \, \int\limits_{\pr \Si_{t^\ast}} \widehat{\Theta}_{AB} \ep^A \ep^B \right\vert \lesssim& \Vert \widehat{\Th} \Vert_{L^4(\pr \Si_{t^\ast})} \Vert \ep \Vert_{L^2(\pr \Si_{t^\ast})} \Vert \ep \Vert_{L^4(\pr \Si_{t^\ast})} \\
\lesssim& \Vert \widehat{\Th} \Vert_{L^4(\pr \Si_{t^\ast})} \Vert \ep \Vert_{L^2(\pr \Si_{t^\ast})}\Big( \Vert \ep \Vert_{L^2(\Si_{t^\ast})} + \Vert \nab \ep \Vert_{L^2(\Si_{t^\ast})} \Big)\\
\lesssim&  (D\varep)^3 .\end{aligned} \label{eq:i4TOi7estimate} \end{align}
The terms $I_5$ and $I_6$ are bounded similarly as
\begin{align}
I_5 +I_6 \lesssim (D\varep)^3 ;
\end{align}
details are left to the reader. \\

{\bf Estimation of $I_7,I_8$ and $I_9$.} By \eqref{eq:SmallData} and \eqref{eq:deltanuRelation}, we have
\begin{align} \begin{aligned}
I_7:=&\left\vert \, \int\limits_{\pr \Si_{t^\ast}} \zeta \cdot \Nd \de \right\vert \\
\leq& \left\vert \, \int\limits_{\pr \Si_{t^\ast}} F(\nu, \tr \chi, \tr \chib) \zeta \cdot \Nd \nu + \half \zeta\cdot \left( \nu \Nd \tr \chi + \nu^{-1} \Nd \tr \chib  \right) \right\vert \\
\lesssim&\left(1+ \Vert \nu-1 \Vert_{L^\infty(\pr \Si_{t^\ast})} + \left\Vert \tr \chi - \frac{2}{t^\ast} \right\Vert_{L^\infty(\pr \Si_{t^\ast})} +\left\Vert \tr \chib + \frac{2}{t^\ast} \right\Vert_{L^\infty(\pr \Si_{t^\ast})} \right) \\
& \cdot \Vert \zeta \Vert_{L^2(\pr \Si_{t^\ast})} \Big( \Vert \Nd \nu \Vert_{L^2(\pr \Si_{t^\ast})} + \Vert \Nd \tr \chi \Vert_{L^2(\pr \Si_{t^\ast})} + \Vert \Nd \tr \chib \Vert_{L^2(\pr \Si_{t^\ast})} \Big) \\
\lesssim& \varep(D\varep) + \varep^2.
\end{aligned} \end{align}

The terms $I_8$ and $I_9$ are bounded similarly by \eqref{eq:SmallData},
\begin{align} \label{eq:EightAndNinthTermMainBAk}
I_8+I_9 \lesssim (\sqrt{D}\varep)^2;
\end{align}
details are left to the reader. \end{subequations} \\

Plugging \eqref{eq:i1estimateIMPROVED}-\eqref{eq:EightAndNinthTermMainBAk} into \eqref{eq:EllEstKFullGeneral2}, we get that for $\CMD>0$ and $\varep>0$ sufficiently small, 
\begin{align*}
\int\limits_{\Si_{t^\ast}} \vert \nab k \vert^2 + \frac{1}{4} \int\limits_{\Si_{t^\ast}}\vert k \vert^4 + \int\limits_{\pr \Si_{t^\ast}} \vert k \vert^2 + \int\limits_{\pr \Si_{t^\ast}} \vert \Nd \nu \vert^2 \les(\sqrt{D}\varep)^2.
\end{align*}
This finishes the proof of Proposition \ref{prop:Kestim1}. \end{proof}

It remains to control the $L^2$-norm of $k$ on $\Si_{t^\ast}$.
\begin{lemma} \label{lemma:kBAmainImprov2} It holds that
$$\Vert k \Vert_{L^2(\Si_{t^\ast})} \lesssim \sqrt{D}\varep.$$
\end{lemma}

\begin{proof} The estimate follows in a standard way from the above bounds for $\Vert \nab k \Vert_{L^2(\Si_{t^\ast})}$ and $\Vert k \Vert_{L^2(\pr \Si_{t^\ast})}$, and the $L^\infty$-control of $g_{ij}$ established in Section \ref{SECweakreguM0}. For completeness, we write out full details. Using \eqref{EqBounds2}, define spherical coordinates $(r,\th^1, \th^2)$ on $\Si_{t^\ast}$ as in Section \ref{SECfoliationSpheres}. Let ${\ga}$ and $d \mu_{\ga}$ denote the standard round metric on $\SS_r$ (of radius $r>0$) and its volume element, respectively. Using the fundamental theorem of calculus, it holds that
\begin{align*}
\frac{1}{r^2} \int\limits_{\SS_r} \vert k \vert_g^2 d\mu_\ga =& \int\limits_{t^\ast}^r \pr_{r'} \left( \frac{1}{r'^2} \int\limits_{\SS_{r'}} \vert k \vert^2_g d\mu_\ga \right) dr' + \frac{1}{{t^\ast}^2}\int\limits_{\SS_{t^\ast}} \vert k \vert_g^2 d\mu_\ga \\
=& \int\limits_{t^\ast}^r \frac{1}{r'^2} \lrpar{\, \int\limits_{\SS_{r'}} \pr_{r'} \left( \vert k \vert^2_g \right) d\mu_\ga } dr' + \frac{1}{{t^\ast}^2}\int\limits_{\SS_{t^\ast}} \vert k \vert_g^2 d\mu_\ga \\
=&  \int\limits_{t^\ast}^r \frac{1}{r'^2} \lrpar{\, \int\limits_{\SS_{r'}} \nab_{\pr_{r'}} k \cdot k \, d\mu_\ga } dr' + \frac{1}{{t^\ast}^2}\int\limits_{\SS_{t^\ast}} \vert k \vert_g^2 d\mu_\ga \\
\leq& \frac{1}{r^2} \left\Vert  \vert \pr_r \vert_g^2 \right\Vert_{L^\infty(\ol{B(0,t^\ast)})} \left(  \int\limits_{t^\ast}^r \int\limits_{\SS_{r'}} \vert \nab k \vert^2_g d\mu_\ga dr' \right)^{1/2} \left(  \int\limits_{t^\ast}^r \int\limits_{\SS_{r'}} \vert k \vert^2_g d\mu_\ga dr' \right)^{1/2} \\
&+ \frac{1}{{t^\ast}^2}\int\limits_{\SS_{t^\ast}} \vert k \vert_g^2 d\mu_\ga \\
\les&  \frac{1}{r^2} \left\Vert  \vert \pr_r \vert_g^2 \right\Vert_{L^\infty(\ol{B(0,t^\ast)})} (1+\CMD) \Vert \nab k \Vert_{L^2(\Si_{t^\ast})} (1+\CMD) \Vert k \Vert_{L^2(\Si_{t^\ast})} \\
&+ \frac{1}{{t^\ast}^2}\int\limits_{\SS_{t^\ast}} \vert k \vert_g^2 d\mu_\ga.
\end{align*}
Multiplying the above with $r^2$, using Proposition \ref{prop:Kestim1} and that by definition of the spherical coordinates on $\overline{B(0,t^\ast)}$ and the weak regularity of $\Si_{t^\ast}$ with constant $\CMD$, for $r>0$,
\begin{align*}
 \vert \pr_r \vert_g^2 =  g_{rr} =  \frac{x^i}{r} \frac{x^j}{r} g_{ij} \leq (1+ \CMD)\frac{x^i}{r} \frac{x^j}{r} e_{ij} =1+\CMD,
\end{align*}
it follows that for $0< r \leq t^\ast$,
\begin{align*}
\int\limits_{\SS_r} \vert k \vert_g^2 d\mu_\ga \lesssim \Vert \nab k \Vert_{L^2(\Si_{t^\ast})} \Vert k \Vert_{L^2(\Si_{t^\ast})} + \int\limits_{\pr \Si_{t^\ast}} \vert k \vert_g^2 \lesssim (\sqrt{D} \varep)^2.
\end{align*}
In particular, using that $\Si_{t^\ast}$ is a weakly regular ball with constant $0<\CMD<1/2$, we get that
\begin{align*}
\Vert k \Vert_{L^2(\Si_{t^\ast})}^2 \lesssim \int\limits_{\Si_{t^\ast}} \vert k \vert^2_g d\mu_e \lesssim \int\limits_0^{t^\ast} \lrpar{\, \int\limits_{\SS_r} \vert k \vert_g^2 d\mu_\ga} dr \les& (\sqrt{D} \varep)^2,
\end{align*}
where $e$ denotes the Euclidean metric on $\Si_{t^\ast}= \ol{B(0,1)}$. This finishes the proof of Lemma \ref{lemma:kBAmainImprov2}. \end{proof}

Lemmas \ref{prop:Kestim1} and \ref{lemma:kBAmainImprov2} together with the trace estimate of Lemma \ref{Lemma:TraceEstimateH1toL4bdry} yield the following.
\begin{corollary} \label{corollary:KL4onbdry} It holds that
\begin{align*}
\Vert k \Vert_{H^{1/2}(\pr \Si_{t^\ast})}+\Vert k \Vert_{L^4(\pr \Si_{t^\ast})} \lesssim \sqrt{D}\varep.
\end{align*}
\end{corollary}

\subsection{Improvement of $\nu-1$} We first prove the next lemma.
\begin{lemma} \label{lem:NuImprov1} It holds that
\begin{align*}
\Vert \nu -1 \Vert_{L^2(\pr \Si_{t^\ast})} \les \sqrt{D}\varep.
\end{align*}
\end{lemma}

\begin{proof} Recall from Lemma \ref{lemma:TNrelations} that on $\pr \Si_{t^\ast}$,
$$\de= \half \nu \tr \chi + \half \nu^{-1} \tr \chib.$$ 
This can be rewritten as
\begin{align*}
\de = \half \nu \left( \tr \chi - \frac{2}{t^\ast} \right) + \half \nu^{-1} \left( \tr \chib + \frac{2}{t^\ast} \right) + \frac{1}{t^\ast} \nu (\nu+1)(\nu-1),
\end{align*}
which leads to
\begin{align*}
\nu-1 = \frac{ t^\ast}{\nu(\nu+1)} \left( \de - \half \nu \left( \tr \chi - \frac{2}{t^\ast} \right) -\half \nu^{-1} \left( \tr \chib + \frac{2}{t^\ast} \right)\right).
\end{align*}
Consequently, using Proposition \ref{prop:Kestim1} and \eqref{eq:SmallData}, we can estimate
\begin{align*} \begin{aligned}
\Vert \nu -1 \Vert_{L^2(\pr \Si_{t^\ast})} \lesssim& \Vert \de \Vert_{L^2(\pr \Si_{t^\ast})} + \left\Vert \tr \chi - \frac{2}{t^\ast} \right\Vert_{L^2(\pr \Si_{t^\ast})} + \left\Vert \tr \chib + \frac{2}{t^\ast} \right\Vert_{L^2(\pr \Si_{t^\ast})} \\
\lesssim& \sqrt{D}\varep.
\end{aligned}
\end{align*}
This finishes the proof of Lemma \ref{lem:NuImprov1}. \end{proof}

Moreover, we have the following.
\begin{lemma} \label{lem:numinus1infty} It holds that
$$ \Vert \Nd \nu \Vert_{L^4(\pr \Si_{t^\ast})}+\Vert \nu -1 \Vert_{L^\infty(\pr \Si_{t^\ast})} +\Vert \nu^{-1} \Nd \nu \Vert_{H^{1/2}(\pr \Si_{t^\ast})}\lesssim \sqrt{D} \varep.$$
\end{lemma}

\begin{proof} Indeed, by Lemma \ref{lemma:slopeEquation}, \eqref{eq:SmallData} and Corollary \ref{corollary:KL4onbdry}, we have
\begin{align*}
\Vert \Nd \nu \Vert_{L^4(\pr \Si_{t^\ast})}  \lesssim& \Vert \ep \Vert_{L^4(\pr \Si_{t^\ast})} + \Vert \zeta \Vert_{L^4(\pr \Si_{t^\ast})} \\
\les& \sqrt{D} \varep+ \varep.
\end{align*}
Consequently, by Lemmas \ref{lemma:CalculusOnGat1} and \ref{lem:NuImprov1}, we have
\begin{align*}
\Vert \nu-1 \Vert_{L^\infty(\pr \Si_{t^\ast})} \les& \Vert \Nd (\nu-1) \Vert_{L^4(\pr \Si_{t^\ast})} + \Vert \nu-1 \Vert_{L^2(\pr \Si_{t^\ast})}\\
 \les&  \sqrt{D} \varep.
\end{align*}

Further, by Lemmas \ref{lemma:slopeEquation} and \ref{PROPtraceEstimate}, with \eqref{eq:SmallData} and Proposition \ref{prop:Kestim1} and Lemma \ref{lemma:kBAmainImprov2},
\begin{align*}
\Vert \nu^{-1} \Nd \nu \Vert_{H^{1/2}(\pr \Si_{t^\ast})} \les& \Vert \zeta \Vert_{H^{1/2}(\pr \Si_{t^\ast})} + \Vert \ep \Vert_{H^{1/2}(\pr \Si_{t^\ast})} \\
\les& \Vert \zeta \Vert_{H^{1/2}(\pr \Si_{t^\ast})} + \lrpar{ \Vert \ep \Vert_{L^2(\Si_{t^\ast})} + \Vert \nab \ep \Vert_{L^2(\Si_{t^\ast})} }\\
\les&  \sqrt{D} \varep.
\end{align*}
This finishes the proof of Lemma \ref{lem:numinus1infty}. \end{proof}

We note that at this point we can reapply the estimates of Section \ref{SECestimatesForLapseN} for $n$ to get
\begin{align*}
\Vert n-1 \Vert_{L^\infty(\Si_{t^\ast})} + \Vert \nab n \Vert_{L^2(\Si_{t^\ast})} + \Vert \nab^2 n \Vert_{L^2(\Si_{t^\ast})} \les \sqrt{D}\varep.
\end{align*}

As a consequence of the above, we can improve the bound for $\D_T k$. 
\begin{lemma} \label{LemmaDTcontrolk} It holds that
\begin{align*}
\Vert \D_T k \Vert_{L^2(\Si_{t^\ast})} \les \sqrt{D}\varep.
\end{align*}
\end{lemma}

\begin{proof} Indeed, by the second variation equation \eqref{eq:sndvar}, that is,
 \begin{align*}
    \D_Tk_{ij} = E_{ij} -n^{-1} \nab_i\nab_j n +k_{il}k_{\,\,\,j}^l,
  \end{align*}
  we get that
  \begin{align*}
\Vert \D_T k \Vert_{L^2(\Si_{t^\ast})} \les& \Vert E \Vert_{L^2(\Si_{t^\ast})} + \Vert \nab^2 n \Vert_{L^2(\Si_{t^\ast})} + \left( \Vert \nab k \Vert_{L^2(\Si_{t^\ast})} + \Vert k \Vert_{L^2(\Si_{t^\ast})}\right)^2\\
\les& \sqrt{D}\varep + (\sqrt{D}\varep)^2.
\end{align*}
This finishes the proof of Lemma \ref{LemmaDTcontrolk} and the improvement of \eqref{eq:BAspacetimeEstimates01}.
\end{proof}

\section{Higher regularity estimates} \label{SectionHigherRegularity}  \label{SEChigherEstimatesM1} In this section we prove Proposition \ref{thm:HigherRegularitySpacetimeEstimates}. In Sections \ref{SECm1estimatesOUTLINE}-\ref{SECm1estimatesLAPSE} we prove the higher regularity estimates for $m=1$, and in Section \ref{SECm2estimatesOUTLINE} we outline the estimates for $m\geq2$. As remarked in Section \ref{sec:ProofOfMainTheorem}, the case $m=1$ requires a trilinear estimate which necessitates an inspection of the Yang-Mills formalism and wave parametrix construction of \cite{KRS}, see Proposition \ref{PropTrilinearM1} and its proof in Appendix \ref{SECtrilinearEstimateM1}. On the contrary, the cases $m\geq2$ are proved by a classical Gr\"onwall argument together with straight-forward generalisations of the methods for $m=1$.

\subsection{Proof of the case $m=1$} \label{SECm1estimatesOUTLINE}

In this section, we prove the case $m=1$ of Proposition \ref{thm:HigherRegularitySpacetimeEstimates}. Assume that 
\begin{align*}
 \OO_1^{\Si} + \RR_1^{\Si}+ \OO_1^{\HH} + \RRt_1^\HH < \infty,
\end{align*}
and that for two reals $1<t_0^\ast \leq 2$ and $\varep>0$, it holds for $1\leq t \leq t_0^\ast$ that
\begin{align*}
\text{$\Si_t$ is a weakly regular ball with constant $\CMD>0$ }
\end{align*}
and 
\begin{align} \begin{aligned}
\Vert \RRRic \Vert_{L^\infty_tL^2(\Si_t)} \les&\, \varep, \\
\Vert k \Vert_{L^\infty_tL^2(\Si_t)} + \Vert \nab k \Vert_{L^\infty_tL^2(\Si_t)} + \Vert k \Vert_{L^\infty_tL^2(\pr \Si_t)} \les&\, \varep, \\
 \Vert \nu -1 \Vert_{L^\infty_t L^\infty(S_t)} +\Vert \Nd \nu \Vert_{L^\infty_t L^4(S_t)} +\Vert \Nd \nu \Vert_{L^\infty_t H^{1/2}(S_t)} \les&\,\varep,\\
\Vert n-1 \Vert_{L^\infty_t L^\infty(\Si_t)} + \Vert \nab n \Vert_{L^\infty_t L^2(\Si_t)} + \Vert \nab^2 n \Vert_{L^\infty_tL^2(\Si_t)} \les& \,\varep,
\end{aligned} \label{EQm1regEstimatesAssumption}
\end{align}

In the following we show that for $\CMD>0$ and $\varep>0$ sufficiently small, it holds that for $1 \leq t \leq t_0^\ast$,
\begin{align}
\Vert \nab E \Vert_{L^\infty_t L^2(\Si_t)} +\Vert \nab H \Vert_{L^\infty_t L^2(\Si_t)} \les&\, \OO_1^{\HH} + \RRt_1^\HH +\OO_1^{\Si} + \RRt_1^\Si+\CMD, \label{EQStatementHigherReg11000}  \\
\Vert \nab^2 k \Vert_{L^\infty_t L^2(\Si_t)}+ \Vert \Nd^2 \nu \Vert_{L^\infty_t L^2(S_t)}&+ \Vert \Nd^2 \nu \Vert_{L^\infty_t H^{1/2}(S_t)} \label{EQStatementHigherReg11} \\
\les& \,\OO_1^{\HH} + \RRt_1^\HH +\OO_1^{\Si} + \RRt_1^\Si +\CMD\nonumber  \\
\Vert \D\Rbf \Vert_{L^\infty_t L^2(\Si_t)}  \les&\, \OO_1^{\HH} + \RRt_1^\HH +\OO_1^{\Si} + \RRt_1^\Si+\CMD, \label{EQStatementHigherReg112} \\
\Vert \nab^3 n \Vert_{L^\infty_t L^2(\Si_t)}+ \Vert \nab^2 T(n) \Vert_{L^\infty_t L^2(\Si_t)}&+\Vert \nab T^2(n) \Vert_{L^\infty_t L^2(\Si_t)}\label{EQStatementHigherReg13} \\
\les& \,\OO_1^{\Si} + \RR_1^\Si + \OO_1^{\HH}+ \RRt_1^\HH +\CMD. \nonumber
\end{align}

\textbf{Notation.} Pick $1\leq t^\ast \leq t_0^\ast$. In the following, we prove \eqref{EQStatementHigherReg11000}, \eqref{EQStatementHigherReg11}, \eqref{EQStatementHigherReg112} and \eqref{EQStatementHigherReg13} on $\Si_{t^\ast}$. As $t^\ast$ was chosen arbitrarily, this implies  \eqref{EQStatementHigherReg11000}, \eqref{EQStatementHigherReg11}, \eqref{EQStatementHigherReg112} and \eqref{EQStatementHigherReg13} for $1\leq t \leq t_0^\ast$.

\begin{remark} The smallness of $\CMD>0$ and $\varep>0$ is only used in the proof of the estimates for $m=1$. For the cases $m\geq2$, no further smallness assumption is made.
\end{remark}

We start by setting up the geometric framework. By the assumption that $\Si_{t^\ast}$ is a weakly regular ball with constant $\CMD>0$ and \eqref{EQm1regEstimatesAssumption}, we can pick $\CMD>0$ and $\varep>0$ sufficiently small such that
\begin{align} \begin{aligned} 
\Vert g_{ij}-e_{ij} \Vert_{H^2(\Si_{t^\ast})} + \Vert k_{ij} \Vert_{H^1(\Si_{t^\ast})} \les&\, \CMD.
\end{aligned} \label{EQM1higherRegCoordinates1} \end{align}

By \eqref{EQM1higherRegCoordinates1} and for $\CMD>0$ sufficiently small, we can use Theorem \ref{THMextensionConstraintsCZ1} to extend $(\Si_{t^\ast}, g,k)$ to an asymptotically flat, regular maximal initial data set on $\RRR^3$ which satisfies the assumptions of the bounded $L^2$ curvature theorem, see Theorem \ref{THMsmalldataL2details}. Consequently, applying Theorem \ref{THMsmalldataL2details} \emph{backwards} from $\Si_{t^\ast}$, we get:
\begin{enumerate}
\item The past of $\Si_{t^\ast}$ in $\MM$, denoted by $\MM_{t^\ast}$, is foliated by maximal spacelike hypersurfaces $(\wt{\Si}_\ttt)_{0 \leq \ttt \leq t^\ast}$ given as level sets of a time function $\ttt$ with $\wt{\Si}_{t^\ast} \cap \MM_{t^\ast}= \Si_{t^\ast}$ and satisfying
\begin{align} \begin{aligned}
\Vert \R \Vert_{L^\infty_\ttt L^2(\wt{\Si}_\ttt)} \les&\, \CMD, \\
\Vert \wt{k} \Vert_{L^\infty_\ttt L^2(\wt{\Si}_\ttt)} + \Vert \wt\nab \wt{k} \Vert_{L^\infty_\ttt L^2(\wt{\Si}_\ttt)}+\Vert \wt{\mathbf{A}} \Vert_{L^\infty_\ttt L^4(\Sitt_\ttt)} \les& \,\CMD, \\
\Vert \wt{n}-1 \Vert_{L^\infty_\ttt L^\infty(\wt{\Si}_{\ttt})} + \Vert \wt\nab \wt{n} \Vert_{L^\infty_\ttt L^\infty(\wt{\Si}_\ttt)}+ \Vert \wt{\nab}^2 \wt{n} \Vert_{L^\infty_\ttt L^2(\wt{\Si}_\ttt)}+\Vert \wt{\nab} \Ttt( \wt{n}) \Vert_{L^\infty_\ttt L^2(\wt{\Si}_\ttt)} \les&\, \CMD,
\end{aligned} \label{EQM1KRSfoliationEPest1} \end{align}
where $\wt\nab$ denotes the covariant derivative on $\wt{\Si}_\ttt$, and moreover, by the combined higher regularity estimates of Theorems \ref{THMextensionConstraintsCZ1} and \ref{THMsmalldataL2details}, it holds that on $0 \leq \ttt \leq t^\ast$, 
\begin{align*}
\norm{\D \R}_{L^\infty_{\ttt} L^2(\wt{\Si}_\ttt)}+ \Vert \D^2 \wt{\pi} \Vert_{L^\infty_{\ttt} L^2(\wt{\Si}_\ttt)} \les& \Vert \nab \RRRic \Vert_{L^2(\Si_{t^\ast})}+ \Vert \nab^2 k \Vert_{L^2(\Si_{t^\ast})}+\CMD.
\end{align*}
Let $\wt{e}_0:=\Ttt$ denote the timelike unit normal to $\wt{\Si}_\Ttt$, and let $\wt E$ and $\wt H$ be the electric-magnetic decomposition with respect to $\Ttt$.

\item For each $\om \in\mathbb{S}^2$, the spacetime $\MM_{t^\ast}$ is foliated by a family of null hyperplanes $(\HH_{{}^\om u})_{{}^\om u \in \RRR}$ given as level sets of an optical function ${}^\om u$ satisfying
\begin{align*}
\sup\limits_{\om \in \mathbb{S}^2} \Vert \R \cdot \tilde{L} \Vert_{L^\infty_{{}^\om u}L^2(\HH_{{}^\om u})} \les \CMD,
\end{align*}
where $\Ltt$ is the $\HH_{{}^\om u}$-tangent null vectorfield with $\g(\Ltt,\Ttt)=-1$. 
\item Recall from \eqref{EQangleDefinition} that the angle $\nutt$ between $T$ and $\Ttt$ is defined by
\begin{align*}
\nutt := -\g(T, \Ttt).
\end{align*}
By Lemma \ref{lem:comparisonfoliationMM} proved in Appendix \ref{secComparisonAppendix}, it holds that for $\varep>0$ and $\CMD>0$ sufficiently small, along the foliation $(\Si_t)_{1 \leq t \leq t^\ast}$,
\begin{align} \label{EQlemmaComparisonTwoEstimates222}
\norm{\nutt-1}_{L^\infty(\MM_{t^\ast})} \les \CMD, \,\, \Vert \wt{k} \Vert_{L^\infty_t L^4(\Si_t)} \les \CMD,
\end{align}
where $\wt{k}$ denotes the second fundamental form of $\wt{\Si}_\ttt$.
\end{enumerate}

In the rest of this section we proceed as follows.
\begin{itemize}

\item In Section \ref{SECcurv1112} we prove by \emph{elliptic estimates} that on $\Si_{t^\ast}$,
\begin{align}\begin{aligned}
\Vert \nab {E} \Vert_{L^2(\Si_{t^\ast})}+ \Vert \nab {H} \Vert_{L^2({\Si}_{t^\ast})} \les \Vert \Lieh_\Ttt \R \Vert_{L^2({\Si}_{t^\ast})}+ \RR_1^\HH+\CMD.
\end{aligned}  \label{EQCurvatureEstimate1M1} \end{align}

\item In Section \ref{SECcurv111} we prove by using the Bel-Robinson tensor that for $0 \leq \ttt \leq t^\ast$,
\begin{align} \begin{aligned}
\Vert \Lieh_{\Ttt} \R \Vert_{L^\infty_{\ttt} L^2(\wt{\Si}_{\ttt}\cap \MM_{t^\ast})} 
\les & \RR_1^\Si+\RR_1^\HH \\
&+ \sqrt{\CMD} \lrpar{\norm{\nab \RRRic}_{L^2(\Si_{t^\ast})}+ \norm{\nab^2 k}_{L^2(\Si_{t^\ast})}} + \CMD,
\end{aligned} \label{EQgeneralResult} \end{align}
which necessitates a trilinear estimate.

\item In Section \ref{SECellipticEstimateM1}, we prove by \emph{elliptic estimates for $k$} that on $\Si_{t^\ast}$,
\begin{align} \begin{aligned} 
&\Vert \nab^2 k \Vert_{L^2(\Si_{t^\ast})}+ \Vert \Nd^2 \nu \Vert_{L^2(S_{t^\ast})}\\
 \les& \RR_1^\HH + \sqrt{\CMD} \lrpar{\norm{\nab \RRRic}_{L^2(\Si_{t^\ast})} +\norm{\nab^2 k}_{L^2(\Si_{t^\ast})} } + \CMD.
\end{aligned}\label{EQStatementHigherReg12} \end{align}

\item In Section \ref{SECm1estimatesConclusion}, we combine \eqref{EQCurvatureEstimate1M1}, \eqref{EQgeneralResult} and \eqref{EQStatementHigherReg12} to conclude the proof of \eqref{EQStatementHigherReg11000} and \eqref{EQStatementHigherReg11}, that is,
\begin{align*}
\Vert \nab E \Vert_{L^2(\Si_{t^\ast})} +\Vert \nab H \Vert_{L^2(\Si_{t^\ast})} \les& \, \OO_1^{\HH} + \RRt_1^\HH +\OO_1^{\Si} + \RRt_1^\Si +\CMD, \\
\Vert \nab^2 k \Vert_{L^2(\Si_{t^\ast})}+ \Vert \Nd^2 \nu \Vert_{L^2(S_{t^\ast})}+ \Vert \Nd^2 \nu \Vert_{ H^{1/2}(S_{t^\ast})}  \les& \, \OO_1^{\HH} + \RRt_1^\HH +\OO_1^{\Si} + \RRt_1^\Si +\CMD.
\end{align*}

\item In Section \ref{SECTestimateEHm1}, we prove \eqref{EQStatementHigherReg112}, that is,
\begin{align*}
\Vert \D\Rbf \Vert_{L^2(\Si_{t^\ast})}  \les& \, \OO_1^{\HH} + \RRt_1^\HH +\OO_1^{\Si} + \RRt_1^\Si +\CMD.
\end{align*}

\item In Section \ref{SECm1estimatesLAPSE}, we prove \eqref{EQStatementHigherReg13}, that is,
\begin{align*}
\Vert \nab^3 n \Vert_{L^2(\Si_{t^\ast})} +\Vert \nab^2 T( n) \Vert_{L^2(\Si_{t^\ast})} +\Vert \nab T^2(n) \Vert_{L^2(\Si_{t^\ast})}\les& \, \OO_1^{\Si} + \RR_1^\Si + \OO_1^{\HH}+ \RRt_1^\HH+\CMD.
\end{align*}

\end{itemize}

\subsubsection{Elliptic estimates for curvature: The proof of \eqref{EQCurvatureEstimate1M1}} \label{SECcurv1112} In this section, we prove \eqref{EQCurvatureEstimate1M1}. First we note that on $\Si_{t^\ast}$, by construction, $E=\wt{E}, H=\wt{H}, \wt\nab \wt{E} = \nab E$ and $\wt\nab \wt{H} = \nab H$. Therefore it suffices to prove that
\begin{align*}
\Vert \wt\nab \wt{E} \Vert_{L^2({\Si}_{t^\ast})}+ \Vert \wt\nab \wt{H} \Vert_{L^2({\Si}_{t^\ast})}\les \Vert \Lieh_\Ttt \R \Vert_{L^2(\wt{\Si}_{\ttt} \cap \MM_{t^\ast})}+ \RR_1^\HH+\CMD.
\end{align*}
By Proposition \ref{prop:MaxwellsEq1} with \eqref{EQTrelationEH}, and using that $\Rbf$ satisfies the homogeneous Bianchi equations, $\wt{E}$ and $\wt{H}$ satisfy the following Hodge system on ${\Si}_{t^\ast}$,
\begin{align*} \begin{aligned}
\wt\Div \wt{E} =& + \wt{k} \wedge \wt{H}, \\
\wt\Curl \wt{E} =& + \wt{H}(\Lieh_{\Ttt} \R ) -\frac{3}{2} \wt{k} \times \wt{H} - 3 \wt{n}^{-1} \wt\nab \wt{n} \wedge \wt{E}, \\
\wt\Div \wt{H}=& - \wt{k} \wedge \wt{E},\\
\wt\Curl \wt{H} =& - \wt{E}(\Lieh_\Ttt \R) + \frac{3}{2} \wt{k} \times \wt{E} - 3 \wt{n}^{-1} \wt\nab \wt{n} \wedge \wt{H}.
\end{aligned} 
\end{align*}
where $\wt\Div$ and $\wt\Curl$ denote the divergence and symmetrised curl operators on ${\Si}_\ast$, respectively. By application of the elliptic estimates of Corollary \ref{CORellipticEstimatesEH1} and using \eqref{EQM1KRSfoliationEPest1}, we thus get
\begin{align*}
&\int\limits_{{\Si}_{t^\ast}} \vert \wt\nab \wt{E} \vert^2 + \vert \wt\nab \wt{H} \vert^2 \\
\les& \int\limits_{{\Si}_{t^\ast}} \vert \Lieh_\Ttt\mathbf{R} \vert^2 + \int\limits_{\pr{\Si}_{t^\ast}} \wt{\nab}_b \wt{E}_{aN} \, \wt{E}^{ab} +  \int\limits_{\pr{\Si}_{t^\ast}} \wt{\nab}_b \wt{H}_{aN} \, \wt{H}^{ab} \\
&+ \CMD^2 \lrpar{\Vert \wt{\nab} \wt{H} \Vert^2_{L^2({\Si}_{t^\ast})} +\Vert \wt{\nab} \wt{E} \Vert^2_{L^2({\Si}_{t^\ast})} } +\CMD^4.
\end{align*}
Thus for $\CMD>0$ and $\varep>0$ sufficiently small, we get
\begin{align*}
\int\limits_{{\Si}_{t^\ast}} \vert \wt\nab \wt{E} \vert^2 + \vert \wt\nab \wt{E} \vert^2 \les&  \int\limits_{{\Si}_{t^\ast}} \vert \Lieh_\Ttt\mathbf{R} \vert^2 + \int\limits_{\pr{\Si}_{t^\ast}} \nab_b {E}_{aN} \, {E}^{ab} +  \int\limits_{\pr{\Si}_{t^\ast}} \nab_b {H}_{aN} \, {H}^{ab} +\CMD^2,
\end{align*}
where we used that $\nab= \wt{\nab}$ and $E=\wt{E}, H=\wt{H}$ on $\Si_{t^\ast}$.\\

Using the spacetime relations
\begin{align*}
{\nab}_a {E}_{bN} =& \D_a \Rbf_{T bT N} - {k}_{ac} \Rbf_{cb{T}N} - {k}_{ac} \Rbf_{T b c N},\\
{\nab}_a {H}_{bN} =& \D_a {}^\ast\Rbf_{T bT N} - {k}_{ac} {}^\ast \Rbf_{cb{T}N} - {k}_{ac} {}^\ast\Rbf_{T b c N},
\end{align*}
we can estimate the boundary integrals on the right-hand side above for $\CMD>0$ and $\varep>0$ small by
\begin{align*}
& \int\limits_{\pr{\Si}_{t^\ast}} {\nab}_b {E}_{aN} \, {E}^{ab} +  \int\limits_{\pr{\Si}_{t^\ast}} {\nab}_b {H}_{aN} \, {H}^{ab} \\
 \les& \int\limits_{\pr{\Si}_{t^\ast}} \vert \D \Rbf \vert_{{\mathbf{h}^t}}^2 + \vert \Rbf \vert_{{\mathbf{h}^t}}^2 + \vert {k}\vert \vert \Rbf \vert_{{\mathbf{h}^t}}^2 \\
 \les& \Vert \D \Rbf \Vert_{L^\infty(\HH)}^2+ \Vert \Rbf \Vert_{L^\infty(\HH)}^2 + \Vert {k} \Vert_{L^1\lrpar{\pr{\Si}_{t^\ast}}}  \Vert \Rbf \Vert_{L^\infty(\HH)}^2 \\
  \les& \Vert \D \Rbf \Vert_{L^\infty(\HH)}^2 + \Vert \Rbf \Vert_{L^\infty(\HH)}^2+ \lrpar{\Vert \nab {k} \Vert_{L^2({\Si}_{t^\ast})}+\Vert {k} \Vert_{L^2({\Si}_{t^\ast})} }  \Vert \Rbf \Vert_{L^\infty(\HH)}^2 \\
 \les& \Vert \D \Rbf \Vert_{L^\infty(\HH)}^2 + \Vert \Rbf \Vert_{L^\infty(\HH)}^2+ \varep  \Vert \Rbf \Vert_{L^\infty(\HH)}^2 \\
  \les& \RR_1^\HH.
\end{align*}

To summarise the above, we get that for $\CMD>0$ and $\varep>0$ sufficiently small, for $0\leq \ttt\leq t^\ast$,
\begin{align*}
\Vert \wt\nab \wt{E} \Vert_{L^2({\Si}_{t^\ast})}+ \Vert \wt\nab \wt{H} \Vert_{L^2({\Si}_{t^\ast})}\les \Vert \Lieh_\Ttt \R \Vert_{L^2(\wt{\Si}_{\ttt} \cap \MM_{t^\ast})}+ \RR_1^\HH+\CMD.
\end{align*}
This finishes the proof of \eqref{EQCurvatureEstimate1M1}.

\subsubsection{Energy estimate for the curvature tensor: The proof of \eqref{EQgeneralResult}} \label{SECcurv111} In this section, we prove that for $0 \leq \ttt \leq t^\ast$,
\begin{align*}
\Vert \Lieh_{\Ttt} \R \Vert_{L^\infty_{\ttt} L^2(\wt{\Si}_{\ttt}\cap \MM_{t^\ast})} 
\les \RR_1^\Si+\RR_1^\HH + \sqrt{\CMD} \lrpar{\norm{\nab \RRRic}_{L^2(\Si_{t^\ast})}+ \norm{\nab^2 k}_{L^2(\Si_{t^\ast})}} + \CMD.
\end{align*}

Indeed, applying \eqref{lemma:IntegralidentityWeyl1} to the Weyl tensor $W := \Lieh_{\Ttt} \R$ yields
\begin{align} \begin{aligned}
&\Vert \Lieh_{\Ttt} \R \Vert^2_{L^\infty_{\ttt} L^2(\wt{\Si}_{\ttt} \cap \MM_{t^\ast})} +  \sup\limits_{\om \in \mathbb{S}^2} \Vert \Lieh_{\Ttt}\Rbf \cdot \Ltt \Vert^2_{L^\infty_{{}^\om u}L^2(\HH_{{}^\om u} \cap \MM_{t^\ast})} \\
\les& \int\limits_{\Si_1} Q(\Lieh_{\Ttt} \R)_{\Ttt \Ttt \Ttt  T}+ \int\limits_{\HH} Q(\Lieh_{\Ttt} \R)_{\Ttt \Ttt \Ttt L} \\
&- \underbrace{\int\limits_{\MM_{t^\ast}} \frac{3}{2} Q(\Lieh_{\Ttt} \R)_{\a \be \Ttt \Ttt}\wt{\pi}^{\a\be}}_{:=\EE_1} - \underbrace{\int\limits_{\MM_{t^\ast}} \D^\a Q(\Lieh_{\Ttt} \R)_{\a \Ttt \Ttt \Ttt}}_{:=\EE_2},
\end{aligned} \label{eq:initialBRestimate} \end{align}
where the integral over $\HH$ is defined in Definition \ref{DEFintegrationH}.\\

The terms $\EE_1$ and $\EE_2$ are estimated by the following trilinear estimate.
\begin{proposition}[Trilinear estimate for $m=1$] \label{PropTrilinearM1} For $\varep>0$ and $\CMD>0$ sufficiently small, it holds that
\begin{align} \begin{aligned} 
\vert \EE_1\vert  + \vert \EE_2 \vert \les& \CMD \Vert \Lieh_{\Ttt} \R \Vert^2_{L^\infty_{\ttt} L^2(\wt{\Si}_{\ttt}\cap \MM_{t^\ast})} + \CMD \sup\limits_{\om \in \mathbb{S}^2} \Vert \Lieh_{\Ttt} \Rbf \cdot \Ltt \Vert^2_{L^\infty_{{}^\om u}L^2(\HH_{{}^\om u} \cap \MM_{t^\ast})} \\
&+ \CMD \lrpar{\norm{\nab \RRRic}_{L^2(\Si_{t^\ast})}+ \norm{\nab^2 k}_{L^2(\Si_{t^\ast})} + \CMD}^2 + \CMD^2.
\end{aligned} \label{EQpreliminaryTrilinearM1}\end{align}
\end{proposition}

\begin{remark} In Appendix \ref{SECtrilinearEstimateM1}, we argue that Proposition \ref{eq:initialBRestimate} follows readily from the $(m=1)$-estimates proved in Section 13 of \cite{KRS} in the Yang-Mills formalism by estimating the most difficult integrand terms of $\EE_1$ and $\EE_2$ through their corresponding estimates in \cite{KRS}.\end{remark}


Plugging \eqref{EQpreliminaryTrilinearM1} into \eqref{eq:initialBRestimate}, we get that for $\CMD>0$ and $\varep>0$ sufficiently small,
\begin{align} \begin{aligned}
&\Vert \Lieh_{\Ttt} \R \Vert^2_{L^\infty_{\ttt} L^2(\wt{\Si}_{\ttt} \cap \MM_{t^\ast})} +  \sup\limits_{\om \in \mathbb{S}^2} \Vert \Lieh_{\Ttt}\Rbf \cdot \Ltt \Vert^2_{L^\infty_{{}^\om u}L^2(\HH_{{}^\om u} \cap \MM_{t^\ast})} \\
\les&\int\limits_{\HH} Q(\Lieh_{\Ttt} \R)_{\Ttt \Ttt \Ttt L} + \int\limits_{\Si_1} Q(\Lieh_{\Ttt} \R)_{\Ttt \Ttt \Ttt  T} \\
&+ \CMD \Vert \Lieh_{\Ttt} \R \Vert^2_{L^\infty_{\ttt} L^2(\wt{\Si}_{\ttt}\cap \MM_{t^\ast})} + \CMD \sup\limits_{\om \in \mathbb{S}^2} \Vert \Lieh_{\Ttt} \Rbf \cdot \Ltt \Vert^2_{L^\infty_{{}^\om u}L^2(\HH_{{}^\om u} \cap \MM_{t^\ast})} \\
&+ \CMD \lrpar{\norm{\nab \RRRic}_{L^2(\Si_{t^\ast})}+ \norm{\nab^2 k}_{L^2(\Si_{t^\ast})} + \CMD}^2 + \CMD^2\\
\les&\underbrace{ \int\limits_{\HH} Q(\Lieh_{\Ttt} \R)_{\Ttt \Ttt \Ttt L}}_{:=\II_1} + \underbrace{\int\limits_{\Si_1} Q(\Lieh_{\Ttt} \R)_{\Ttt \Ttt \Ttt  T}}_{:=\II_2} \\
&+ \CMD \lrpar{\norm{\nab \RRRic}_{L^2(\Si_{t^\ast})}+ \norm{\nab^2 k}_{L^2(\Si_{t^\ast})}}^2 + \CMD^2,
\end{aligned} \label{EQM1curvatureestimateWterms} \end{align}
where we used the smallness of $\CMD>0$ to absorb the second and third term on the right-hand side of the first inequality into the left-hand side. It remains to estimate $\II_1$ and $\II_2$ on the right-hand side of \eqref{EQM1curvatureestimateWterms}.\\

\textbf{Estimation of $\II_1$.} By definition of $Q$ in Definition \ref{def:BelRobTensor},
\begin{align*}
Q(\Lieh_\Ttt \R)_{\Ttt \Ttt \Ttt L} = (\Lieh_\Ttt \R)_{\Ttt \mu \Ttt \nu} (\Lieh_\Ttt \R)_{\Ttt \,\,\,\, L}^{\,\,\,\, \nu \,\,\,\, \mu} + \text{dual term}, 
\end{align*}
where by Definition \ref{DEFmodifiedLiederivative} and using that $(\wt{\Si}_\ttt)$ is maximal,
\begin{align*}
(\Lieh_\Ttt \R)_{\a \be \ga \de} :=& (\Lie_\Ttt \R)_{\a \be \ga \de} - \half \left(\wt{\pi}^\mu_{\,\,\, \a} \R_{\mu \be \ga \de} + \wt{\pi}^\mu_{\,\,\, \be} \R_{\a \mu \ga \de}+ \wt{\pi}^\mu_{\,\,\, \ga} \R_{\a \be \mu \de} + \wt{\pi}^\mu_{\,\,\, \de} \R_{\a \be \ga \mu} \right),
\end{align*}
which can be written schematically as
\begin{align*}
\Lieh_\Ttt \R = \D \R + \wt{\pi} \cdot \R,
\end{align*}
and thus
\begin{align*}
Q(\Lieh_\Ttt \R)_{\Ttt \Ttt \Ttt L} = \D \R \cdot \D \R + \wt{\pi} \cdot \R  \cdot \D \R + \wt{\pi} \cdot \wt{\pi} \cdot \R \cdot \R.
\end{align*}
Therefore, for $\varep>0$ and $\CMD>0$ sufficiently small,
\begin{align} \begin{aligned} 
\II_1 :=\int\limits_{\HH} Q(\Lieh_{\Ttt} \R)_{\Ttt \Ttt \Ttt L} \les& \int\limits_{\HH} \vert \D \R \vert_{\mathbf{{h}}^\ttt}^2 + \vert \wt{\pi} \vert_{\mathbf{{h}}^\ttt} \vert \R \vert_{\mathbf{{h}}^\ttt} \vert \D \R \vert_{\mathbf{{h}}^\ttt} + \vert\wt{\pi} \vert_{\mathbf{{h}}^\ttt}^2 \vert \R \vert_{\mathbf{{h}}^\ttt}^2 \\
\les& \int\limits_{\HH} \vert \D \R \vert_{\mathbf{{h}}^v}^2 + \vert \wt{\pi} \vert_{\mathbf{{h}}^\ttt} \vert \R \vert_{\mathbf{{h}}^v} \vert \D \R \vert_{\mathbf{{h}}^v} + \vert \wt{\pi} \vert_{\mathbf{{h}}^\ttt}^2 \vert \R \vert_{\mathbf{{h}}^v}^2
\end{aligned} \label{ESTQQ1estimM1new} \end{align} 
where we used \eqref{EQm1regEstimatesAssumption} and \eqref{EQlemmaComparisonTwoEstimates222} to compare $\mathbf{{h}}^\ttt$ with $\mathbf{{h}}^v$ on $\HH$.\\

By Lemma \ref{Lemma:TraceEstimateH1toL4bdry} and \eqref{EQM1KRSfoliationEPest1},
\begin{align*}
\int\limits_\HH \vert \wt{\pi} \vert_{\mathbf{{h}}^\ttt} + \int\limits_\HH \vert \wt{\pi} \vert^2_{\mathbf{{h}}^\ttt} \les& \Vert \wt{\pi} \Vert_{L^\infty_\ttt L^2(\pr \wt{\Si}_\ttt)} +  \Vert \wt{\pi} \Vert_{L^\infty_\ttt L^2(\pr \wt{\Si}_\ttt)}^2 \\
\les& \lrpar{\Vert \wt{\pi} \Vert_{L^\infty_\ttt L^2(\wt{\Si}_\ttt)} + \Vert \nab\wt{\pi} \Vert_{L^\infty_\ttt L^2(\wt{\Si}_\ttt)} }+\lrpar{\Vert \wt{\pi} \Vert_{L^\infty_\ttt L^2(\wt{\Si}_\ttt)} + \Vert \nab \wt{\pi} \Vert_{L^\infty_\ttt L^2(\wt{\Si}_\ttt)}}^2  \\
 \les&\, \CMD + \CMD^2.
\end{align*}

Plugging this into \eqref{ESTQQ1estimM1new} yields that for $\CMD>0$ sufficiently small,
\begin{align} \label{ESTQQ1M1}
\II_1 \les& \Vert \R \Vert^2_{L^\infty(\HH)} + \Vert \D \R \Vert^2_{L^\infty(\HH)} \les \lrpar{\RR_1^\HH}^2.
\end{align}

\textbf{Estimation of $\II_2$.} First, by \eqref{EQrelationQothervectors1} and \eqref{EQlemmaComparisonTwoEstimates222}, 
\begin{align*}
\int\limits_{\Si_1} Q(\Lieh_\Ttt \R)_{\Ttt \Ttt \Ttt T} \les \int\limits_{\Si_1} Q(\Lieh_\Ttt \R)_{\Ttt \Ttt \Ttt \Ttt} \les \int\limits_{\Si_1} \vert \wt{E}(\Lieh_\Ttt \R) \vert^2 + \vert \wt{H}(\Lieh_\Ttt \R) \vert^2.
\end{align*}
Second, by definition of $\Lieh_\Ttt$, see Definition \ref{DEFmodifiedLiederivative}, and using that $(\wt{\Si}_\ttt)$ is maximal, we have for an $\wt{\Si}_\ttt$-tangential frame $(\ett_a)_{a=1,2,3}$, 
\begin{align*}
\wt{E}(\Lieh_\Ttt \R)_{ab} =& (\Lie_\Ttt \R)_{\Ttt a\Ttt b} -\half \left( \wt{\pi}^c_{\,\,\, \Ttt} \R_{ca\Ttt b} + \wt{\pi}^c_{\,\,\, a} \R_{\Ttt c\Ttt b} + \wt{\pi}^c_{\,\,\, \Ttt} \R_{\Ttt acb} + \wt{\pi}^c_b \R_{\Ttt a\Ttt c} \right),
\end{align*}
The Lie derivative on the right-hand side can be rewritten as
\begin{align*}
(\Lie_\Ttt \R)_{\Ttt a\Ttt b} =& \D_\Ttt \R_{\Ttt  a\Ttt b} -\ntt^{-1} \nabtt_c \ntt \left( \R_{c a \Ttt b} + \R_{\Ttt a c b} \right) - \ktt_{ac} \R_{\Ttt c \Ttt b} - \ktt_{bc} \R_{\Ttt a \Ttt c}.
\end{align*}
From the above two, we get that
\begin{align*}
\int\limits_{\Si_1} \vert \wt{E}( \Lieh_\Ttt \R) \vert^2 \les& \Vert \vert \D \R \vert_{\mathbf{\tilde{h}}} \Vert^2_{L^2(\Si_1)} + \Vert \vert \R \vert_{\mathbf{\tilde{h}}} \Vert^2_{L^2(\Si_1)} \Vert \nabtt \ntt \Vert^2_{L^\infty(\MM)} + \Vert \vert \R \vert_{\mathbf{\tilde{h}}} \Vert_{L^\infty(\Si_1)}^2 \Vert \ktt \Vert^2_{L^2(\Si_1)} \\
\les& \Vert \D \R \Vert^2_{L^\infty(\Si_1)} + \Vert \R \Vert^2_{L^\infty(\Si_1)} \\
\les& \lrpar{\RR_1^\Si}^2,
\end{align*}
where we used \eqref{EQM1KRSfoliationEPest1}, and $\mathbf{{h}}^\ttt$ and $\mathbf{{h}}^t$ denote the Riemannian metrics corresponding to $(\wt{\Si}_\ttt)$ and $(\Si_t)$, respectively, and we used \eqref{EQlemmaComparisonTwoEstimates222} to compare $\mathbf{{h}}^\ttt$ and $\mathbf{{h}}^t$.\\

It follows similarly that
\begin{align*}
\int\limits_{\Si_1} \vert \wt{H}( \Lieh_\Ttt \R) \vert^2 \les \lrpar{\RR_1^\Si}^2;
\end{align*}
details are left to the reader. To summarise the above, we proved that
\begin{align} \label{EQFinalTERM2estimateM1Curvature}
\II_2 := \int\limits_{\Si_1} Q(\Lieh_\Ttt \R)_{\Ttt \Ttt \Ttt T} \les \lrpar{\RR_1^\Si}^2. 
\end{align}

Plugging \eqref{ESTQQ1M1} and \eqref{EQFinalTERM2estimateM1Curvature} into \eqref{EQM1curvatureestimateWterms} shows that 
\begin{align*}
\Vert \Lieh_{\Ttt} \R \Vert^2_{L^\infty_{\ttt} L^2(\wt{\Si}_{\ttt})} \les& \lrpar{\RR_1^\Si}^2+\lrpar{\RR_1^\HH}^2 \\
&+ \CMD \lrpar{\norm{\nab \RRRic}_{L^2(\Si_{t^\ast})}+ \norm{\nab^2 k}_{L^2(\Si_{t^\ast})}}^2 + \CMD^2.
\end{align*}
This finishes the proof of \eqref{EQgeneralResult}.

\subsubsection{Elliptic estimates for $k$ on $\Si_{t^\ast}$ for $m=1$} \label{SECellipticEstimateM1} In this section we prove \eqref{EQStatementHigherReg12}, that is,
\begin{align*}
&\Vert \nab^2 k \Vert_{L^2(\Si_{t^\ast})}+ \Vert \Nd^2 \nu \Vert_{L^2(\pr \Si_{t^\ast})}\\
\les& \RR_1^\HH + \CMD + \CMD \lrpar{\Vert \nab \RRRic \Vert_{L^2(\Si_{t^\ast})}+\Vert \nab^2 k \Vert_{L^2(\Si_{t^\ast})} }.
\end{align*}

Analogously to Section \ref{SectionKEllipticEstimateWithBoundaryTerm}, the idea is to use elliptic estimates for $k$ and exploit the special structure of the appearing boundary integral. For completeness, we provide more details below.\\

We recall that $k$ satisfies on $\Si_{t^\ast}$ the Hodge system
\begin{align*}
\Div_g k=&0, \\
{\Curl}_g k=&H, \\
\tr_g k =&0.
\end{align*}
In the following higher regularity estimates for $k$, we use the notation of Appendix \ref{sec:AppendixProofOfLowRegEllipticEstimateForK}. In the notation of Appendix \ref{sec:AppendixProofOfLowRegEllipticEstimateForK}, the above Hodge system of $k$ implies that
\begin{align} \label{EQkHodgem1Cite}
A(k)_{iab} = \in^m_{\,\,\,\, ab}H_{im}, \,\, D(k)=0.
\end{align}
We note that by \eqref{EQkHodgem1Cite} together with Lemma \ref{lem:controlSYM}, we can express the symmetrised derivative $\ol{\nab}k$ of $k$ as
\begin{align} \begin{aligned}
\lrpar{\ol{\nab}k}_{a_1 a_2 b} =&\nab_b k_{a_1a_2} + \frac{1}{3} \in^m_{\,\,\, ba_1} H_{a_2m}  +\frac{1}{3}  \in^m_{\,\,\, ba_2} H_{a_1m}.
\end{aligned} \label{EQm1RELATIONSk} \end{align}

Applying the fundamental elliptic estimate for Hodge systems (see Lemma \ref{lem:IntegrationByPartsIdentity} and note that it applies only to symmetric tensors) to the symmetrised derivative $\ol{\nab}k$ of $k$ and using Lemmas \ref{lemma:commutatorDA} and \ref{lem:controlSYM}, we get the next elliptic estimate (see also Lemma \ref{LEMkHigherRegEllEstBdryTerms})
\begin{align} \begin{aligned}
&\int\limits_{\Si_{t^\ast}} \vert \nab^2 k \vert^2 + \int\limits_{\pr \Si_{t^\ast}} \lrpar{\ol{\nab}k}^{a_1 a_2 N} D(\ol{\nab} k)_{a_1 a_2} - \int\limits_{\pr \Si_{t^\ast}} \nab_b \lrpar{\ol{\nab}k}_{a_1a_2N} \lrpar{\ol{\nab}k}^{a_1 a_2 b} \\
& \les \int\limits_{\Si_{t^\ast}} \vert \nab H \vert^2  + \CMD^2,
\end{aligned} \label{EQellipticEstimatekM1overview} \end{align}
By the definition of the divergence $D$, see Definition \ref{def:AandDdef}, we can rewrite the boundary integrals as
\begin{align} \begin{aligned}
&\int\limits_{\pr \Si_{t^\ast}} \lrpar{\ol{\nab}k}^{a_1 a_2 N} D(\ol{\nab} k)_{a_1 a_2} - \int\limits_{\pr \Si_{t^\ast}} \nab_b \lrpar{\ol{\nab}k}_{a_1a_2N} \lrpar{\ol{\nab}k}^{a_1 a_2 b} \\
=&  \int\limits_{\pr \Si_{t^\ast}} \lrpar{\ol{\nab}k}^{a_1 a_2 N} \nab^D (\ol{\nab} k)_{a_1 a_2 D} - \int\limits_{\pr \Si_{t^\ast}} \nab_D \lrpar{\ol{\nab}k}_{a_1a_2N} \lrpar{\ol{\nab}k}^{a_1 a_2 D}.
\end{aligned} \label{EQboundaryTermsKm1est1} \end{align}

In the following, it suffices to analyse the first term on the right-hand side of \eqref{EQboundaryTermsKm1est1}. Indeed, by an integration by parts on $\pr \Si_{t^\ast}$, the second term equals the first term up to error terms $Q$ which can be estimated as
\begin{align*}
\vert Q \vert 
\les& \, \Vert \nab E \Vert^2_{L^2(\Si_{t^\ast})} +\Vert \nab H \Vert^2_{L^2(\Si_{t^\ast})} +\CMD \\
&\, + \CMD (\Vert \nab \RRRic \Vert^2_{L^2(\Si_{t^\ast})}+\Vert \nab^2 k \Vert^2_{L^2(\Si_{t^\ast})} + \CMD^2).
\end{align*}
where we used the property that $\Si_{t^\ast}$ is a weakly regular ball of constant $\CMD$. In the following, we write $Q$ as general notation for such error terms.\\

We turn to the analysis of the first term on the right-hand side of \eqref{EQboundaryTermsKm1est1}. We consider three cases.\\

\textbf{Case 1: $a_1, a_2 \in \{1,2\}$.} In this case denote $A_1 := a_1$ and $A_2 := a_2$. By \eqref{EQm1RELATIONSk}, see also \eqref{eq:FirstOfDivEq}-\eqref{eq:divEtaFoundation}, and integration by parts on $\pr \Si_{t^\ast}$,
\begin{align*}
&\int\limits_{\pr \Si_{t^\ast}} \lrpar{\ol{\nab}k}^{A_1 A_2 N} \nab^D \lrpar{\ol{\nab}k}_{A_1 A_2 D} \\
=&\int\limits_{\pr \Si_{t^\ast}} \lrpar{\nab^{A_1} k^{A_2 N}+ \frac{1}{3} \in^{mNA_1} H^{A_2}_{\,\,\,\,\,\,m}+ \frac{1}{3} \in^{mNA_2} H^{A_1}_{\,\,\,\,\,\,m}} \\
&\qquad \cdot \nab^D \lrpar{\nab_{A_2} k_{A_1D} + \frac{1}{3} \in^{mDA_2} H^{A_1}_{\,\,\,\,\,\,m}+ \frac{1}{3} \in^{mdA_1} H^{A_2}_{\,\,\,\,\,\,m}} \\
=&\int\limits_{\pr \Si_{t^\ast}} \Nd_{A_2} k^{A_2N} \Nd^{A_1} \Nd^D k_{A_1 D} + Q\\
=& \int\limits_{\pr \Si_{t^\ast}} \Divd \ep \Divd \Divd \eta + Q\\
=&- \int\limits_{\pr \Si_{t^\ast}} \Divd \ep \Divd \Nd \de + Q.
\end{align*}
Using the slope equation \eqref{eq:slope} and \eqref{eq:deltanuRelation},
\begin{align*}
\ep_A =& -\nut^{-1}\Nd_A\nut + \zet_A,\\
 \Nd \de =& \Nd \left( \half \nu \tr \chi + \half \nu^{-1} \tr \chib \right) =  \underbrace{F(\nu, \tr \chi, \tr \chib)}_{\geq 1/8.} \Nd \nu + \half \nu \Nd \tr \chi + \half \nu^{-1} \Nd \tr \chib,
\end{align*}
we get from the above and standard elliptic estimates on $\pr \Si_{t^\ast}$ that
\begin{align*}
\int\limits_{\pr \Si_{t^\ast}} \lrpar{\ol{\nab}k}^{A_1 A_2 N} \nab^D \lrpar{\ol{\nab}k}_{A_1 A_2 D} =& \int\limits_{\pr \Si_{t^\ast}} F(\nu, \tr \chi, \tr \chib) \vert \Ld \nu \vert^2 + Q\\
\gtrsim& \int\limits_{\pr \Si_{t^\ast}} \vert \Ld \nu \vert^2 + Q\\
\gtrsim& \int\limits_{\pr \Si_{t^\ast}} \vert \Nd^2 \nu \vert^2 + Q.
\end{align*}
This finishes our discussion of Case 1.\\

\textbf{Case 2: $a_1 \in \{1,2\}, a_2 =N$.} In this case let $A_1 := a_1$. We have
\begin{align*}
\int\limits_{\pr \Si_{t^\ast}} \lrpar{\ol{\nab}k}^{A_1NN} \nab^D (\ol{\nab} k)_{A_1ND} =& \int\limits_{\pr \Si_{t^\ast}} \nab^{A_1}k^{NN} \nab^D \nab_{A_1} k_{ND} + Q \\
=&\int\limits_{\pr \Si_{t^\ast}} \Nd^{A_1}\de \, \Nd_{A_1} \Divd \ep + Q \\
=&-\int\limits_{\pr \Si_{t^\ast}} \Divd \Nd\de \Divd \ep + Q.
\end{align*}
Hence by the same reasoning as in Case 1, we get that
\begin{align*}
\int\limits_{\pr \Si_{t^\ast}} \lrpar{\ol{\nab}k}^{A_1NN} \nab^D (\ol{\nab} k)_{A_1ND} \gtrsim \int\limits_{\pr \Si_{t^\ast}} \vert \Nd^2 \nu \vert^2 + Q.
\end{align*}
This finishes our discussion of Case 2.\\

\textbf{Case 3: $a_1=a_2 =N$.} The idea is to use \eqref{EQm1RELATIONSk}, see also \eqref{eq:FirstOfDivEq}-\eqref{eq:divEtaFoundation}, to reduce the number of $N$'s. Indeed, we have
\begin{align*}
\int\limits_{\pr \Si_{t^\ast}} \lrpar{\ol{\nab}k}^{NNN} \nab^D (\ol{\nab} k)_{NND} =& \int\limits_{\pr \Si_{t^\ast}} \nab^N k^{NN} \nab^D \nab_D k_{NN} +Q \\
=& -\int\limits_{\pr \Si_{t^\ast}} \nab^A k^{NA} \nab^D \nab_D k_{NN} +Q \\
=& -\int\limits_{\pr \Si_{t^\ast}} \Divd \ep \Divd \Nd \de +Q.
\end{align*}
Hence by the same reasoning as in Case 1, we get that
\begin{align*}
\int\limits_{\pr \Si_{t^\ast}} \lrpar{\ol{\nab}k}^{NNN} \nab^D (\ol{\nab} k)_{NND} \gtrsim \int\limits_{\pr \Si_{t^\ast}} \vert \Nd^2 \nu \vert^2 + Q.
\end{align*}
This finishes our discussion of Case 3.\\

By plugging the above estimates for Cases 1, 2 and 3 with \eqref{EQboundaryTermsKm1est1} into \eqref{EQellipticEstimatekM1overview}, we get that
\begin{align*}
&\, \Vert \nab^2 k \Vert^2_{L^2(\Si_{t^\ast})} + \Vert \Nd^2 \nu \Vert^2_{L^2(\pr \Si_{t^\ast})} \\
\les& \,\Vert \nab H \Vert^2_{L^2(\Si_{t^\ast})}  + \CMD^2 + Q \\
\les& \,\Vert \nab E \Vert^2_{L^2(\Si_{t^\ast})} +\Vert \nab H \Vert^2_{L^2(\Si_{t^\ast})} +\CMD^2 + \CMD (\Vert \nab \RRRic \Vert^2_{L^2(\Si_{t^\ast})}+\Vert \nab^2 k \Vert^2_{L^2(\Si_{t^\ast})} + \CMD^2) \\
\les&\, \RR_1^\HH +\CMD^2 + \CMD (\Vert \nab \RRRic \Vert^2_{L^2(\Si_{t^\ast})}+\Vert \nab^2 k \Vert^2_{L^2(\Si_{t^\ast})} + \CMD^2).
\end{align*}
This finishes the proof of \eqref{EQStatementHigherReg12}.

\subsubsection{Conclusion of the proof of \eqref{EQStatementHigherReg11000} and \eqref{EQStatementHigherReg11}} \label{SECm1estimatesConclusion} In this section, we conclude the proof of \eqref{EQStatementHigherReg11000} and \eqref{EQStatementHigherReg11}, that is,
\begin{align*} 
\Vert \nab E \Vert_{L^2(\Si_{t^\ast})} +\Vert \nab H \Vert_{L^2(\Si_{t^\ast})} \les& \OO_1^{\HH} + \RRt_1^\HH +\OO_1^{\Si} + \RRt_1^\Si+\CMD, \\
\Vert \nab^2 k \Vert_{L^2(\Si_{t^\ast})}+ \Vert \Nd^2 \nu \Vert_{L^2(\pr \Si_{t^\ast})}+\Vert \Nd^2 \nu \Vert_{H^{1/2}(\pr \Si_{t^\ast})} \les& \OO_1^{\HH} + \RRt_1^\HH +\OO_1^{\Si} + \RRt_1^\Si+\CMD.
\end{align*}

Combining \eqref{EQCurvatureEstimate1M1}, \eqref{EQgeneralResult} and \eqref{EQStatementHigherReg12}, and noting that on $\Si_{t^\ast}$, $\wt\nab \wt{E} =\nab E$ and $\wt\nab \wt{H} =\nab H$, we have
\begin{align} \begin{aligned}
&\Vert \nab E \Vert_{L^2(\Si_{t^\ast})} +\Vert \nab H \Vert_{L^2(\Si_{t^\ast})} +\Vert \nab^2 k \Vert_{L^2(\Si_{t^\ast})}+ \Vert \Nd^2 \nu \Vert_{L^2(\pr \Si_{t^\ast})} \\
 \les& \OO_1^{\HH} + \RRt_1^\HH +\OO_1^{\Si} + \RRt_1^\Si +\sqrt{ \CMD} \lrpar{\Vert \nab^2 k \Vert^2_{L^2(\Si_{t^\ast})} + \Vert \nab \RRRic \Vert_{L^2(\Si_{t^\ast})}^2}+\CMD.
\end{aligned} \label{EQcombinedEstimateConclusion1} \end{align}

Using that by \eqref{eq:RicE},
\begin{align*}
\Ric_{ij} = E_{ij}+k_{ia}k^{a}_j,
\end{align*}
we have that 
\begin{align} \label{EQRicEconnectionEstimate1}
\Vert \nab \RRRic \Vert_{L^2(\Si_{t^\ast})} \les \Vert \nab E \Vert_{L^2(\Si_{t^\ast})} + \sqrt{\CMD} \Vert \nab^2 k \Vert_{L^2(\Si_{t^\ast})}+ \CMD.
\end{align}
Plugging \eqref{EQRicEconnectionEstimate1} into \eqref{EQcombinedEstimateConclusion1}, we get that for $\CMD>0$ and $\varep>0$ sufficiently small,  
\begin{align*}
&\Vert \nab E \Vert_{L^2(\Si_{t^\ast})} +\Vert \nab H \Vert_{L^2(\Si_{t^\ast})}+ \Vert \nab^2 k \Vert_{L^2(\Si_{t^\ast})}+ \Vert \Nd^2 \nu \Vert_{L^2(\pr \Si_{t^\ast})}\\
\les& \OO_1^{\HH} + \RRt_1^\HH +\OO_1^{\Si} + \RRt_1^\Si + \sqrt{\CMD} \lrpar{\norm{\nab \RRRic}_{L^2(\Si_{t^\ast})}+ \norm{\nab^2 k}_{L^2(\Si_{t^\ast})}} +\CMD \\
\les& \OO_1^{\HH} + \RRt_1^\HH +\OO_1^{\Si} + \RRt_1^\Si + \sqrt{\CMD} \lrpar{\norm{\nab E}_{L^2(\Si_{t^\ast})}+(1+\sqrt{\CMD}) \norm{\nab^2 k}_{L^2(\Si_{t^\ast})} +\CMD} \\
\les&  \OO_1^{\HH} + \RRt_1^\HH +\OO_1^{\Si} + \RRt_1^\Si + \CMD,
\end{align*}
where we used the smallness of $\CMD>0$ to absorb the term into the left-hand side. This finishes the proof of \eqref{EQStatementHigherReg11000}. \\

For the proof of \eqref{EQStatementHigherReg11}, it remains to estimate $\Vert \Nd^2 \nu \Vert_{H^{1/2}(\pr \Si_{t^\ast})}$. Using the slope equation \eqref{eq:slope}, we get that
\begin{align*}
\Vert \Nd^2 \nu \Vert_{H^{1/2}(\pr \Si_{t^\ast})} \les& \Vert \Nd \zeta \Vert_{H^{1/2}(\pr \Si_{t^\ast})} + \Vert \Nd \ep \Vert_{H^{1/2}(\pr \Si_{t^\ast})} \\
\les& \OO_1^\HH + \lrpar{\Vert \nab^2 k \Vert_{L^2(\Si_{t^\ast})} + \Vert \nab \Ric \Vert_{L^2(\Si_{t^\ast})}+ \CMD} \\
\les& \OO_1^\HH + \RR_1^\Si + \RR_1^\HH + \CMD,
\end{align*}
where we used the above estimate for $k$ and $\nab \RRRic$ and Lemma \ref{Lemma:TraceEstimateH1toL4bdry}.

\subsubsection{Proof of \eqref{EQStatementHigherReg112}} \label{SECTestimateEHm1}

In this section we prove \eqref{EQStatementHigherReg112}, that is,
\begin{align*}
\Vert \D \Rbf \Vert_{L^2(\Si_{t^\ast})} \les  \OO_1^{\HH} + \RRt_1^\HH +\OO_1^{\Si} + \RRt_1^\Si + \CMD.
\end{align*}

By the electric-magnetic decomposition of $\Rbf$ into ${E}_{ab}:= \Rbf_{T a T b}$ and ${H}_{ab}:= {}^\ast\Rbf_{T a T b}$, it suffices to prove that
\begin{align} \label{EQreductioncurvatureestimate}
\int\limits_{\Si_{t^\ast}} \vert \D \Rbf_{T \cdot T \cdot} \vert_{\mathbf{{h}}}^2 + \vert \D {}^\ast \Rbf_{T \cdot T \cdot} \vert_{{\mathbf{h}}}^2 \les \lrpar{ \OO_1^{\HH} + \RRt_1^\HH +\OO_1^{\Si} + \RRt_1^\Si + \CMD}^2.
\end{align}

In the following, we prove \eqref{EQreductioncurvatureestimate}. We first bound $\D_{T} \Rbf_{T a T b}$ and $\D_{T} {}^\ast\Rbf_{T a T b}$. On the one hand, 
\begin{align*}
\D_T {E}_{ab} =& \D_T \Rbf_{T a T b} - {n}^{-1} {\nab}^c {n} \lrpar{\Rbf_{caT b} + \Rbf_{T a d b}}, \\
\D_T {H}_{ab} =& \D_T {}^\ast \Rbf_{T a T b} - {n}^{-1} {\nab}^c {n} \lrpar{{}^\ast\Rbf_{caT b} + {}^\ast\Rbf_{T a c b}}.
\end{align*}

On the other hand, by the Bianchi equations \eqref{eq:BianchiEH} we have
\begin{align*}
\D_{T} {E}_{ab} =& \Lieh_T {E}_{ab} -\lrpar{{k}_{ac}{E}_{cb} + {k}_{bc}{E}_{ca}- {k} \cdot {E} \, g_{ab}}, \\
=& -\Curl H_{ab} -(n^{-1} \nab n \wedge H)_{ab} + \half (k \times E)_{ab} -\lrpar{{k}_{ac}{E}_{cb} + {k}_{bc}{E}_{ca}- {k} \cdot {E} \, g_{ab}},\\
\D_{T} {H}_{ab} =& \Lieh_T {H}_{ab} -\lrpar{{k}_{ac}{H}_{cb} + {k}_{bc}{H}_{ca}- {k} \cdot {H} \, g_{ab}}\\
=&  \Curl E_{ab} + (n^{-1} \nab n \wedge E)_{ab} + \half (k \times H)_{ab}-\lrpar{{k}_{ac}{H}_{cb} + {k}_{bc}{H}_{ca}- {k} \cdot {H} \, g_{ab}}.
\end{align*}

By combining the two above, we get that
\begin{align*} \begin{aligned}
&\Vert \D_T \Rbf_{T a T b} \Vert_{L^2(\Si_{t^\ast})} + \Vert \D_T {}^\ast \Rbf_{T a T b}\Vert_{L^2(\Si_{t^\ast})} \\
\les& \Vert {\nab} {E} \Vert_{L^2(\Si_{t^\ast})} + \Vert {\nab} {H} \Vert_{L^2(\Si_{t^\ast})}\\
&+ \lrpar{\Vert {\nab}^2 {n} \Vert_{L^2(\Si_{t^\ast})}+\Vert {\nab} {n} \Vert_{L^2(\Si_{t^\ast})}+\Vert {n} \Vert_{L^2(\Si_{t^\ast})}}  \lrpar{\Vert {\nab} {E} \Vert_{L^2(\Si_{t^\ast})}+\Vert {E} \Vert_{L^2(\Si_{t^\ast})}} \\
&+ \lrpar{\Vert {\nab}^2 {n} \Vert_{L^2(\Si_{t^\ast})}+ \Vert {\nab} {n} \Vert_{L^2(\Si_{t^\ast})}+\Vert {n} \Vert_{L^2(\Si_{t^\ast})}}  \lrpar{\Vert {\nab} {H} \Vert_{L^2(\Si_{t^\ast})}+\Vert {H} \Vert_{L^2(\Si_{t^\ast})}} \\
&+ \lrpar{\Vert {\nab} {k} \Vert_{L^2(\Si_{t^\ast})}+\Vert {k} \Vert_{L^2(\Si_{t^\ast})}}  \lrpar{\Vert {\nab} {E} \Vert_{L^2(\Si_{t^\ast})}+\Vert {E} \Vert_{L^2(\Si_{t^\ast})}} \\
&+ \lrpar{\Vert {\nab} {k} \Vert_{L^2(\Si_{t^\ast})}+\Vert {k} \Vert_{L^2(\Si_{t^\ast})}} \lrpar{\Vert {\nab} {H} \Vert_{L^2(\Si_{t^\ast})}+\Vert {H} \Vert_{L^2(\Si_{t^\ast})}} \\
\les& \OO_1^{\HH} + \RRt_1^\HH +\OO_1^{\Si} + \RRt_1^\Si + \CMD,
\end{aligned} 
\end{align*}
where we used \eqref{EQStatementHigherReg11000}, \eqref{EQStatementHigherReg11} and \eqref{EQgeneralResult}. \\

We next bound $\D_{c} \Rbf_{T a T b}$ and $\D_{c} {}^\ast\Rbf_{T a T b}$. We have that
\begin{align*}
\D_c \Rbf_{T a T b}=& \nab_c E_{ab} + k_{cd} \Rbf_{daT b} + k_{cd} \Rbf_{T a T d}, \\
\D_c {}^\ast \Rbf_{T a T b}=& \nab_c H_{ab} + k_{cd} {}^\ast\Rbf_{daT b} + k_{cd} {}^\ast\Rbf_{T a T d}.
\end{align*}
Hence,
\begin{align*}
&\Vert \D_c \Rbf_{T a T b} \Vert_{L^2(\Si_{t^\ast})} + \Vert \D_c {}^\ast \Rbf_{T a T b}\Vert_{L^2(\Si_{t^\ast})} \\
\les& \OO_1^{\HH} + \RRt_1^\HH +\OO_1^{\Si} + \RRt_1^\Si + \CMD,
\end{align*}
where we used \eqref{EQStatementHigherReg11000} and \eqref{EQStatementHigherReg11}. This finishes the proof of \eqref{EQStatementHigherReg112}. 

\subsubsection{Proof of \eqref{EQStatementHigherReg13}} \label{SECm1estimatesLAPSE} In this section we prove \eqref{EQStatementHigherReg13}, that is,
\begin{align*}
\Vert \nab^3 n \Vert_{L^2(\Si_{t^\ast})} + \Vert \nab^2 T(n) \Vert_{L^2(\Si_{t^\ast})} + \Vert \nab T^2(n) \Vert_{L^2(\Si_{t^\ast})} \les& \OO_1^{\Si} + \RR_1^\Si + \OO_1^{\HH}+ \RRt_1^\HH +\CMD. 
\end{align*}

First, by applying Proposition \ref{PropEllipticEstimatesFORtriangle} to the boundary value problem for $n$ in \eqref{eq:ellEQforN}, that is,
\begin{align} \begin{aligned}
\Delta n =& \,n \vert k\vert_g^2 &\text{ on } \Si_{t^\ast}, \\
n=&\, \nu^{-1} \Om^{-1} &\text{ on } \pr \Si_{t^\ast},
\end{aligned}\label{EQnequationm1recap}\end{align}
and using \eqref{EQStatementHigherReg11000} and \eqref{EQStatementHigherReg11}, we have that for $\CMD>0$ sufficiently small,
\begin{align*}
&\sum\limits_{\vert \a \vert \leq 3} \Vert \nab^{\a} n \Vert_{L^2(\Si_{t^\ast})}\\
 \les& \Vert \nab\triangle n \Vert_{L^2(\Si_{t^\ast})} +\Vert \triangle n \Vert_{L^2(\Si_{t^\ast})} +\Vert \Nd^2 n \Vert_{H^{1/2}(\pr\Si_{t^\ast})} + \Vert \Nd n\Vert_{H^{1/2}(\pr\Si_{t^\ast})} + \Vert n\Vert_{L^2(\pr \Si_{t^\ast})}\\
\les&\Vert \nab \lrpar{n \vert k \vert^2} \Vert_{L^2(\Si_{t^\ast})} +\Vert n\vert k \vert^2 \Vert_{L^2(\Si_{t^\ast})} +\Vert \Nd^2 (\Om^{-1}\nu^{-1}) \Vert_{H^{1/2}(\pr\Si_{t^\ast})} \\
& + \Vert \Nd (\Om^{-1}\nu^{-1})\Vert_{H^{1/2}(\pr\Si_{t^\ast})} + \Vert \Om^{-1}\nu^{-1}\Vert_{L^2(\pr \Si_{t^\ast})} \\
\les& \OO_1^{\Si} + \RR_1^\Si + \OO_1^{\HH}+ \RRt_1^\HH+\CMD.
\end{align*}

Second, we turn to the estimation of $T(n)$ and $TT(n)$. On the one hand, by Lemma \ref{LEMcommutatorTriangleDT} and \eqref{eq:sndvar}, $T(n)$ satisfies on $\Si_{t^\ast}$ the equation
\begin{align} \begin{aligned}
\triangle (T(n)) =& T(\triangle n) +[\triangle, T]n \\ 
=& T(n) \vert k \vert^2 + 2n k \D_T k + 2 k \nab^2 n - 2 n^{-1} \nab n  \nab T(n) - \vert k \vert^2 T(n) \\
&+ 2 n^{-1} k\vert \nab n \vert^2\\
=& 2n k (E - n^{-1} \nab^2 n + k \cdot k)+  2 k \nab^2 n - 2 n^{-1} \nab n  \nab T(n) + 2 n^{-1} k \vert \nab n\vert^2 
\end{aligned} \label{EQDTNequation} \end{align}

On the other hand, we have by Lemma \ref{lemTransportNUalongL} that on $\pr \Si_t$,
\begin{align} \begin{aligned}
T(n) =& \nu L(n) - N(n)\\
=& \nu L(\nu^{-1}\Om^{-1}) -N(n) \\
=& -\frac{1}{\nu \Om} L(\nu) - \frac{1}{\Om^2} L(\Om) - N(n)\\
=&  -\frac{1}{\nu \Om} \Big(-n^{-1}N(n) - \de\Big) - \frac{1}{\Om^2} L(\Om) - N(n),
\end{aligned} \label{EQTnBoundaryValue} \end{align}
which implies by Lemma \ref{PROPtraceEstimate} and \eqref{eq:SmallData} that 
\begin{align} \label{EQboundaryEstimateTN}
\Vert T(n) \Vert_{H^{1/2}(\pr \Si_{t^\ast})}+\Vert \Nd T(n) \Vert_{H^{1/2}(\pr \Si_{t^\ast})} \les \OO_1^{\Si} + \RR_1^\Si + \OO_1^{\HH}+ \RRt_1^\HH+\CMD.
\end{align}

Applying standard elliptic estimates to \eqref{EQDTNequation}, see for example Lemma \ref{PropEllipticEstimatesFORtriangle}, and using  \eqref{EQboundaryEstimateTN}, \eqref{eq:SmallData} and \eqref{eq:BAspacetimeEstimates01}, we get that
\begin{align} \label{EQestimationTN1m1}
\Vert \nab^2 T(n) \Vert_{L^2(\Si_t)}+ \Vert \nab T(n) \Vert_{L^2(\Si_t)} \les \OO_1^{\Si} + \RR_1^\Si + \OO_1^{\HH}+ \RRt_1^\HH+\CMD.
\end{align}

The proof of the control of $\nab T^2(n)$ follows by commuting \eqref{EQnequationm1recap} once more with $T$, applying standard elliptic estimates and using Lemma \ref{lemTransportNUalongL} to control the boundary value of  $T^2(n)$; we leave details to the reader, and see also Appendix E in \cite{KRS}. This finishes our discussion of the proof of \eqref{EQStatementHigherReg13}.

\subsection{Higher regularity estimates for $m\geq2$} \label{SECm2estimatesOUTLINE} In this section we outline the proof of higher regularity estimates for $m\geq2$. The proof is based on an induction in $m\geq1$. The base case $m=1$ is proved in the previous sections. In the following we discuss the induction step $m \to m+1$.\\

We recall the geometric setup. Let $(\MM,\g)$ be a vacuum spacetime whose past is bounded by a compact spacelike maximal hypersurface $\Si \simeq \ol{B(0,1)}$ and the outgoing null hypersurface $\HH$ emanating from $\pr \Si$. For some real $1<t^\ast_0 \leq 2$, assume there exists a foliation $(\Si_t)_{1\leq t \leq t^\ast_0}$ of spacelike maximal hypersurfaces given as level sets of a time function $t$ with $t(\Si)=1$ and such that $\pr \Si_t = S_t$, where $(S_t)_{1\leq t \leq t_0^\ast}$ denotes the canonical foliation on $\HH$. Assume that for $\varep>0$ it holds for $1\leq t \leq t_0^\ast$ that
\begin{align*} \begin{aligned}
\Vert \RRRic \Vert_{L^\infty_tL^2(\Si_t)} \les& \,\varep, \\
\Vert k \Vert_{L^\infty_tL^2(\Si_t)} + \Vert \nab k \Vert_{L^\infty_tL^2(\Si_t)} + \Vert k \Vert_{L^\infty_tL^2(\pr \Si_t)} \les&\, \varep.
\end{aligned} 
\end{align*}

Assume as \emph{induction hypothesis} that for an integer $m\geq1$, we have that for $1 \leq t \leq t^\ast_0$,
\begin{align}\begin{aligned}
&\sum\limits_{\vert \a \vert \leq m} \Vert \D^{\a} \Rbf \Vert_{L^2(\Si_{t})}+\sum\limits_{\vert \a \vert \leq m+1} \Vert \D^{\a} \pi \Vert_{L^2(\Si_{t})} \\
&+\sum\limits_{\vert \a \vert \leq m} \Vert \nab^{\a} \Ric \Vert_{L^2(\Si_{t})}  + \Vert \Nd^{\a}\Nd \nu \Vert_{H^{1/2}(\pr \Si_t)}\\
\les& \, C(\OO_m^\HH, \RR_{m}^\HH, \OO_m^\Si, \RR_m^\Si, m).
\end{aligned} \label{EQinductionHypothesisMGEQ2}
\end{align}

In the following we prove the \emph{induction step}, that is, we show that for $0 \leq t \leq t^\ast_0$,
\begin{align} \begin{aligned}
&\sum\limits_{\vert \a \vert \leq m+1} \Vert \D^{\a} \Rbf \Vert_{L^2(\Si_{t})}+\sum\limits_{\vert \a \vert \leq m+2} \Vert \D^{\a} \pi \Vert_{L^2(\Si_{t})} \\
&+\sum\limits_{\vert \a \vert \leq m+1} \Vert \nab^{\a} \Ric \Vert_{L^2(\Si_{t})}  + \Vert \Nd^{\a}\Nd \nu \Vert_{H^{1/2}(\pr \Si_t)}\\
\les&\, C(\OO_{m+1}^\HH, \RR_{m+1}^\HH, \OO_{m+1}^\Si, \RR_{m+1}^\Si, m+1).
\end{aligned} \label{EQm2mand1Estimate1} \end{align}

We proceed as follows.
\begin{enumerate}

\item In Section \ref{SECm2EllipticCURVATUREestimates}, we prove that for each $0 \leq t \leq t^\ast_0$,
\begin{align}\begin{aligned}
&\sum\limits_{\vert \a \vert \leq m+1} \Vert {\nab}^{\a} {E} \Vert_{L^2({\Si}_{t})}+ \Vert {\nab}^{\a} {H} \Vert_{L^2({\Si}_{t})} \\
 \les& \, \Vert  \Lieh_T^{m+1} \Rbf \Vert_{L^2({\Si}_{t})} +C(\OO_{m+1}^\HH, \RR_{m+1}^\HH,\OO_{m+1}^\Si, \RR_{m+1}^\Si, {m+1}).
\end{aligned} \label{EQM2EllipticCurvatureEstimates1} \end{align}
The proof of \eqref{EQM2EllipticCurvatureEstimates1} is based on the fact that $ E$ and $ H$ satisfy a $3$-dimensional Hodge system on ${\Si}_t$ by the Bianchi equations.

\item In Section \ref{SECm2CurvatureEstimates}, we prove that for $0\leq t \leq t^\ast_0$,
\begin{align} \label{EQm2LIEHTRBFestimate1}
\Vert \Lieh_T^{m+1} \Rbf \Vert_{L^2({\Si}_t)} \les& \, C(\OO_{m+1}^\HH, \RR_{m+1}^\HH, \OO_{m+1}^\Si, \RR_{m+1}^\Si, m).
\end{align}
The proof of \eqref{EQm2LIEHTRBFestimate1} is based on an energy estimate for the curvature using the Bel-Robinson tensor together with the classical Gr\"onwall lemma.
\begin{remark} Contrary to the case $m=1$ where the error integral in the Bel-Robinson energy estimate needed to be bounded by a trilinear estimate, in the case $m\geq2$ we can argue solely by the classical Gr\"onwall lemma and the estimates for $m=1$. In particular, we do not use the bounded $L^2$ curvature theorem.
\end{remark}

\item In Section \ref{SECm2KEstimates}, we show that for $1 \leq t \leq t^\ast_0$,
\begin{align} \begin{aligned} 
&\sum\limits_{\vert \a \vert \leq m+1} \Vert \nab^{\a} \nab k \Vert_{L^2(\Si_{t^\ast})} + \sum\limits_{\vert \a \vert \leq m+1} \Vert \Nd^{\a} \Nd \nu \Vert_{L^2(\pr \Si_{t^\ast})}\\
 \les& \, \OO_{m+1}^\HH+ \RR_{m+1}^\HH + \OO_{m+1}^\Si +\RR_{m+1}^\Si + C_m \CMD.
\end{aligned}\label{EQm2KestimatesElliptic} \end{align}
The proof of \eqref{EQm2KestimatesElliptic} is based on standard higher regularity estimates for the Hodge system satisfied by $k$ together with the special structure of the boundary term.

\item In Section \ref{SECm2ConclusionCurvature}, we conclude the proof of \eqref{EQm2mand1Estimate1}.

\end{enumerate}

\subsubsection{Elliptic curvature estimates on ${\Si}_{t}$: The proof of \eqref{EQM2EllipticCurvatureEstimates1}} \label{SECm2EllipticCURVATUREestimates}

In this section we prove \eqref{EQM2EllipticCurvatureEstimates1}, that is,
\begin{align*}
&\sum\limits_{\vert \a \vert \leq m+1} \Vert {\nab}^{\a} {E} \Vert_{L^2({\Si}_{t})}+ \Vert {\nab}^{\a} {H} \Vert_{L^2({\Si}_{t})} \\
 \les&  \Vert  \Lieh_T^{m+1} \Rbf \Vert_{L^2({\Si}_{t})} + C(\OO_{m+1}^\HH, \RR_{m+1}^\HH, \OO_{m+1}^\Si , \RR_{m+1}^\Si , {m+1}).
\end{align*}

The idea is to apply elliptic estimates to the Hodge systems satisfied by $E(\Lieh_T^{i} \Rbf)$ and $H(\Lieh_T^{i} \Rbf)$ on ${\Si}_{t}$ for $0\leq i \leq m$. More specifically, denoting
\begin{align*}
J\lrpar{\Lieh_T^{i} \Rbf}_{\be \ga \de} := \D^\a \lrpar{\Lieh_T^{i} \Rbf}_{\a \be \ga \de},
\end{align*}
it holds by Proposition \ref{prop:MaxwellsEq1} with \eqref{EQTrelationEH}, see also \eqref{eq:BianchiEH2}, that
\begin{align} \begin{aligned}
\Div \,  E\lrpar{\Lieh_T^{i} \Rbf}_a =& +\lrpar{ k \wedge  H\lrpar{\Lieh_T^{i} \Rbf}}_a + J\lrpar{\Lieh_T^{i} \Rbf}_{T aT}, \\
\Curl \,  E\lrpar{\Lieh_T^{i} \Rbf}_{ab} =& + H\lrpar{\Lieh_T^{i+1} \Rbf}_{ab} - 3 \lrpar{{n}^{-1} \nab  n \wedge  E\lrpar{\Lieh_T^{i} \Rbf}}_{ab} \\
&- \frac{3}{2} \lrpar{ k \times  H\lrpar{\Lieh_T^{i} \Rbf}}_{ab} - J^\ast\lrpar{\Lieh_T^{i} \Rbf}_{aT b}, \\
\Div \,  H\lrpar{\Lieh_T^{i} \Rbf}_a =& -\lrpar{ k \wedge  E\lrpar{\Lieh_T^{i} \Rbf}}_a + J^\ast\lrpar{\Lieh_T^{i} \Rbf}_{T aT}, \\
\Curl \,  H\lrpar{\Lieh_T^{i} \Rbf}_{ab} =&- E \lrpar{\Lieh_T^{i+1} \Rbf}_{ab}  -3 \lrpar{n^{-1} \nab n \wedge  H\lrpar{\Lieh_T^{i} \Rbf}}_{ab} \\
&+ \frac{3}{2} \lrpar{ k \times  E\lrpar{\Lieh_T^{i} \Rbf}}_{ab} - J\lrpar{\Lieh_T^{i} \Rbf}_{aT b}.
\end{aligned} \label{eq:BianchiEH22222} 
\end{align}

By standard higher regularity \emph{elliptic estimates} applied to the above Hodge system (see the methods developed in Sections \ref{SUBSECgeneralHodgeEstimates}, \ref{SECehEnergyEstimateNotation} and \ref{SECcurv1112}), we have for each $0\leq i \leq m$, 
\begin{align} \begin{aligned}
&\norm{ {\nab}^{m+1-i} \, \lrpar{\Lieh_T^{i} \Rbf} }_{L^2({\Si}_t)}\\
\les&\norm{ {\nab}^{m+1-i} \,  E\lrpar{\Lieh_T^{i} \Rbf} }_{L^2({\Si}_t)}+\norm{ {\nab}^{m+1-i} \,  H\lrpar{\Lieh_T^{i} \Rbf} }_{L^2(\wt{\Si}_t)} \\
&+ C(\OO_{m+1}^\HH, \RR_{m+1}^\HH, \OO_{m+1}^\Si, \RR_{m+1}^\Si, {m+1}) \\
\les& \norm{ {\nab}^{m-i} \Lieh_T^{i+1} \Rbf }_{L^2({\Si}_t)} + \Vert {\nab}^{m-i} J(\Lieh^i_T \Rbf) \Vert_{L^2(\Si_t)} + \Vert {\nab}^{m-i} J^\ast(\Lieh^i_T \Rbf) \Vert_{L^2(\Si_t)} \\
&+ C(\OO_{m+1}^\HH, \RR_{m+1}^\HH, \OO_{m+1}^\Si, \RR_{m+1}^\Si, {m+1}) \\
\les& \norm{ {\nab}^{m-i} \Lieh_T^{i+1} \Rbf }_{L^2({\Si}_t)} + C(\OO_{m+1}^\HH, \RR_{m+1}^\HH, \OO_{m+1}^\Si, \RR_{m+1}^\Si, {m+1}),
\end{aligned} \label{EQsuccessiveEllipticEstimatem2} \end{align}
where we used \eqref{EQinductionHypothesisMGEQ2} and directly bounded the boundary integrals appearing in the elliptic estimates by initial data norms, see also Section \ref{SECcurv1112}. Furthermore, in \eqref{EQsuccessiveEllipticEstimatem2} we estimated the currents $J\lrpar{\Lieh_T^{i}\Rbf}$ on the right-hand side of \eqref{eq:BianchiEH22222} as product terms by standard product estimates with \eqref{EQinductionHypothesisMGEQ2}. Indeed, by Proposition 7.1.2 in \cite{ChrKl93}, $J\lrpar{\Lieh_T^{i}\Rbf}$, for $i\geq1$, can be expressed as follows,
\begin{align} \begin{aligned}
J\lrpar{\Lieh_T^{i}\Rbf} :=&\D^\a \lrpar{\Lieh_T^{i}\Rbf}_{\a \be \ga \de} \\
=& \widehat{\Lie}_T J\lrpar{\Lieh_T^{i-1}\Rbf}_{\be \ga \de} + \half {\pi}^{\mu \nu} \D_\nu \lrpar{\Lieh_T^{i-1}\Rbf}_{\mu \be \ga \de}\\
& + \half \D^\a {\pi}_{\a \la} \lrpar{\Lieh_T^{i-1}\Rbf}^\la_{\,\,\, \be \ga \de} + \half (\D_\be {\pi}_{\a \la} - \D_\la {\pi}_{\a \be}) \lrpar{\Lieh_T^{i-1}\Rbf}^{\a \la}_{\,\,\, \,\,\, \ga \de} \\
&+ \half (\D_\ga {\pi}_{\a \la} - \D_\la {\pi}_{\a \ga}) \lrpar{\Lieh_T^{i-1}\Rbf}^{\a \,\,\, \la}_{\,\,\, \be \,\,\, \de} \\
&+ \half(\D_\de {\pi}_{\a \la} - \D_\la {\pi}_{\a \de})  \lrpar{\Lieh_T^{i-1}\Rbf}^{\a \,\,\, \,\,\, \la}_{\,\,\, \be \ga  \,\,\,},
\end{aligned}\label{EQrecursiveJcurrentformula} \end{align}
where 
\begin{align*}
\widehat{\Lie}_T J\lrpar{\Lieh_T^{i-1}\Rbf}_{\be \ga \de} :=& \Lie_T J\lrpar{\Lieh_T^{i-1}\Rbf}_{\be \ga \de} -\half  {\pi}_\be^{\,\,\, \mu} J\lrpar{\Lieh_T^{i-1}\Rbf}_{\mu \ga \de} \\
&-\half {\pi}_\ga^{\,\,\,\mu} J\lrpar{\Lieh_T^{i-1}\Rbf}_{\be \mu \de} -\half {\pi}_\de^{\,\,\, \mu} J\lrpar{\Lieh_T^{i-1}\Rbf}_{\be \ga \mu}.
\end{align*}
By the above recursive relation together with the fact that $J(\Rbf)=0$ by the Bianchi equations, it follows that $J\lrpar{\Lieh_T^{i}\Rbf}$ consists of product terms.\\

Returning to the proof of \eqref{EQM2EllipticCurvatureEstimates1}, by successive combination of \eqref{EQsuccessiveEllipticEstimatem2} for $0\leq i \leq m$, we get that
\begin{align} \begin{aligned}
&\sum\limits_{\vert \a \vert + \vert \be \vert \leq m+1} \norm{{\nab}^{\a} \Lieh_T^{\be} \Rbf}_{L^2({\Si}_t)}\\
 \les& \norm{  \Lieh_T^{m+1} \Rbf }_{L^2({\Si}_{t})} + C(\OO_{m+1}^\HH, \RR_{m+1}^\HH,
 \OO_{m+1}^\Si, \RR_{m+1}^\Si,{m+1}),
\end{aligned}\label{EQliehTRbfm2estimateselliptic}\end{align}
which implies in particular that
\begin{align*}
&\sum\limits_{\vert \a \vert \leq m+1} \Vert {\nab}^{\a} {E} \Vert_{L^2({\Si}_{t})}+ \Vert {\nab}^{\a} {H} \Vert_{L^2({\Si}_{t})} \\
\les& \sum\limits_{\vert \a \vert \leq m+1} \norm{{\nab}^{\a} \Rbf}_{L^2({\Si}_t)} + C(\OO_{m+1}^\HH, \RR_{m+1}^\HH, \OO_{m+1}^\Si, \RR_{m+1}^\Si,{m+1}) \\
 \les&  \Vert  \Lieh_T^{m+1} \Rbf \Vert_{L^2({\Si}_{t})} + C(\OO_{m+1}^\HH, \RR_{m+1}^\HH, \OO_{m+1}^\Si, \RR_{m+1}^\Si,{m+1}).
\end{align*} 
This finishes the proof of \eqref{EQM2EllipticCurvatureEstimates1}.

\subsubsection{Energy estimate for the curvature: Proof of \eqref{EQm2LIEHTRBFestimate1}} \label{SECm2CurvatureEstimates} In this section we prove \eqref{EQm2LIEHTRBFestimate1}, that is, for $0\leq \ttt \leq t^\ast$,
\begin{align*}
\Vert \Lieh_T^{m+1} \Rbf \Vert_{L^2({\Si}_t} \leq& \, C\lrpar{\OO_{m+1}^\HH, \RR_{m+1}^\HH, \OO_{m+1}^\Si,\RR_{m+1}^\Si,m}.
\end{align*}

The idea is to apply the integral identity \eqref{lemma:IntegralidentityWeyl1} to the Weyl tensor $\Lieh_{T}^{m+1} \Rbf$ with multiplier field $T$. This yields for $0 \leq t \leq t^\ast_0$,
\begin{align} \begin{aligned}
\Vert \Lieh_{T}^{m+1} \R \Vert^2_{L^2({\Si}_{t})}
\les& \int\limits_{\Si_1} Q(\Lieh_{T}^{m+1} \R)_{TTTT}+ \int\limits_{\HH} Q(\Lieh_{T}^{m+1} \R)_{TTTL} \\
&- \underbrace{\int\limits_{\MM_{t}} \frac{3}{2} Q(\Lieh_{T}^{m+1} \R)_{\a \be T T}{\pi}^{\a\be}}_{:=\EE_1} - \underbrace{\int\limits_{\MM_{t}} \D^\a Q(\Lieh_{T}^{m+1} \R)_{\a T T T}}_{:=\EE_2},
\end{aligned} \label{EQintegralm2identity} \end{align}
where $\MM_{t}$ denotes the past of ${\Si}_t$ in $\MM_{t^\ast_0}$. In the following, we first bound $\EE_1$ and $\EE_2$. \\ 

\textbf{Estimation of $\EE_1$.} By \eqref{EQrelationQothervectors1},
\begin{align*}
\EE_1 :=& \int\limits_{\MM_{t}} \frac{3}{2} Q(\Lieh_{T}^{m+1} \R)_{\a \be T T}{\pi}^{\a\be} \\
\les& \norm{{\pi} }_{L^\infty(\MM_{t^\ast_0})} \int\limits_{\MM_{t}} \vert Q(\Lieh_{T}^{m+1} \R)_{\a \be T T} \vert\\
\les&  \norm{{\pi} }_{L^\infty(\MM_{t^\ast_0})} \lrpar{1+ \norm{{n}-1 }_{L^\infty(\MM_{t^\ast_0})}} \int\limits_0^{t} \norm{ \Lieh_{T}^{m+1} \R }_{L^2({\Si}_{t'})}^2 dt'.
\end{align*}
This finishes the estimation of $\EE_1$. \\

\textbf{Estimation of $\EE_2$.} Using that
\begin{align*}
\D^\a Q\lrpar{\Lieh_{T}^{m+1} \Rbf}_{\a T T T} =& 2 \lrpar{\Lieh_{T}^{m+1} \R}_{T \,\,\, T}^{\,\,\,\mu \,\,\, \nu} J\lrpar{\Lieh_T^{m+1} \R}_{\mu T \nu} \\
&+ 2 {}^\ast\lrpar{\Lieh_{T}^{m+1} \R}_{T \,\,\, T}^{\,\,\,\mu \,\,\, \nu} J^\ast\lrpar{\Lieh_{T}^{m+1} \R}_{\mu T \nu},
\end{align*}
we can estimate
\begin{align*}
\EE_2 :=& \int\limits_{\MM_{t}} \D^\a Q(\Lieh_{T}^{m+1} \R)_{\a T T T} \\
\les& \int\limits_0^{t} \norm{ \Lieh_{T}^{m+1} \R }_{L^2({\Si}_{t'})} \norm{ J\lrpar{\Lieh_{T}^{m+1} \R} }_{L^2({\Si}_{t'})} dt' \\
\les& \lrpar{\sum\limits_{\vert \a \vert \leq m} \Vert \D^{\a} \Rbf \Vert_{L^\infty_t L^2(\Si_{t})}+ \sum\limits_{\vert \a \vert \leq m+1}\Vert \D^{\a} \pi \Vert_{L^\infty_tL^2(\Si_{t})}+1} \\
&\times \int\limits_0^{t} \norm{ \Lieh_{T}^{m+1} \R }^2_{L^2({\Si}_{t'})} dt' +C(\OO_{m+1}^\HH,\RR_{m+1}^\HH,\OO_{m+1}^\Si,\RR_{m+1}^\Si,m+1).
\end{align*}

where we used that by \eqref{EQrecursiveJcurrentformula}, $J\lrpar{\Lieh_{T}^{m+1} \R}$ consists of product terms which can, together with the elliptic estimates \eqref{EQliehTRbfm2estimateselliptic}, be estimated as follows,
\begin{align*}
&\norm{ J\lrpar{\Lieh_{T}^{m+1} \R} }_{L^2({\Si}_{t'})} \\
\les& \norm{ \Lieh_{T}^{m+1} \R }_{L^2({\Si}_{t'} )} \lrpar{\sum\limits_{\vert \a \vert \leq 2}\Vert \D^{\a} \pi \Vert_{L^2({\Si}_{t'})}} \\
&+ \Vert \Rbf \Vert_{L^\infty({\Si}_{t'})} \lrpar{\sum\limits_{\vert \a \vert \leq m+1} \Vert \D^{\a}\pi \Vert_{L^2({\Si}_{t'})} +\sum\limits_{\vert \a \vert \leq m} \Vert \D^{\a} \Rbf \Vert_{ L^2(\Si_{t'})}} \\
&+C(\OO_{m+1}^\HH,\RR_{m+1}^\HH,\OO_{m+1}^\Si,\RR_{m+1}^\Si,m+1) \\
\les& \norm{ \Lieh_{T}^{m+1} \R }_{L^2({\Si}_{t'})} \lrpar{\sum\limits_{\vert \a \vert \leq m+1} \Vert \D^{\a}\pi \Vert_{L^\infty L^2({\Si}_{t})}+ \sum\limits_{\vert \a \vert \leq m} \Vert \D^{\a} \Rbf \Vert_{L^\infty_t L^2(\Si_{t})}} \\
&+C(\OO_{m+1}^\HH,\RR_{m+1}^\HH,\OO_{m+1}^\Si,\RR_{m+1}^\Si,m+1).
\end{align*}

Plugging the above estimates for $\EE_1$ and $\EE_2$ into \eqref{EQintegralm2identity}, we get that for $0 \leq t \leq t^\ast_0$,
\begin{align*}
&\Vert \Lieh_{T}^{m+1} \R \Vert^2_{L^2({\Si}_{t})} \\
\les& \int\limits_{\Si_1} Q(\Lieh_{T}^{m+1} \R)_{T T T  T}+ \int\limits_{\HH } Q(\Lieh_{T}^{m+1} \R)_{T T T L}+ C(\OO_{m+1}^\HH,\RR_{m+1}^\HH,\OO_{m+1}^\Si,\RR_{m+1}^\Si,m+1)\\
&+ \lrpar{\sum\limits_{\vert \a \vert \leq m+1} \Vert \D^{\a}\pi \Vert_{L^\infty L^2({\Si}_{t})}+ \sum\limits_{\vert \a \vert \leq m} \Vert \D^{\a} \Rbf \Vert_{L^\infty_t L^2(\Si_{t})}+1}  \int\limits_0^{t} \norm{ \Lieh_{T}^{m+1} \R }^2_{L^2({\Si}_{t'} )} dt'.
\end{align*}

Therefore Gr\"onwall's lemma yields that for $0 \leq t \leq t^\ast_0$,
\begin{align*}
&\Vert \Lieh_{T}^{m+1} \R \Vert^2_{L^\infty_{t}L^2({\Si}_{t})}\\
\les& \underbrace{\int\limits_{\Si_1} Q(\Lieh_{T}^{m+1} \R)_{T T T  T}}_{:= \II_1}+ \underbrace{\int\limits_{\HH} Q(\Lieh_{T}^{m+1} \R)_{T T T L}}_{:= \II_2} +C(\OO_{m+1}^\HH,\RR_{m+1}^\HH,\OO_{m+1}^\Si,\RR_{m+1}^\Si,m).
\end{align*}
The boundary integrals $\II_1$ and $\II_2$ are estimated similarly as in Sections \ref{sec:CurvatureEstimatesForBackground} and \ref{SECcurv111} by
\begin{align*}
\II_1 + \II_2 \les C(\OO_{m+1}^\HH,\RR_{m+1}^\HH,\OO_{m+1}^\Si,\RR_{m+1}^\Si,m).
\end{align*}
This finishes the outline of the proof of \eqref{EQM2EllipticCurvatureEstimates1}.

\subsubsection{Elliptic estimate for $k$ on $\Si_{t^\ast}$: Proof of \eqref{EQm2KestimatesElliptic}} \label{SECm2KEstimates} In this section, we prove \eqref{EQm2KestimatesElliptic}, that is,
\begin{align*}
\sum\limits_{\vert \a \vert \leq m+1} \Vert \nab^{\a} \nab k \Vert_{L^2(\Si_{t})} + \Vert \Nd^{\a} \nu \Vert_{L^2(\pr \Si_{t})}\les& C(\OO_{m+1}^\HH,\RR_{m+1}^\HH,\OO_{m+1}^\Si,\RR_{m+1}^\Si,m+1).
\end{align*}

The idea is to use higher regularity elliptic estimates for the Hodge system satisfied by $k$, see Sections \ref{SUBSECgeneralHodgeEstimates} and \ref{SECellipticEstimatesKnotation}, which yields
\begin{align*}
&\sum\limits_{\vert \a \vert \leq m+1} \Vert \nab^{\a} \nab k \Vert^2_{L^2(\Si_{t^\ast})} +  \BB \les C(\OO_{m+1}^\HH,\RR_{m+1}^\HH,\OO_{m+1}^\Si,\RR_{m+1}^\Si,m+1),
\end{align*}
where we used \eqref{EQM2EllipticCurvatureEstimates1} \eqref{EQm2LIEHTRBFestimate1}, and $\BB$ is a boundary integral whose special structure allows to bound it from below by
\begin{align} \label{EQcoercivityBoundary}
\Vert \Nd^{m+1} \nu \Vert^2_{L^2(\pr \Si_{t^\ast})} \les \BB.
\end{align}
Indeed, the analysis of the structure of the boundary integral $\BB$ is similar to the analysis done in Section \ref{SECellipticEstimateM1}. The only difference in the analysis of $m\geq2$ is that normal derivatives in the boundary integral are systematically reduced to $\pr \Si_t$-tangential derivatives by applying \eqref{eq:FirstOfDivEq} and \eqref{eq:NepNabDeRelation}, that is,
\begin{align*}
N(\de)  =& - \Divd \ep-2a^{-1} \Nd a \cdot \ep + \eta \cdot \Th - \de \tr \Th, \\
\Nd_N \ep_B  =& - (\Divd \eta)_B-a^{-1} \Nd_C a \eta^{C}_{\,\,\, B} + a^{-1} \de \Nd_B a  - \tr \Th \ep_B - \Th_{BC} \ep^C,
\end{align*}
an \emph{even number of times}. This is due to the fact that the integrand of the boundary integral is a contraction of two tensors. As a consequence, \emph{the sign is conserved} and the coercivity, that is, the constant of \eqref{EQcoercivityBoundary} is bounded from below by a positive constant. This finishes our discussion of \eqref{EQm2KestimatesElliptic}.

\subsubsection{Conclusion of \eqref{EQm2mand1Estimate1}} \label{SECm2ConclusionCurvature} In this section, we conclude the proof of \eqref{EQm2mand1Estimate1}, that is,
\begin{align*} \begin{aligned}
&\sum\limits_{\vert \a \vert \leq m+1} \Vert \D^{\a} \Rbf \Vert_{L^2(\Si_{t})}+ \sum\limits_{\vert \a \vert \leq m+2} \Vert \D^{\a} \pi \Vert_{L^2(\Si_{t})} \\
&+\sum\limits_{\vert \a \vert \leq m+1} \Vert \nab^{\a} \Ric \Vert_{L^2(\Si_{t})}  + \Vert \Nd^{\a}\Nd \nu \Vert_{H^{1/2}(\pr \Si_t)}\\
\les&C(\OO_{m+1}^\HH, \RR_{m+1}^\HH, \OO_{m+1}^\Si, \RR_{m+1}^\Si, m+1).
\end{aligned}  \end{align*}

Indeed, first, the estimate
\begin{align*}
\sum\limits_{\vert \a \vert \leq m+1} \Vert \D^{\a} \Rbf \Vert_{L^2(\Si_{t})} \les&C(\OO_{m+1}^\HH, \RR_{m+1}^\HH, \OO_{m+1}^\Si, \RR_{m+1}^\Si, m+1)
\end{align*}
follows directly from the previous estimates for $\nab^{\a} \Rbf$, $\vert \a \vert=m+1$, and $\Lieh_T^{m+1} \Rbf$, that is, \eqref{EQM2EllipticCurvatureEstimates1} and \eqref{EQm2LIEHTRBFestimate1}; see also Section \ref{SECTestimateEHm1}.\\

Second, the estimate
\begin{align*}
\sum\limits_{\vert \a \vert \leq m+2} \Vert \D^{\a} \pi \Vert_{L^2(\Si_{t})} \les C(\OO_{m+1}^\HH, \RR_{m+1}^\HH, \OO_{m+1}^\Si, \RR_{m+1}^\Si, m+1)
\end{align*}
follows by the previous estimates for $\nab^{\a} E$ and $\nab^{\a} H$, $\vert \a \vert=m+1$, and $\nab^{\a} k$, $\vert \a \vert =m+2$, on $\Si_t$, that is, \eqref{EQM2EllipticCurvatureEstimates1} and \eqref{EQm2KestimatesElliptic}, by consecutively applying higher regularity elliptic estimates (see Proposition \ref{PropEllipticEstimatesFORtriangle}) to the boundary value problems satisfied by $T^{i}(n)$, $0\leq i \leq m+1$ (see also Section \ref{SECm1estimatesLAPSE}) and applying $T$-derivatives to the second variation equation \eqref{eq:sndvar} (see Lemma \ref{LemmaDTcontrolk}). We leave details to the reader.\\

Third, the estimate
\begin{align*}
\sum\limits_{\vert \a \vert \leq m+1} \Vert \nab^{\a} \Ric \Vert_{L^2(\Si_{t})} \les C(\OO_{m+1}^\HH, \RR_{m+1}^\HH, \OO_{m+1}^\Si, \RR_{m+1}^\Si, m+1)
\end{align*}
follows by the previous estimates for $\nab^{\a} E$ and $\nab^{\a} H$, $\vert \a \vert=m+1$, and $\nab^{\a} k$, $\vert \a \vert =m+2$ on $\Si_t$, that is, \eqref{EQM2EllipticCurvatureEstimates1} and \eqref{EQm2KestimatesElliptic}, by using the traced Gauss equation \eqref{eq:RicE}.\\

Fourth, the estimate
\begin{align*}
\Vert \Nd^{\a}\Nd \nu \Vert_{H^{1/2}(\pr \Si_t)}\les&C(\OO_{m+1}^\HH, \RR_{m+1}^\HH, \OO_{m+1}^\Si, \RR_{m+1}^\Si, m+1)
\end{align*}
follows by using the slope equation \eqref{eq:slope}, that is,
\begin{align*}
\nut^{-1}\Nd_A\nut = -\ep_A + \zet_A \text{ on } \pr \Si_t.
\end{align*}
together with \eqref{EQM2EllipticCurvatureEstimates1}, \eqref{EQm2LIEHTRBFestimate1}, \eqref{EQm2KestimatesElliptic} and Lemma \ref{Lemma:TraceEstimateH1toL4bdry}. This finishes our discussion of the proof of \eqref{EQm2mand1Estimate1}.

\section{Classical local existence of spacetime with maximal foliation} \label{Section:LocalExistence} In this section, we prove Proposition \ref{thm:LocalExistenceMixedMaximal}. First, we have the next classical local existence result for the spacelike-characteristic Cauchy problem of general relativity.

\begin{theorem}[Classical local existence] \label{thm:ClassicalExistenceMixedProblem} 
Let there be given smooth initial data for the spacelike-characteristic Cauchy problem on a maximal hypersurface with boundary $\Si$ and the outgoing null hypersurface $\HH$ emanating from $\pr \Si$. Then its unique maximal smooth globally hyperbolic future development $(\MM,\g)$ has past boundary $\Si \cup \HH$. \end{theorem}

\begin{proof} The proof follows from the literature results \cite{BruhatExistence} and \cite{Rendall}, see also \cite{LukChar}. Indeed, on the one hand, by classical local existence for the spacelike Cauchy problem \cite{BruhatExistence} with data on $\Si$, the past boundary of the intersection of $(\MM,\g)$ with the future domain of dependence $\DD(\Si)$ of $\Si$ equals $\Si$.\\

On the other hand, considering data on $\pr \DD(\Si)$ together with the data on $\HH$, it follows by the work of Rendall \cite{Rendall} that the past boundary of $(\MM,\g)$ intersected with the future of $\HH$ contains $\HH$. We remark that the data at $\pr \Si$ is well-posed by the assumption that given initial data for the spacelike-characteristic Cauchy problem satisfies the necessary compatibility conditions on $\pr \Si$, see Section 7.6 in \cite{ChruscielPaetz2} for details. This finishes the proof of Theorem \ref{thm:ClassicalExistenceMixedProblem}. \end{proof}

In the rest of this section, we prove Proposition \ref{thm:LocalExistenceMixedMaximal}. Let $\Si$ be a compact maximal hypersurface with boundary and let $\HH$ be the outgoing null hypersurface emanating from $\pr \Si$. Let $(S_v)_{v\geq1}$ be a foliation on $\HH$ by spacelike $2$-spheres $S_v$. Let $(\MM, \g)$ denote the maximal globally hyperbolic future development of the initial data on $\Si$ and $\HH$. \\

In the following, we construct a local time function $t$ in the future of $\Si$ such that $t\vert_{\Si}=1$ and for small values $t\geq1$,
\begin{itemize}
\item we have $t = v$ on $\HH$, 
\item the level sets $\Si_t$ of $t$ are maximal spacelike hypersurfaces in $(\MM, \g)$.
\end{itemize}

The main ingredient for this construction is the work \cite{PerturbationBruhat} of Bruhat which shows that \emph{in shift-free background coordinates} the linearisation of the mean curvature functional is surjective, see Theorem \ref{theorem:BruhatPerturbation} below.\\

Our construction of the time function is thus split into three steps.
\begin{enumerate}
\item Construction of local shift-free background coordinates $(x^\mu)_{\mu=0,1,2,3}$ in the future of $\Si$,
\item Construction of a family of maximal spacelike hypersurfaces on $\MM$, 
\item Proof that the above family of spacelike maximal hypersurfaces can be written as level sets of a smooth time function $t$ on $\MM$ satisfying $\{ t=1\}=\Si$ and $t(S_v)=v$ for $v\geq1$ small.
\end{enumerate}
 
{\bf Step 1. Construction of shift-free background coordinates.} First, define a scalar function $x^0$ on $\Si \cup \HH$ by
\begin{align*}
x^0 = 1 \text{ on } \Si, \,\, x^0 = v \text{ on } \HH,
\end{align*}
where $v$ denotes the parameter of the given foliation $(S_v)_{v\geq1}$ on $\HH$. By the Whitney extension theorem (see its similar application in \cite{Rendall} and references therein) there exists a smooth extension of $x^0$ into $\MM$ such that its level sets $(\Si_{x^0})_{x^0 \geq1}$ are a local foliation of the future of $\Si$ in $(\MM,\g)$. Let $e_0$ be the future-pointing timelike unit normal to $\Si_{x^0}$. \\

Second, let $(\overline{x}^i)_{i=1,2,3}$ be given coordinates on $\Si$. We extend them as local coordinates $(x^i)_{i=1,2,3}$ onto $\MM$ as follows. First let 
$$x^i := f_i \text{ on } \HH,$$ 
where $(f_i)_{i=1,2,3} \in C^\infty(\HH)$ are smooth, increasing functions chosen below. Then define $(x^i)_{i=1,2,3}$ on $\MM$ as solution to
\begin{align*}
e_0(x^i)=0 \text{ on } \MM, \,\, x^i = \overline{x}^i \text{ on } \Si_0, \,\, x^i = f_i \text{ on } \HH.
\end{align*}
The smoothness of $x^i$ in $\MM$ requires compatibility conditions on $(f_i)_{i=1,2,3}$ and their derivatives at $\pr \Si = \HH \cap \Si$. By the Whitney extension theorem (see its similar application in \cite{Rendall} and references therein), there exists a choice $(f_i)_{i=1,2,3}$ such that these compability conditions are satisfied. \\

By construction, $(x^\mu)_{\mu=0,1,2,3}$ locally form a coordinate system on the future of $\Si$. Moreover, the coordinates are by construction shift-free, that is, $e_0(x^i)=0$ for $i=1,2,3$.\\

{\bf Step 2: Construction of a foliation of maximal spacelike hypersurfaces.} In the following, near the maximal hypersurface $\Si$ we perturb the level sets $\Si_{x^0}$ on $\MM$ to maximal hypersurfaces. The next perturbation result is a paraphrase of \cite{PerturbationBruhat}.

\begin{theorem}[Construction of nearby maximal spacelike hypersurfaces by perturbation] \label{theorem:BruhatPerturbation} 
Let $m\geq0$ be an integer. Let $(\MM,\g)$ be a vacuum spacetime and let $\Si \subset \MM$ be a compact maximal spacelike hypersurface with boundary, that is, satisfying
\begin{align*}
H_\g(\Si)=0,
\end{align*}
where $H_\g(\Si)$ denotes the mean curvature of $\Si$ with respect to $\g$. Let $(x^\mu)_{\mu=0,1,2,3}$ be a shift-free coordinate system on $\MM$ such that $\Si = \{ x^0 = 1\}$. Let $\g'$ be another Lorentzian metric on $\MM$ such that for some $\varep>0$, with respect to the coordinate system $(x^\mu)_{\mu=0,1,2,3}$,
\begin{align*}
\Vert \g - \g' \Vert_{C^{m'}(\MM)} < \varep.
\end{align*}
There are universal $m'_0>$ and $\varep_0>0$ such that if $m'\geq m'_0$ and $0<\varep<\varep_0$, then there is a $C^{m}(\Si)$-function $\varphi: \Si \to \RRR$ with $\varphi \vert_{\pr \Si} = 0$ such that
\begin{align*}
H_{\g'}(\mathrm{graph}_{\Si}(\varphi)) =0,
\end{align*} 
where 
$$\mathrm{graph}_{\Si}(\varphi) := \{ x^0 = \varphi(x^1,x^2,x^3)+1 \} \subset \MM.$$
Moreover, we have the bound
\begin{align*}
\Vert \varphi \Vert_{C^{m}(\Si)} \leq C \varep,
\end{align*}
where the constant $C>0$ depends on $(\MM,\g)$, $\Si$ and $m$.
\end{theorem}

\begin{proof} The proof is based on the implicit function theorem. It is shown in \cite{PerturbationBruhat} that the linearisation of the mean curvature functional under graphs in shift-free coordinates is an isomorphism. We refer the reader to \cite{PerturbationBruhat} for more details.\end{proof} 

In the following we use Theorem \ref{theorem:BruhatPerturbation} to construct for a sufficiently small real $\tau>0$ a family of functions 
\begin{align} \label{EQfamilyFunctions}
(\varphi_{x^0}: \Si_{x^0} \to \RRR)_{1 \leq x^0 \leq 1+\tau}
\end{align}
such that for $1 \leq x^0 \leq1+ \tau$,
\begin{align*}
H_\g(\mathrm{graph}_{\Si_{x^0}}(\varphi_{x^0})) = 0.
\end{align*}
In Step 3 below, we show that the graphs $\mathrm{graph}_{\Si_{x^0}}(\varphi_{x^0})$ can be realised as level sets of a well-defined smooth time function $t$.\\

We turn to the construction of the family \eqref{EQfamilyFunctions}. By construction, in the $(x^1,x^2,x^3)$-coordinates, the boundary $\pr \Si_{x^0}$ of $\Si_{x^0}$ varies smoothly in $x^0$. Therefore, for $\tau>0$ sufficiently small, there is a smooth family of diffeomorphisms $(\Psi_{\la}: \RRR^3 \to \RRR^3)_{\la \geq0}$ such that for each fixed $\la\geq0$, for $x^0\geq1+\la$,
\begin{align*}
\Psi_{\la}(\Si_{x^0}) = \Si_{x^0-\la}.
\end{align*}
For $\tau>0$ sufficiently small, define further the next smooth family of spacetime diffeomorphisms. For $\la\geq0$, let $\mathbf{\Psi}_{\la}$ be such that
\begin{align*}
\mathbf{\Psi}_{\la}(x^0,x^1,x^2,x^3) = (x^0-\la, \Psi_\la(x^1,x^2,x^3)).
\end{align*}
In particular, $\mathbf{\Psi}_{\la}$ maps $\Si_{x^0}$ into $\Si_{x^0-\la}$ for $x^0 \geq \la$. We note that $\mathbf{\Psi}_{\la}(\Si_{1+\la}) = \Si_1$.\\

Let $m\geq0$ be an integer and $\varep>0$ be a real. For $\tau>0$ sufficiently small, it holds that for $0\leq \la \leq \tau$,
\begin{align*} 
\Vert \mathbf{\Psi}_{\la}^\ast \mathbf{g} - \mathbf{g} \Vert_{C^m(V)} \leq \varep,
\end{align*}
where $V$ is a fixed open portion of the future of $\Si$ in $(\MM,\g)$.\\

Therefore for $\tau>0$ sufficiently small, by Theorem \ref{theorem:BruhatPerturbation}, there exists a family of graphs denoted by
\begin{align*}
(\varphi_{1+\la}: \Si_{1} \to \RRR)_{0 \leq \la \leq \tau}
\end{align*}
such that
\begin{align*}
H_{\mathbf{\Psi}_{\la}^\ast \mathbf{g}}(\mathrm{graph}_{\Si_1}(\varphi_{1+\la})) =0.
\end{align*}
By applying the inverse diffeomorphism $\mathbf{\Psi}_{\la}^{-1}$, it follows that for $0 \leq \la \leq \tau$,
\begin{align*}
H_{\mathbf{g}}(\mathrm{graph}_{\Si_{1+\la}}(\varphi_{1+\la})) =0,
\end{align*}
where we abused notation by writing $\varphi_{1+\la}$ instead of $\varphi_{1+\la} \circ \mathbf{\Psi}_{\la}^{-1}$. This finishes the construction of the family \eqref{EQfamilyFunctions}.\\ 

{\bf Step 3: Analysis of the time function.} Let the scalar function $$1\leq t(x^0,x^1,x^2,x^3)\leq 1+\tau$$ on $\MM$ be implicity defined by
\begin{align} \label{eq:defOFtimeFunction}
x^0= \varphi_{t}(x^1,x^2,x^3) + t.
\end{align}

\begin{claim} \label{claim:AnalysisTimeFunction} The function $t(x^0,x^1,x^2,x^3)$ is locally well-defined and smooth. Moreover, $(t,x^1,x^2,x^3)$ locally are coordinates for the future of $\Si$ in $(\MM,\g)$.
\end{claim}

We start the proof of Claim \ref{claim:AnalysisTimeFunction} by estimating $\pr_t \varphi_t(x^1,x^2,x^3)$. By construction, on the one hand, for each $1\leq t\leq1+ \tau$, $\varphi_t$ satisfies the maximal surface equation
\begin{align} \label{eq:MaximalSurfaceEqH}
H_\g(\mathrm{graph}_{\Si_t}(\varphi_t)) = 0.
\end{align}
On the other hand, $\varphi_t \vert_{t=1} =0$ by the maximality of $\Si$.\\

Therefore, taking the $\pr_t$-derivative of \eqref{eq:MaximalSurfaceEqH} shows that $\pr_t \varphi_t \vert_{t=1}$ lies in the kernel of the linearisation of the mean curvature functional at $t=0$. However, in Theorem \ref{theorem:BruhatPerturbation} (see also the given remarks on the proof) it is shown that this kernel is trivial, hence we conclude
\begin{align*}
\pr_t \varphi_t \vert_{t=0} =0.
\end{align*}
By smoothness of the family $(\varphi_t)_t$ in $t$, it follows that for $\tau>0$ sufficiently small, for $1\leq t\leq 1+\tau$,
\begin{align*}
\Vert \pr_t \varphi_t \Vert_{C^0(\Si_t)} \leq 1/2.
\end{align*}

Plugging this into \eqref{eq:defOFtimeFunction}, we get that for $\tau>0$ sufficiently small, on $1 \leq t \leq 1+\tau$,
\begin{align*}
\pr_{x^0} t = \frac{1}{\pr_t x^0} = \frac{1}{ \pr_t\varphi_t + 1} \leq 2< \infty.
\end{align*}
This shows that $1\leq t \leq 1+\tau$ is well-defined and the level sets of $t$ (which are exactly the constructed maximal spacelike hypersurfaces) locally foliate the future of $\Si$ in $(\MM,\g)$. By construction it follows moreover that $(t,x^1,x^2,x^3)$ locally is a coordinate system. The smoothness of $t$ is then deduced from \eqref{eq:defOFtimeFunction}, details are left to the reader.

\begin{remark} In the continuity argument of Section \ref{sec:ProofOfMainTheorem}, we need to smoothly \emph{continue} a given time function. The proof is similar and details are left to the reader. The main point is to show that
\begin{align*}
\pr_t^{m}\varphi_t = 0 \text{ at } \{t=1\} = \Si.
\end{align*}
Indeed, this holds because $\varphi_t =0$ at $t=1$, so that arguing as for $\pr_t \varphi_t \vert_{t=1}$ above yields the result.
\end{remark}

\section{Existence of global coordinates on $\Si$ by Cheeger-Gromov theory} \label{SectionGlobalExistence} \label{Section:ProofOfCG}

In this section we prove Theorem \ref{prop:CGestimation1} by applying the Cheeger-Gromov theory developed in \cite{Czimek21}. We remark that Theorem \ref{prop:CGestimation1} is an example of a low regularity \emph{curvature pinching result} and its proof is based, like the standard higher regularity curvature pinching theory (see \cite{PetersenBook}), on a convergence result and a rigidity result, see Theorem \ref{THMconvergence123} and Lemma \ref{lemma:rigidityResult} below, respectively.\\

In the following, we first introduce the necessary definitions and prerequisite results before turning to the proof of Theorem \ref{prop:CGestimation1}.\\

\textbf{Notation.} We denote diffeomorphism equivalence and isometry of manifolds by $\simeq$ and $\cong$, respectively.

\begin{definition}[$H^2$-convergence of functions and tensors] Let $(M,g)$ be a compact Riemannian $3$-manifold with boundary. Let $(\varphi_i)$ be a finite number of fixed charts covering $M$. A sequence of functions $(f_n)_{n\in \mathbb{N}}$ on $M$ is said to converge in $H^2$ as $n \to \infty$, if for each $\varphi_i$, the pullbacks $(\varphi_i)^*f_n$ converge in $H^2$ as $n\to \infty$. The convergence of a sequence of tensors on $M$ in $H^2$ is defined similarly. \end{definition}

\begin{definition}[$H^2$-convergence of manifolds with boundary] \label{def:convergenceofmanifolds} 
A sequence $(M_n,g_n)$ of compact Riemannian $3$-manifolds with boundary is said to converge to a Riemannian manifold with boundary $(M,g)$ in the $H^2$-topology as $n \to \infty$, if for large $n$ there exist diffeomorphisms $\Psi_n: M \to M_n$ such that $(\Psi_n)^\ast g_n \to g$ in the $H^2$-topology on $M$.
\end{definition}

The following convergence theorem is a direct consequence of the theory developed in \cite{Czimek21}.
\begin{theorem}[$H^2$-convergence] \label{THMconvergence123} Let $(M_n,g_n)$ be a sequence of smooth compact Riemannian $3$-manifolds with boundary such that $M_n \simeq \overline{B(0,1)} \subset \RRR^3$ and for a real number $0< V< \infty$,
\begin{align} \begin{aligned}
\Vert \RRRic_n \Vert_{L^2(M_n)} \to&\, 0 \text{ as } n\to \infty,\\ 
\Vert \tr \Th_n -2  \Vert_{L^4(\pr M_n)}+ \Vert \widehat{\Th}_n   \Vert_{L^4(\pr M_n)} \to&\, 0 \text{ as } n\to \infty,\\
 r_{vol}(M_n,1/2) \geq& \,1/4, \\
 \mathrm{vol}_{g_n}(M_n) \leq&\, V.
\end{aligned} \label{eq:quantconvASSUM000} \end{align}
Then, there is a smooth compact Riemannian $3$-manifold $(M,g)$ with $M \simeq \overline{B(0,1)} \subset \RRR^3$ such that as $n\to \infty$,
\begin{align*}
(M_n,g_n) \to (M,g) \text{ in the }H^2\text{-topology,}
\end{align*}
that is, for large $n$ there are global diffeomorphisms $\Psi_n: M \to M_n$ such that, with respect to charts on $M$,
\begin{align*}
\Vert (\Psi_n^\ast g_n)_{ij} -g_{ij} \Vert_{H^2} \to 0.
\end{align*}
Moreover, for integers $m\geq1$ it holds that
\begin{align} \begin{aligned}
\Vert (\Psi_n^\ast g_n)_{ij} -g_{ij} \Vert_{H^{m+2}} 
\leq& \, C_{V} \sum\limits_{\vert \a \vert\leq m} \lrpar{ \Vert \nab^{\a} \RRRic \Vert_{L^2(M)} +\Vert \nab^{\a} \RRRic_n \Vert_{L^2(M_n)}} \\
& + C_{m,V} \Vert (\Psi_n^\ast g_n)_{ij} -g_{ij} \Vert_{H^2}.
\end{aligned}\label{EQhigherREGCGestimates1} \end{align}
\end{theorem}

In the proof of Theorem \ref{prop:CGestimation1}, we use in addition to the above the following rigidity result which identifies the limit manifold $(M,g)$ of Theorem \ref{THMconvergence123} as the unit ball in Euclidean space. 

\begin{lemma}[Rigidity result] \label{lemma:rigidityResult} Let $(M,g)$ be a smooth compact Riemannian $3$-manifold with boundary such that $M \simeq \overline{B(0,1)} \subset \RRR^3$ and 
\begin{align*}
\RRRic = 0 \text{ on } M, \,\, \tr\, \Th=2, \, \widehat{\Theta}=0 \text{ on } \pr M.
\end{align*}
Then, $$(M,g) \cong (\overline{B(0,1)},e).$$
\end{lemma}

\begin{proof} First, by the Gauss equation it follows that the Gauss curvature of $(\pr M, \gd)$ is $K=1$, and hence by classical differential geometry, $\gd$ is isometric to the standard round metric, that is, there exist smooth coordinates $(\th^1, \th^2)$ on $\pr M$ such that
\begin{align*}
\gd = (d\th^1)^2 + \sin^2(\th^1) (d\th^2)^2.
\end{align*}
Using the above coordinates $(\th^1, \th^2)$ together with $\tr \, \Th =2$ and $\widehat{\Theta}=0$ on $\pr M$, and $\RRRic=0$ on $M$, it is straight-forward to show that $(M,g)$ smoothly extends to $(\RRR^3 \setminus \overline{B(0,1)},e)$ when identifying $\pr M$ and $\pr \overline{B(0,1)} \subset \RRR^3$; see for example the analogous argument in Section 5.3.3 of \cite{Czimek21}. The resulting smooth Riemannian $3$-manifold is in particular flat, complete and has cubic volume growth of geodesic balls. Hence by Proposition 4.4 in \cite{Czimek21} it must be isometric to $(\RRR^3,e)$. Further, by the classical Liebmann's theorem, the only smooth simply connected closed $2$-surface with constant Gauss curvature $K=1$ in $\RRR^3$ is the round unit sphere, so we deduce that
\begin{align*}
(M,g) \cong (\overline{B(0,1)},e).
\end{align*}
This finishes the proof of Lemma \ref{lemma:rigidityResult}. 
\end{proof}

We are now in position to prove Theorem \ref{prop:CGestimation1}.

\begin{proof}[Proof of Theorem \ref{prop:CGestimation1}] The proof is by contradiction. Let $(M_n,g_n)$ be a sequence of smooth compact Riemannian $3$-manifolds with boundary such that $M_n \simeq \overline{B(0,1)} \subset \RRR^3$ and for a real number $0< V< \infty$,
\begin{align} \begin{aligned}
\Vert \RRRic_n \Vert_{L^2(M_n)} \leq&\, \frac{1}{n},\\ 
\Vert \tr \Th_n -2  \Vert_{L^4(\pr M_n)}+ \Vert \widehat{\Th}_n   \Vert_{L^4(\pr M_n)} \leq&\, \frac{1}{n},\\
 r_{vol}(M_n,1/2) \geq& \,1/4, \\
 \mathrm{vol}_{g_n}(M_n) \leq&\, V.
\end{aligned} \label{eq:quantconvASSUM} \end{align}

Given $0<\CMD<1/2$, \emph{assume there does not exist} an integer $N \geq1$ and finite constants $C_V>0$ and $C_{m,V}>0$, together with a family of global charts 
$$(\varphi_n: \overline{B(0,1)} \to M_n)_{n\geq N}$$ 
such that on $\overline{B(0,1)}$,
\begin{align} \begin{aligned} 
&\Vert (g_n)_{ij} -e_{ij} \Vert_{H^2(B(0,1))} \lesssim \CMD,\\  
&(1-\CMD) \vert \xi \vert^2 \leq (g_n)_{ij}\xi^i \xi^j \leq (1+\CMD) \vert \xi \vert^2 \text{ for all } \xi \in \RRR^2. 
\end{aligned} \label{eq:contraAssum} \end{align}
and for integers $m\geq1$,
\begin{align} \label{eq:contraAssum2}
\Vert (g_n)_{ij} -e_{ij} \Vert_{H^{m+2}({B(0,1)})} \les C_{V}\left( \sum\limits_{\vert \a \vert \leq m} \Vert \nab^{\a} \RRRic_n \Vert_{L^2(M_n)}+ C_{m,V} \right),
\end{align}
where, as in the following, we abuse notation by writing $g_n$ instead of $\Psi_n^\ast g_n$.\\

On the one hand, by Theorem \ref{THMconvergence123}, there is a smooth limit manifold $(M,g)$ such that as $n\to \infty$,
\begin{align} \label{EQCGproof1}
(M_n,g_n) \to (M,g) \text{ in the }H^2\text{-topology}
\end{align}
and the estimates \eqref{EQhigherREGCGestimates1} hold.\\

From the above convergence and \eqref{eq:quantconvASSUM}, it follows that the smooth limit manifold $(M,g)$ satisfies $\RRRic=0$ in $M$, and $\tr\Th=2$ and $\widehat{\Th}=0$ on $\pr M$. Hence by Lemma \ref{lemma:rigidityResult},
\begin{align}\label{EQCGproof2}
(M,g) \cong (\overline{B(0,1)},e),
\end{align}
which trivially admits global smooth coordinates. \\

Combining \eqref{EQCGproof1}, \eqref{EQCGproof2} with \eqref{EQhigherREGCGestimates1}, we get that for $n$ large, there are global diffeomorphisms $\Psi_n: \overline{B(0,1)} \to M_n$ satisfying
\begin{align*}
\Vert (g_n)_{ij} - e_{ij} \Vert_{H^2({B(0,1)})} \to 0 \text{ as } n \to \infty,
\end{align*}
and, for integers $m\geq1$,
\begin{align*}
\Vert (g_n)_{ij} - e_{ij} \Vert_{H^{m+2}({B(0,1)})} \leq& \, C_{V} \sum\limits_{\vert \a \vert\leq m} \lrpar{ \Vert \nab^{\a} \RRRic \Vert_{L^2(B(0,1))} +\Vert \nab^{\a} \RRRic_n \Vert_{L^2(M_n)}} \\
& + C_{m,V} \Vert (g_n)_{ij} -g_{ij} \Vert_{H^2(B(0,1))} \\
\les& \, C_{V} \sum\limits_{\vert \a \vert \leq m} \Vert \nab^{\a} \RRRic_n \Vert_{L^2(M_n)}  + C_{m,V},
\end{align*}
where we used that $n$ is large and $\RRRic=0$ on $M$. This yields a contradiction to \eqref{eq:contraAssum} and \eqref{eq:contraAssum2}, and hence finishes the proof of Theorem \ref{prop:CGestimation1}. \end{proof}

\appendix


\section{Global elliptic estimates for Hodge systems on $\Si$} \label{sec:AppendixProofOfLowRegEllipticEstimateForK} \label{sec:TensorialIdentities} \label{sec:AppendixHodgeElliptic} In this section we discuss global elliptic estimates for general Hodge systems on compact Riemannian $3$-manifolds with boundary $\Si$. This is a slight generalisation of the elliptic estimates in \cite{ChrKl93} where non-compact manifolds without boundary are considered.


\subsection{General Hodge systems on $\Si$} \label{SUBSECgeneralHodgeEstimates} In this section, we introduce tools and results to obtain elliptic estimates for general Hodge systems. They are applied in Sections \ref{SECellipticEstimatesKnotation} and \ref{SECehEnergyEstimateNotation} to the specific Hodge systems of this paper. We have the following notation.
\begin{definition} \label{def:AandDdef} Let $m\geq0$ be an integer. For a given totally symmetric $(m+2)$-tensor $F$, define
\begin{align*}
A(F)_{a_1 \dots a_{m+1}bc} :=& \nab_c F_{a_1 \dots a_{m+1} b} - \nab_b F_{a_1 \dots a_{m+1} c}, \\
D(F)_{a_1 \dots a_{m+1}} :=& \nab^c F_{a_{1} \dots a_{m+1} c}.
\end{align*} 
\end{definition}

The following lemma is a straight-forward generalisation of Lemma 4.4.1 in \cite{ChrKl93} to manifolds with boundary. The proof is by integration by parts and left to the reader.
\begin{lemma}[Fundamental integral identity for Hodge systems] \label{lem:IntegrationByPartsIdentity} Let $(\Si,g)$ be a compact Riemannian $3$-manifold with boundary and let $m\geq0$ be an integer. Let $F$ be a totally symmetric $(m+2)$-tensor on $\Si$. Then it holds that
\begin{align*}
\intSI \vert \nab F \vert^2 =& \intSI \half \vert A(F) \vert^2 + \vert D(F) \vert^2 \\
&- \intSI \sum\limits_{i=1}^{m+1}\Big( \mathrm{R}^l_{\,\,\, a_i bc} F_{a_1 \dots l \dots a_{m+1}c} +\RRRic^{l}_{\,\,\, b} F_{a_1 \dots a_{m+1}l} \Big) F^{a_1 \dots a_{m+1} b} \\
&- \intpSI F^{a_1 \dots a_{m+1} N} D(F)_{a_1 \dots a_{m+1}} + \intpSI \nab_b F_{a_1 \dots a_{m+1} N} F^{a_1 \dots a_{m+1} b}.
\end{align*}
\end{lemma}
In this paper we use Lemma \ref{lem:IntegrationByPartsIdentity} to derive elliptic estimates of Hodge systems. Higher regularity elliptic estimates for Hodge systems are proved by introducing the symmetrised derivative and reapplying Lemma \ref{lem:IntegrationByPartsIdentity}, see details below.

\begin{definition}[Symmetrised derivative] Let $m\geq0$ be an integer. For a given totally symmetric $(m+2)$-tensor $F$, let
\begin{align*}
(\overline{\nab} F)_{a_1 \dots a_{m+2} c} :=&  \frac{1}{m+3}\left( \nab_{c} F_{a_1 \dots a_{m+2}} + \sum\limits_{i=1}^{m+2} \nab_{a_i} F_{a_1 \dots c \dots a_{m+2}} \right).
\end{align*}
\end{definition}

In the next lemma, we express $A(\ol{\nab}F)$ and $D(\ol{\nab}F)$ in terms of $A(F)$ and $D(F)$.
\begin{lemma} \label{lemma:commutatorDA} Let $F$ be a totally symmetric $(m+2)$-tensor. Then it holds that
\begin{align*}
A(\overline{\nab} F) =& \nab A(F) + \mathrm{R} \cdot F, \\
D(\overline{\nab} F) =& \nab A(F) + \nab D(F) + \mathrm{R} \cdot F,
\end{align*}
where $\mathrm{R} \cdot F$ denotes contractions between $\mathrm{R}$ and $F$.
\end{lemma}

\begin{proof} First,
\begin{align*}
A(\ol{\nab}F)_{a_1 \cdots a_{m+2}bc} =& \nab_c \lrpar{\ol{\nab}F}_{a_1 \cdots a_{m+2}b}-\nab_b \lrpar{\ol{\nab}F}_{a_1 \cdots a_{m+2}c} \\
=& \frac{1}{m+3} \nab_c \lrpar{\nab_b F_{a_1 \cdots a_{m+2}} +\sum\limits_{i=1}^{m+2} \nab_{a_i} F_{a_1 \cdots b \cdots a_{m+2}}} \\
& -\frac{1}{m+3} \nab_b \lrpar{\nab_c F_{a_1 \cdots a_{m+2}} +\sum\limits_{i=1}^{m+2} \nab_{a_i} F_{a_1 \cdots c \cdots a_{m+2}}} \\
=& \frac{1}{m+3} \lrpar{\nab_c \nab_b - \nab_b \nab_c} F_{a_1 \cdots a_{m+2}} \\
&+\frac{1}{m+3} \sum\limits_{i=1}^{m+2} \lrpar{\lrpar{\nab_c \nab_{a_i} - \nab_{a_i} \nab_c} F_{a_1 \cdots b \cdots m+2} - \lrpar{\nab_b \nab_{a_i} - \nab_{a_i} \nab_b} F_{a_1 \cdots c \cdots m+2}} \\
&+\frac{1}{m+3} \sum\limits_{i=1}^{m+2} \nab_{a_i} \underbrace{\lrpar{\nab_c F_{a_1 \cdots b \cdots a_{m+2}} - \nab_b F_{a_1 \cdots c \cdots a_{m+2}}}}_{=A(F)_{a_1 \cdots a_{i-1} a_{i+1} \cdots a_{m+2} bc}},
\end{align*}
where we can further express
\begin{align*}
\lrpar{\nab_c \nab_b - \nab_b \nab_c} F_{a_1 \cdots a_{m+2}} = \sum\limits_{i=1}^{m+2} \mathrm{R}^d_{\,\,\, a_i b c} F_{a_1 \cdots d  \cdots a_{m+2}}.
\end{align*}

Second,
\begin{align*} \begin{aligned}
D(\ol{\nab}F)_{a_1 \cdots a_{m+2}} =&  \nab^c \lrpar{\ol{\nab}F}_{a_1 \cdots a_{m+2}c} \\
=& \frac{1}{m+3} \nab^c \lrpar{\nab_c F_{a_1 \cdots a_{m+2}} + \sum\limits_{i=1}^{m+2} \nab_{a_i} F_{a_1 \cdots c \cdots a_{m+2}}}\\
=& \frac{1}{m+3} \nab^c \nab_c F_{a_1 \cdots a_{m+2}} + \frac{1}{m+3}\sum\limits_{i=1}^{m+2} \nab^c \nab_{a_i} F_{a_1 \cdots c \cdots a_{m+2}},
\end{aligned} 
\end{align*}
where the first term on the right-hand side equals
\begin{align*}
 \nab^c \nab_c F_{a_1 \cdots a_{m+2}} =& \nab^c A(F)_{a_1 \cdots a_{m+2}c} +\lrpar{\nab^c \nab_{a_{m+2}} -\nab_{a_{m+2}} \nab^c} F_{a_1 \cdots a_{m+1} c} + \nab_{a_{m+2}} D(F)_{a_1 \cdots a_{m+1}},
\end{align*}
and for the sum on the right-hand side, 
\begin{align*}
\nab^c \nab_{a_i} F_{a_1 \cdots c \cdots a_{m+2}} = \lrpar{\nab^c \nab_{a_i}- \nab_{a_i} \nab^c} F_{a_1 \cdots c \cdots a_{m+2}} + \nab_{a_i} D(F)_{a_1 \cdots a_{i-1}a_{i+1} \cdots a_{m+2}}.
\end{align*}

This finishes the proof of Lemma \ref{lemma:commutatorDA}. \end{proof}

Thus by Lemma \ref{lemma:commutatorDA}, to derive higher regularity estimates for $F$ we can apply Lemma \ref{lem:IntegrationByPartsIdentity} to the Hodge system satisfied by $\ol{\nab}F$ and get an estimate for $\nab \ol{\nab} F$. To relate the regularity of the symmetrised derivative $\ol{\nab}F$ to the regularity of $\nab F$, we have the next lemma.

\begin{lemma} \label{lem:controlSYM} Let $m\geq0$ be an integer. Let $F$ be a totally symmetric $(m+2)$-tensor. Then, schematically,
\begin{align*}
\nab F = \overline{\nab} F + A(F).
\end{align*}
\end{lemma}

\begin{proof}[Proof of Lemma \ref{lem:controlSYM}] By direct calculation,
\begin{align*}
\nab_b F_{a_1 \dots a_{m+2}} =&(m+3) (\overline{\nab}F)_{a_1 \dots a_{m+2} b} - \sum\limits_{i=1}^{m+2} \nab_{a_i} F_{a_1 \dots b \dots a_{m+2}} \\
=& (m+3) (\overline{\nab}F)_{a_1 \dots a_{m+2} b} - \sum\limits_{i=1}^{m+2} \nab_{a_i} F_{a_1 \dots a_{i-1} a_{i+1} \dots a_{m+2}b} \\
=& (m+3) (\overline{\nab}F)_{a_1 \dots a_{m+2} b} - \sum\limits_{i=1}^{m+2} \Big(  \nab_{b} F_{a_1 \dots a_{m+2}} +  A(F)_{a_1 \dots a_{i-1}a_{i+1} \dots a_{m+2}b a_i } \Big) \\
=& (m+3) (\overline{\nab}F)_{a_1 \dots a_{m+2} b} - (m+2) \nab_{b} F_{a_1 \dots a_{m+2}} - \sum\limits_{i=1}^{m+2}  A(F)_{a_1 \dots a_{i-1}a_{i+1} \dots a_{m+2} b a_i}
\end{align*}
which shows that
\begin{align*}
\nab_b F_{a_1 \dots a_{m+2}}=&(\overline{\nab}F)_{a_1 \dots a_{m+2}b} - \frac{1}{m+3} \sum\limits_{i=1}^{m+2} A(F)_{a_1 \dots a_{i-1} a_{i+1} \dots a_{m+2} b a_i}.
\end{align*}
This finishes the proof of Lemma \ref{lem:controlSYM}. \end{proof}

To summarise the above, higher regularity estimates for Hodge systems can be proved by induction, using the recursive relation of Lemma \ref{lemma:commutatorDA}, the relation between $\ol{\nab}$ and $\nab$ of Lemma \ref{lem:controlSYM} and the basic integral identity of Lemma \ref{lem:IntegrationByPartsIdentity}. For ease of presentation, in the next sections we discuss more specifically the Hodge systems which appear in this paper.

\subsection{Elliptic estimates for the second fundamental form $k$ on a maximal hypersurface $\Si$} \label{SECellipticEstimatesKnotation}

Let $(\MM,\g)$ be a vacuum spacetime and let $\Si \simeq \ol{B(0,1)}$ be a compact spacelike maximal hypersurface in $\MM$. By \eqref{eq:divk}, \eqref{eq:curlk} and \eqref{eqtrkiszero}, the second fundamental form $k$ of $\Si$ satisfies the following Hodge system,
\begin{align*} \begin{aligned}
\Rscal(g) =& \vert k \vert_g^2, \\
\Div_g k =&0, \\
{\Curl}_g \, k=&H, \\
\tr_g \, k=&0.
\end{aligned} 
\end{align*}

In the notation of Definition \ref{def:AandDdef}, $k$ satisfies in particular
\begin{align} \label{eq:RelationAH}
A(k)_{iab} = \in^m_{\,\,\,\, ab}H_{im}, \,\, D(k)=0.
\end{align}

Lemma \ref{lem:IntegrationByPartsIdentity} together with \eqref{eq:RelationAH} yields the following corollary (see Section 8.3 in \cite{KlRodBreakdown} for the case of manifolds without boundary).
\begin{corollary}[Fundamental global elliptic estimate for $k$] \label{integralIDKEST} Let $(\MM,\g)$ be a vacuum spacetime and let $\Si \simeq \ol{B(0,1)}$ be a compact spacelike maximal hypersurface in $\MM$. Then it holds that
\begin{align*}
\int\limits_{\Si} \vert \nab k \vert^2 + \frac{1}{4} \vert k \vert^4- \int\limits_{\pr \Si}  \nab_a k_{bN} k^{ba} \lesssim \int\limits_{\Si} \vert \mathbf{R} \vert_{\mathbf{h}^t}^2,
\end{align*}
where $N$ denotes the outward-pointing unit normal to $\pr \Si \subset \Si$ and $T$ denotes the timelike unit normal to $\Si$.
\end{corollary}

\begin{proof} By Lemma \ref{lem:IntegrationByPartsIdentity} with $F=k$, we have
\begin{align} \label{EQkLowRegEllEst1}
\int\limits_{\Si} \vert \nab k \vert^2 =& \int\limits_{\Si} \half \vert H \vert^2 - \int\limits_{\Si} \left( \mathrm{R}^l_{\,\,\,abc} k_{lc} + \RRRic^l_{\,\,\, b} k_{al}\right) k^{ab} + \int\limits_{\pr \Si} \nab_b k_{aN} k^{ab}.
\end{align}
In dimension $n=3$, the full Riemann curvature tensor is determined by $\RRRic$, yielding
\begin{align*}
\mathrm{R}^l_{\,\,\,abc} k_{lc} k^{ab} = 2 \RRRic_{jl} k^{ji} k^{l}_{\,\,\, i} - \half \Rscal \vert k \vert^2.
\end{align*}
Plugging this into \eqref{EQkLowRegEllEst1}, we get
\begin{align} \label{EQkLowRegEllEst2}
\int\limits_{\Si} \vert \nab k \vert^2 =& \int\limits_{\Si} \half \vert H \vert^2 - \int\limits_{\Si} \left( 3 \RRRic^l_{\,\,\, b} k_{al} k^{ab} - \half \Rscal \vert k \vert^2 \right) + \int\limits_{\pr \Si} \nab_b k_{aN} k^{ab}.
\end{align}

On the one hand, by \eqref{eq:RicE}, we have
$$E_{ij} := \R_{TiTj} = \RRRic_{ij} - k_{im}k^m_{\,\,j}.$$
On the other hand, in dimension $n=3$, it holds for symmetric tracefree $2$-tensors $F$ that
$$3tr (F^4) \geq \vert F \vert^4.$$
Hence it follows from \eqref{EQkLowRegEllEst2} that
\begin{align*}
\int\limits_{\Si}\half \vert H \vert^2 =& \int\limits_{\Si} \vert \nab k \vert^2 + 3 \RRRic_{\,\,a}^{s} k_{bs} k^{ba} - \half \vert k \vert^4 - \int\limits_{\pr \Si}  \nab_a k_{bN} k^{ba} \\
=&   \int\limits_{\Si} \vert \nab k \vert^2 + 3 (E_{\,\,a}^{s} + k_m^{\,\, \, s} k^m_{\,\,\,  a} ) k_{bs} k^{ba} - \half \vert k \vert^4 - \int\limits_{\pr \Si}  \nab_a k_{bN} k^{ba}\\
\geq&   \int\limits_{\Si} \vert \nab k \vert^2 + 3 E_{\,\,a}^{s} k_{bs} k^{ba} + \half \vert k \vert^4 - \int\limits_{\pr \Si}  \nab_a k_{bN} k^{ba}.
\end{align*}
Using that $\vert E \vert^2 + \vert H \vert^2 \les \vert \mathbf{R} \vert_{\mathbf{h}^t}^2$ by \eqref{EQEquivalenceNORMS1} and \eqref{EQEquivalenceNORMS2}, we obtain
\begin{align*}
\int\limits_{\Si} \vert \nab k \vert^2 + \frac{1}{4} \vert k \vert^4- \int\limits_{\pr \Si}  \nab_a k_{bN} k^{ba} \lesssim \int\limits_{\Si} \vert \mathbf{R} \vert_{\mathbf{h}^t}^2.
\end{align*}
This finishes the proof of Corollary \ref{integralIDKEST}. \end{proof}

Furthermore, the next higher regularity estimates for $k$ follow by a standard induction argument as outlined above in Section \ref{SUBSECgeneralHodgeEstimates}; we leave details to the reader.
\begin{lemma}[Higher regularity elliptic estimates for $k$] \label{LEMkHigherRegEllEstBdryTerms} Let $(\MM,\g)$ be a vacuum spacetime and let $\Si \simeq \ol{B(0,1)}$ be a compact spacelike maximal hypersurface in $\MM$. For integers $m\geq1$, it holds that
\begin{align*}
\sum\limits_{\vert \a \vert \leq m+1} \Vert \nab^{\a} k \Vert_{L^2(\Si)}+\BB  \les&  \sum\limits_{\vert \a \vert \leq m} \Vert \nab^{\a} H \Vert_{L^2(\Si)}+C\lrpar{ \Vert k \Vert_{L^2(\Si)}+ \Vert \nab k \Vert_{L^2(\Si)}} ,
\end{align*}
where $\BB$ denotes boundary integrals and the constant $C>0$ depends on $m$.
\end{lemma} 
\begin{remark} The boundary integrals in Corollary \ref{integralIDKEST} and Lemma \ref{LEMkHigherRegEllEstBdryTerms} discussed in detail in Sections \ref{SectionBAImprovementMAIN} and \ref{SectionHigherRegularity}.
\end{remark}

\subsection{Elliptic estimates on $\Si_t$ for curvature} \label{SECehEnergyEstimateNotation} Let $(\MM,\g)$ be a vacuum spacetime and let $(\Si_t)$ be a maximal foliation on $\MM$. Let $T$ denote the timelike unit normal to $\Si_t$. We recall from Proposition \ref{prop:MaxwellsEq1} and \eqref{EQTrelationEH} that for a Weyl tensors $\mathbf{W}$ satisfying the inhomogeneous Bianchi equations
\begin{align*}
\D^\a \mathbf{W}_{\a \be \ga \de} = J_{\be \ga \de}, 
\end{align*}
it holds that
\begin{align} \begin{aligned}
\Div E(\mathbf{W})_a =& +\lrpar{k \wedge H(\mathbf{W})}_a + J_{TaT}, \\
\Curl E(\mathbf{W})_{ab} =& +H\lrpar{\hat{\Lie}_T \mathbf{W}}_{ab} - 3 \lrpar{n^{-1} \nab n \wedge E(\mathbf{W})}_{ab} \\
& - \frac{3}{2} \lrpar{ k \times H(\mathbf{W})}_{ab} - J^\ast_{aTb}, \\
\Div H(\mathbf{W})_a =& -\lrpar{k \wedge E(\mathbf{W})}_a + J^\ast_{TaT}, \\
\Curl H(\mathbf{W})_{ab} =&-E \lrpar{\hat{\Lie}_T \mathbf{W}}_{ab}  -3 \lrpar{n^{-1} \nab n \wedge H(\mathbf{W})}_{ab} \\
&+ \frac{3}{2} \lrpar{k \times E(\mathbf{W})}_{ab} - J_{aTb}.
\end{aligned} \label{eq:BianchiEH2} \end{align}

Interpreting \eqref{eq:BianchiEH2} as coupled Hodge system for $E(\mathbf{W})$ and $H(\mathbf{W})$, we get the next global elliptic estimates on $\Si$ as corollary of the fundamental integral identity of Lemma \ref{lem:IntegrationByPartsIdentity}; we leave details to the reader.

\begin{corollary}[Elliptic estimates for $E(\mathbf{W})$ and $H(\mathbf{W})$] \label{CORellipticEstimatesEH1} Let $E(\mathbf{W})$ and $H(\mathbf{W})$ be solutions to \eqref{eq:BianchiEH2} on $\Si_t$. Then it holds that
\begin{align*}
&\int\limits_{\Si_t} \vert \nab E(\mathbf{W}) \vert^2 + \vert \nab H(\mathbf{W}) \vert^2 \\
\les& \int\limits_{\Si_t} \lrpar{\vert \Lieh_T\mathbf{W} \vert_{\mathbf{h}^t}^2 + \vert J \vert^2} + \int\limits_{\pr\Si_t} \nab_b E(\mathbf{W})_{aN} \, E(\mathbf{W})^{ab} +  \int\limits_{\pr\Si_t} \nab_b H(\mathbf{W})_{aN} \, H(\mathbf{W})^{ab} \\
&- \int\limits_{\pr\Si_t} J_{TaT} \, E(\mathbf{W})^{aN} -  \int\limits_{\pr\Si_t} J^\ast_{TaT} \, H(\mathbf{W})^{aN} \\
&+ \lrpar{ \Vert n^{-1} \nab n \Vert_{L^\infty(\Si_t)} + \Vert \nab k \Vert^2_{L^2(\Si_t)} + \Vert k \Vert^2_{L^2(\Si_t)} } \lrpar{\Vert \nab E \Vert^2_{L^2(\Si_t)} + \Vert E \Vert^2_{L^2(\Si_t)}}\\
&+ \lrpar{ \Vert n^{-1} \nab n \Vert_{L^\infty(\Si_t)} + \Vert \nab k \Vert^2_{L^2(\Si_t)} + \Vert k \Vert^2_{L^2(\Si_t)} } \lrpar{\Vert \nab H \Vert^2_{L^2(\Si_t)} + \Vert H \Vert^2_{L^2(\Si_t)}}.
\end{align*} 

\end{corollary}

Furthermore, the next higher regularity estimates for $E(\mathbf{W})$ and $H(\mathbf{W})$ follow by a standard induction argument as outlined above in Section \ref{SUBSECgeneralHodgeEstimates}; we leave details to the reader.
\begin{lemma}[Higher regularity elliptic estimates for $E(\mathbf{W})$ and $H(\mathbf{W})$] \label{LEMEandHHigherRegEllEstBdryTerms} Let $(\MM,\g)$ be a vacuum spacetime and let $\Si_t \simeq \ol{B(0,1)}$ be a compact spacelike maximal hypersurface in $\MM$. For integers $m\geq1$, it holds that
\begin{align*}
&\sum\limits_{\vert \a \vert \leq m+1} \Vert \nab^{\a} E(\mathbf{W}) \Vert_{L^2(\Si_t)}+\Vert \nab^{\a} E(\mathbf{W}) \Vert_{L^2(\Si_t)}  \\
\les&  \sum\limits_{\vert \a \vert \leq m} \Vert \nab^{\a} \Lieh_T\mathbf{W} \Vert_{L^2(\Si_t)}+\Vert \nab^{\a} J \Vert_{L^2(\Si_t)} + \Vert \nab^{\a} k \Vert_{L^2(\Si_t)} +  \Vert \nab^{\a} \nab n \Vert_{L^2(\Si_t)}+ \BB \\
&+C\lrpar{ \Vert k \Vert_{L^2(\Si_t)}+ \Vert \nab k \Vert_{L^2(\Si_t)} +\Vert \nab n \Vert_{L^\infty(\Si_t)} } \\
&\qquad \cdot \lrpar{ \sum\limits_{\vert \a \vert \leq m+1} \Vert \nab^{\a} E(\mathbf{W}) \Vert_{L^2(\Si_t)}+\Vert \nab^{\a} E(\mathbf{W}) \Vert_{L^2(\Si_t)}},
\end{align*}
where $\BB$ denotes boundary integrals and the constant $C>0$ depends on $m$.
\end{lemma} 
\begin{remark} The boundary integrals in Corollary \ref{CORellipticEstimatesEH1} and Lemma \ref{LEMEandHHigherRegEllEstBdryTerms} are estimated by the initial data norms in Section \ref{SectionHigherRegularity}, where we do (in contrast to the elliptic estimates for $k$) not use their particular structure.
\end{remark}

\section{Proof of Lemma \ref{Lemma:TraceEstimateH1toL4bdry}} \label{SECproofTRACEEST} In this section we prove the trace estimates of Lemma \ref{Lemma:TraceEstimateH1toL4bdry}. We have to show that for Riemannian metrics $g$ on a manifold $\Si$ satisfying in global coordinates
\begin{align*}
\frac{1}{4} \vert \xi \vert \leq g_{ij}\xi^i \xi^j \leq 2 \vert \xi \vert^2 \text{ for all } \xi \in \RRR^2
\end{align*}
the next trace estimate for tensors $F$ holds,
\begin{align} \begin{aligned}
\Vert F \Vert_{L^2(\pr \Si)} \lesssim&\, \Vert F \Vert_{L^2(\Si)} + \Vert \nab F \Vert_{L^2(\Si)}, \\
\Vert F \Vert_{L^4(\pr \Si)} \lesssim& \,\Vert F \Vert_{L^2(\Si)} + \Vert \nab F \Vert_{L^2(\Si)}.
\end{aligned} \label{EQtraceEstimateAPPPROOF} \end{align}
Moreover, we show that on weakly regular balls $(\Si,g)$ with constant $0<\CMD<1/2$, it holds that
\begin{align} \label{EQtraceEstimate12prereqAPPPROOF}
\Vert F \Vert_{H^{1/2}(\pr \Si)} \lesssim&\, \Vert F \Vert_{L^2(\Si)} + \Vert \nab F \Vert_{L^2(\Si)},
\end{align}
and for integers $m\geq1$,
\begin{align} \label{EQhigherregtraceEstimateAPPPROOF}
\sum\limits_{\vert \a \vert \leq m} \Vert \Nd^{\a}F \Vert_{H^{1/2}(\pr \Si)} \les&\, \sum\limits_{\vert \a \vert \leq m+1} \Vert \nab^{\a} F \Vert_{L^2(\Si)}+ \sum\limits_{\vert \a \vert \leq m} \Vert \nab^{\a} \RRRic \Vert_{L^2(\Si)} + C_m\CMD.
\end{align}


First, the estimates \eqref{EQtraceEstimateAPPPROOF} are straight-forward, see for example Corollary 3.26 in \cite{J3} for a concise proof. \\

We turn to discuss \eqref{EQtraceEstimate12prereqAPPPROOF}. On the one hand, by Sections 7.50 to 7.56 in \cite{Adams}, for each open smooth sets $U \subset \subset \pr \ol{B(0,1)}$ and $V \subset \ol{B(0,1)}$ with $U \subset V \cap \pr \ol{B(0,1)}$, it holds that
\begin{align*}
H^1(V) \hookrightarrow H^{1/2}(U),
\end{align*}
where $H^{1/2}(U)$ is a local, coordinate-defined fractional Sobolev space on $U$. \\

On the other hand, if $g_{ij} \in H^2(\ol{B(0,1)})$ then in particular $\gd_{AB} \in W^{1,4}(\pr \ol{B(0,1)})$ in local coordinates on the unit sphere. By Proposition 3.2 in \cite{ShaoBesov}, this control suffices to compare the coordinate-defined spaces $H^{1/2}(U)$ with the space $H^{1/2}(\pr \ol{B(0,1)})$ which is defined more geometrically in Definition \ref{DEFHSspaces}. We refer also to Appendix B of \cite{ShaoBesov}. This finishes the proof of \eqref{EQtraceEstimate12prereqAPPPROOF}, details are left to the reader.\\

The higher regularity trace estimate \eqref{EQhigherregtraceEstimateAPPPROOF} is a straight-forward generalisation of \eqref{EQtraceEstimate12prereqAPPPROOF} using the fact that on a weakly regular ball we have, by definition,
\begin{align*}
\sum\limits_{\vert \a \vert \leq m+2} \Vert \pr^{\a} (g_{ij}-e_{ij}) \Vert_{L^2({B(0,1)})} \les \sum\limits_{\vert \a \vert \leq m} \Vert \nab^{\a} \RRRic \Vert_{L^2(\Si)} + C_m \CMD.
\end{align*}
This finishes the proof of Lemma \ref{PROPtraceEstimate}.


\section{Comparison estimates between two maximal foliations on $\MM$} \label{secComparisonAppendix} In this section we prove Lemma \ref{lem:comparisonfoliationMM}. First, we recall the geometric setup. Let $1 < t^\ast \leq 2$ be a real and let $(\MM_{t^\ast},\g)$ be a vacuum spacetime whose past boundary consists of a spacelike maximal hypersurface with boundary $\Si \simeq \ol{B(0,1)}$ and its outgoing null hypersurface $\HH$ emanating from $\pr \Si$. \\

Assume that $\MM_{t^\ast}$ is foliated by a maximal spacelike foliations $(\Si_t)_{1\leq t \leq t^\ast}$ with $\Si_1 = \Si$ and satisfying
\begin{align} \begin{aligned}
\Vert \R \Vert_{L^\infty_{t} L^2({\Si}_{t})} \leq& \,D\varep, \\
\Vert \nab k \Vert_{L^\infty_{ t} L^2({\Si}_{ t})} + \Vert {k} \Vert_{L^\infty_{ t} L^2({\Si}_{ t})} \leq&\, D\varep, \\
 \Vert n-1 \Vert_{L^\infty(\MM_{t^\ast})} + \Vert \nab n \Vert_{L^\infty_tL^2(\Si_t)} +\Vert \nab^2 n \Vert_{L^\infty_tL^2(\Si_t)}\les &\,D\varep.
\end{aligned} \label{eqCOMPARISONspacetimeControl} 
\end{align}
Let $T$ denote the timelike unit normal to $\Si_t$.\\

Assume further that there is a second foliation on $\MM_{t^\ast}$ by maximal hypersurfaces $(\wt{\Si}_{\tilde t})_{0 \leq \tilde t \leq t^\ast}$ with $\Si_{t^\ast} = \wt{\Si}_{t^\ast}$ satisfying
\begin{align} \begin{aligned}
\Vert \R \Vert_{L^\infty_{\tilde t} L^2(\wt{\Si}_{\tilde t})} \leq& \,\CMD, \\
\Vert \nabtt \tilde{k} \Vert_{L^\infty_{\tilde t} L^2(\wt{\Si}_{\tilde t})} + \Vert \tilde{k} \Vert_{L^\infty_{\tilde t} L^2(\wt{\Si}_{\tilde t})}+ \Vert \D \wt{\pi} \Vert_{L^\infty_\ttt L^2(\Sitt_\ttt)}+\Vert \wt{\mathbf{A}} \Vert_{L^\infty_\ttt L^4(\Sitt_\ttt)} \leq& \,\CMD, \\
\Vert \tilde{n}-1 \Vert_{L^\infty(\MM_{t^\ast})} + \Vert \nabtt \tilde{n} \Vert_{L^\infty(\MM_{t^\ast})} + \Vert \nabtt^2 \tilde{n} \Vert_{L^\infty_\ttt L^2(\wt{\Si}_\ttt)} \les&\, \CMD,
\end{aligned} \label{eqCOMPARISONspacetimeControl2222} \end{align}
where $\nabtt$ and $\tilde{k}$ denote the induced covariant derivative and the second fundamental form on $\wt{\Si}_{\ttt}$, respectively. Let $\Ttt$ denote the timelike unit normal to  $\wt{\Si}_{\tilde t}$. \\

We turn to the proof of Lemma \ref{lem:comparisonfoliationMM}. We need to show that for $\varep>0$ and $\CMD>0$ sufficiently small, for $1\leq t \leq t^\ast$,
\begin{align} \label{EQlemmaComparisonTwoEstimatesREMINDER}
\norm{\nutt-1}_{L^\infty_t L^\infty(\Si_t)} \les& \,\CMD, \\
\label{EQlemmaComparisonTwoEstimatesREMINDER2222}
\Vert \wt{k} \Vert_{L^\infty_t L^4(\Si_t)} \les&\, \CMD,
\end{align}
where the \emph{angle} $\tilde{\nu}$ between $T$ and $\Ttt$ is defined as
\begin{align} \label{EQreminderDefinitionNUTT}
\nutt := -\g(T, \Ttt).
\end{align}

The proof of Lemma \ref{lem:comparisonfoliationMM} is based on a standard continuity argument starting at $\Si_{t^\ast}=\Sitt_{t^\ast}$, where by construction $\nutt=1$ and $k=\ktt$, and going backwards in $t$. In the following, we only discuss the bootstrap assumption and its improvement. Details are left to the reader. \\

\textbf{Bootstrap assumption.} Let $1\leq t_0^\ast<t^\ast$ be a real. Assume that for a large constant $M>0$,
\begin{align} \begin{aligned}
\norm{\nutt-1}_{L^\infty(\MM_{t^\ast})} \leq M \CMD.
\end{aligned}\label{est:BAnutett}\end{align}

\textbf{First consequences of the bootstrap assumption.} By writing for an orthonormal frame $(\ett_i)_{i=1,2,3}$ on $\wt{\Si}_\ttt$ the relation
\begin{align*}
T=\nutt \Ttt + \g(T, \tilde{e}_i) \tilde{e}_i,
\end{align*}
and using that $\g(T,T)=-1$, we have that 
\begin{align*}
-1 = - \vert \nutt \vert^2 + \sum\limits_{i=1,2,3} \vert \g(T,\ett_i)\vert^2,
\end{align*}
which implies by \eqref{est:BAnutett} that
\begin{align} \label{EQestContraCoeff}
\norm{\g(T,\ett_i)}_{L^\infty(\MM_{t^\ast})}  \les \sqrt{M\CMD}.
\end{align}

Let $(e_i)_{i=1,2,3}$ be the orthonormal frame tangent to $\Si_t$ constructed by the Gram-Schmidt method applied to the $\Si_t$-tangential frame 
$$(\ett_1+\g(\ett_1,T)T,\ett_2+\g(\ett_2,T)T,\ett_3+\g(\ett_3,T)T),$$ 
and set moreover $e_0 := T$. By the Gram-Schmidt construction and \eqref{est:BAnutett} and \eqref{EQestContraCoeff}, we get that for $i \neq j$ and $\varep', \varep>0$ sufficiently small,
\begin{align}
\begin{aligned}
\norm{\g(e_i,\ett_i)-1}_{L^\infty(\MM_{t^\ast})}  + \norm{\g(\ett_i,e_j)}_{L^\infty(\MM_{t^\ast})} + \norm{\g(\Ttt,e_i)}_{L^\infty(\MM_{t^\ast})} \les \sqrt{M\CMD}.                                
\end{aligned} \label{est:nuttettcomp}           
\end{align}

\textbf{Improvement of the bootstrap assumption.} Define shift-free coordinates $(t,x^1,x^2,x^3)$ on $\MM_{t^\ast}$ by transporting $(x^1,x^2,x^3)$ from $\Si_{t^\ast}$ along $T$. Then it holds that $\pr_t = n^{-1} T$.\\

First, by definition of $\nutt$, see \eqref{EQreminderDefinitionNUTT}, we get
\begin{align*}
n \pr_t \nutt = & -\g(\D_{T} T,\Ttt) -\g(T,\D_{T}\Ttt) \\
= & - \left( \g(\Ttt,e_i)\g(\D_{T}T,e_i) + \g(T,\ett_\mu) \g(T,\ett_i) \g(\ett_i, \D_{\ett_\mu} \Ttt) \right) \\
= & -\left(  \g(\Ttt,e_i)(-n^{-1}\nab_{e_i} n) - \g(T,\ett_j) \g(T,\ett_i) \wt{k}_{ij} - \nutt \g(T,\ett_i) (\wt{n}^{-1}\wt{\nab}_i\wt{n}) \right).
\end{align*}
Integrating in $t$ and using that $\nutt \vert_{t=t^\ast}=1$, we get that for $\varep>0$ and $\CMD>0$ sufficiently small, for $t^\ast_0 \leq t \leq t^\ast$,
\begin{align} \begin{aligned}
&\norm{\nutt-1}_{L^\infty_t L^4(\Si_t)}\\
 \les & \norm{\g(\Ttt,e_i)}_{L^\infty(\MM_{t^\ast})}\norm{\nab_{e_i} n}_{L^4(\MM_{t^\ast})} \\
 &+ \norm{\g(T,\ett_i)}_{L^\infty(\MM_{t^\ast})} \norm{\g(T,\ett_j)}_{L^\infty(\MM_{t^\ast})}\norm{\wt{k}_{ij}}_{L^4(\MM_{t^\ast})} \\
&+ \norm{\nutt}_{L^\infty(\MM_{t^\ast})}\norm{\g(T,\ett_i)}_{L^\infty(\MM_{t^\ast})}\norm{\wt{\nab}_i\wt{n}}_{L^4(\MM_{t^\ast})} \\
 \les & \norm{\g(\Ttt,e_i)}_{L^\infty(\MM_{t^\ast})}\left(\norm{\nab n}_{L^\infty_tL^2(\Si_t)}+\norm{\nab^2 n}_{L^\infty_tL^2(\Si_t)} \right) \\
 &+ \norm{\g(T,\ett_i)}_{L^\infty(\MM_{t^\ast})} \norm{\g(T,\ett_j)}_{L^\infty(\MM_{t^\ast})}\left(\norm{\wt{k}}_{L^\infty_\ttt L^2(\wt{\Si}_\ttt)}+\norm{\wt{\nab} \wt{k}}_{L^\infty_\ttt L^2(\wt{\Si}_\ttt)} \right) \\
&+ \norm{\nutt}_{L^\infty(\MM_{t^\ast})}\norm{\g(T,\ett_i)}_{L^\infty(\MM_{t^\ast})}\left(\norm{\wt{\nab}\wt{n}}_{L^\infty_tL^2(\Si_t)}+\norm{\wt{\nab}^2 \wt{n}}_{L^\infty_tL^2(\Si_t)} \right) \\
\les & \sqrt{M\CMD} (D\varep) + \sqrt{M\CMD} \CMD,
\end{aligned} \label{EQcomparisonEST1} \end{align}
where we used \eqref{eqCOMPARISONspacetimeControl}, \eqref{est:BAnutett}, \eqref{EQestContraCoeff} and \eqref{est:nuttettcomp}, and Lemma \ref{LEMsobolevEmbeddingSigmaT}.\\

Second, by definition of $\nutt$ in \eqref{EQreminderDefinitionNUTT}, for $i=1,2,3$,
\begin{align*}
\nab_{e_i}\nutt = & -\g(\D_{e_i}T,\Ttt) - \g(T,\D_{e_i}\Ttt) \\
= & - \left( \g(\Ttt,e_j)k_{ij} + \g(T,\ett_j) \g(e_i,\ett_\mu) \g(\ett_j, \D_{\ett_\mu} \Ttt) \right) \\
=& - \left( \g(\Ttt,e_j)k_{ij} - \g(T,\ett_j) \g(e_i,\ett_l) \wt{k}_{lj}+ \g(T,\ett_j) \g(e_i,\Ttt) (\wt{n}^{-1} \wt{\nab}_j \wt{n}) \right)
\end{align*}

By \eqref{eqCOMPARISONspacetimeControl}, \eqref{est:nuttettcomp} and standard Sobolev embedding on $\Si_t$, we have
\begin{align} \begin{aligned}
\norm{\nab \nutt}_{L^\infty_t L^4(\Si_t)} \les & \norm{\g(\Ttt,e_j)}_{L^\infty(\MM_{t^\ast})}\norm{k}_{L^\infty_t L^4(\Si_t)} \\
&+ \norm{\g(T,\ett_j)}_{L^\infty(\MM_{t^\ast})}\norm{\g(e_i,\ett_l)}_{L^\infty(\MM_{t^\ast})}\norm{\wt{k}}_{L^\infty_t L^4(\Si_t)} \\
&+ \norm{\g(T,\ett_j)}_{L^\infty(\MM_{t^\ast})}\norm{\g(e_i,\Ttt)}_{L^\infty(\MM_{t^\ast})}\norm{\wt{\nab} \wt{n}}_{L^\infty_t L^\infty(\Si_t)} \\
\les & \sqrt{M\CMD} D\varep + \sqrt{M\CMD }\norm{ \wt{k}}_{L^\infty_t L^4(\Si_t)} +\sqrt{M\CMD} \CMD.
\end{aligned} \label{EQcomparisonEST2} \end{align}

By \eqref{EQcomparisonEST1} and \eqref{EQcomparisonEST2}, it follows with standard Sobolev embedding on $\Si_t$ (see Lemma \ref{LEMsobolevEmbeddingSigmaT}) that 
\begin{align} \label{eqPreliminaryComparisonEstimate1}
\norm{\nutt-1}_{L^\infty(\MM_{t^\ast})} \les \sqrt{M\CMD} (D\varep+\CMD) + \sqrt{M\CMD} \norm{\wt{k} }_{L^\infty_t L^4(\Si_t)}. 
\end{align}

To estimate the remaining term $\norm{\wt{k} }_{L^\infty_t L^4(\Si_t)}$ on the right-hand side of \eqref{eqPreliminaryComparisonEstimate1}, we apply the following technical lemma whose proof is postponed to the end of this section.

\begin{lemma}[Technical lemma] \label{lem:compL4MM} Under the assumption of \eqref{eqCOMPARISONspacetimeControl}, \eqref{eqCOMPARISONspacetimeControl2222}, \eqref{est:BAnutett} and $\CMD>0$ and $\varep>0$ sufficiently small, it holds for every scalar function $f$ on $\MM_{t^\ast}$ that for $1\leq t \leq t^\ast$,
\begin{align*}
\norm{f}_{L^\infty_t L^4(\Si_t)} \les& \norm{ \D f }_{L^\infty_\ttt L^2(\Sitt_\ttt)}^{1/4} \norm{f}_{L^\infty_\ttt L^6(\Sitt_\ttt)}^{3/4} + \norm{f}_{L^4(\Si_{t^\ast})}.
\end{align*}
\end{lemma}

Applying Lemma \ref{lem:compL4MM} to $f = \wt{k}_{ij}$, we have
\begin{align}  \begin{aligned}
\Vert\wt{k}_{ij} \Vert_{L^\infty_t L^4(\Si_t)} \les& \, \Vert \D (\wt{k}_{ij})\Vert_{L^\infty_\ttt L^2(\Sitt_\ttt)}^{1/4} \Vert \wt{k}_{ij} \Vert^{3/4}_{L^\infty_\ttt L^6(\Sitt_\ttt)} +\Vert \wt{k}_{ij}\Vert_{L^4(\Si_{t^\ast})}\\
\les& \lrpar{\Vert \D \wt{\pi} \Vert_{L^\infty_\ttt L^2(\Sitt_\ttt)} + \Vert \wt{\pi} \Vert_{L^\infty_\ttt L^4(\Sitt_\ttt)}\Vert \wt{\mathbf{A}} \Vert_{L^\infty_\ttt L^4(\Sitt_\ttt)} }^{1/4} \Vert \wt{k} \Vert^{3/4}_{L^\infty_\ttt L^6(\Sitt_\ttt)} \\
&+\Vert k\Vert_{L^4(\Si_{t^\ast})}\\
\les& \, \CMD^{1/4} \CMD^{3/4} + D\varep.
\end{aligned}\label{EQapplicationTechnicalLemmaAppendix} \end{align}

Plugging \eqref{EQapplicationTechnicalLemmaAppendix} into \eqref{eqPreliminaryComparisonEstimate1}, we get that for $\CMD>0$ and $\varep>0$ sufficiently small, 
\begin{align*}
\norm{\nutt-1}_{L^\infty(\MM_{t^\ast})} \les&  \sqrt{M\CMD}(D\varep+\CMD) \\
\les& \,\CMD,
\end{align*}
where the constant is strictly smaller than $M$. This improves the bootstrap assumption \eqref{est:BAnutett} and hence proves \eqref{EQlemmaComparisonTwoEstimatesREMINDER}. The estimate \eqref{EQlemmaComparisonTwoEstimatesREMINDER2222} follows directly from \eqref{EQapplicationTechnicalLemmaAppendix}. This finishes the proof of Lemma \ref{lem:comparisonfoliationMM}. \\

It remains to prove Lemma~\ref{lem:compL4MM}. Let $(t,x^1,x^2,x^3)$ denote the shift-free coordinate system constructed on $\MM_{t^\ast}$ above. Using that $\tr k =0$, \eqref{eqCOMPARISONspacetimeControl}, \eqref{eqCOMPARISONspacetimeControl2222} and $\pr_t = n^{-1} T$, we have for $t^\ast_0 \leq t \leq t^\ast$,
\begin{align*}
\int\limits_{\Si_{t}} f^4 \,d\mu_{g} =& \int\limits_{t^\ast}^{t} \pr_t\left(\, \int\limits_{\Si_{t'}}f^4 \,d\mu_g \right) dt' + \int\limits_{\Si_{t^\ast}} f^4 \,d\mu_{g} \\
=&\, 4\int\limits_{t^\ast}^{t } \left( \, \int\limits_{\Si_{t'}}\pr_t f f^3 \,d\mu_{g} \right) dt' + \int\limits_{\Si_{t^\ast}} f^4 \,d\mu_{g}\\
\les& \int\limits_{\MM_{t^\ast}} \vert \D f\vert_{\mathbf{h}^t} \vert f \vert^3  + \int\limits_{\Si_{t^\ast}} f^4 \,d\mu_{g} \\ 
\les& \int\limits_{\MM_{t^\ast}} \vert \D f\vert_{\mathbf{h}^\ttt} \vert f \vert^3  + \int\limits_{\Si_{t^\ast}} f^4 \,d\mu_{g} \\ 
\les&  \norm{ \D f}_{L^\infty_\ttt L^2(\Sitt_\ttt)}  \norm{f}_{L^\infty_\ttt L^6(\Sitt_\ttt)}^3+ \int\limits_{\Si_{t^\ast}} f^4 \,d\mu_{g},
\end{align*}
where we used \eqref{est:BAnutett} to compare the positive-definite norms $\mathbf{h}^\ttt$ and $\mathbf{h}^t$. This finishes the proof of Lemma \ref{lem:compL4MM}.


\section{Proof of Proposition \ref{PropTrilinearM1}}

\label{SECtrilinearEstimateM1} In this section, we provide more details on the proof of Proposition \ref{PropTrilinearM1}, that is, the claim that for
\begin{align*}
\EE_1 := \int\limits_{\MM_{t^\ast}} \frac{3}{2} Q(\Lieh_{\Ttt} \R)_{\a \be \Ttt\Ttt}\wt{\pi}^{\a\be}, \,\,\, \EE_2 := \int\limits_{\MM_{t^\ast}} \D^\a Q(\Lieh_{\Ttt} \R)_{\a \Ttt\Ttt\Ttt},
\end{align*}
it holds that 
\begin{align*}
\vert \EE_1\vert  + \vert \EE_2 \vert \les& \CMD \Vert \Lieh_{\Ttt} \R \Vert^2_{L^\infty_{\ttt} L^2(\wt{\Si}_{\ttt}\cap \MM_{t^\ast})} + \CMD \sup\limits_{\om \in \mathbb{S}^2} \Vert \Lieh_{\Ttt} \Rbf \cdot \Ltt \Vert^2_{L^\infty_{{}^\om u}L^2(\HH_{{}^\om u} \cap \MM_{t^\ast})} \\
&+ \CMD \lrpar{\norm{\nab \RRRic}_{L^2(\Si_{t^\ast})}+ \norm{\nab^2 k}_{L^2(\Si_{t^\ast})} + \CMD}^2 + \CMD^2.
\end{align*}

The necessary estimates are essentially provided in Sections 11-13 of \cite{KRS}. For completeness, we outline below how to relate the most crucial terms to the parametrix formalism of \cite{KRS} and the null structure of the Einstein equations. 
\subsection{The wave parametrix formalism of Klainerman-Rodnianski-Szeftel} We recall from \cite{KRS} that the connection $1$-form components $(A_i)_{lm} := (\A_i)_{lm}$ satisfy the following structural equations (see Lemmas 6.5 and 13.1 in \cite{KRS})
\begin{align*}
A =& \Curl B + E, \\
\pr A=& \Curl (\pr B) + E',
\end{align*}
where 
\begin{itemize}
\item the vectorfields $B_i$ and $\pr B_i$ satisfy wave equations which exhibit a null structure (see Sections 7.2 and 13.2 in \cite{KRS}),
\item $E$ and $E'$ are error terms with better regularity. 
\end{itemize}

In our case we applied the bounded $L^2$ curvature theorem on the maximal hypersurface $\wt{\Si}_{t^\ast}$ which is constructed by application of Theorem \ref{THMextensionConstraintsCZ1} to the maximal hypersurface $\Si_{t^\ast}\subset \MM$, see also Section \ref{sec:BAimprovBackgroundFoliation}. In particular, we have by Theorem \ref{THMextensionConstraintsCZ1} that on $\wt{\Si}_{t^\ast}$,
\begin{align*}
\Vert \wt{\RRRic} \Vert_{L^2(\wt{\Si}_{t^\ast})} + \Vert \wt{\nab} \wt{k} \Vert_{L^2(\wt{\Si}_{t^\ast})} \les \CMD,
\end{align*}
and
\begin{align*} \begin{aligned}
&\Vert \wt{\nab} \wt{\RRRic} \Vert_{L^2(\wt{\Si}_{t^\ast})}+\Vert \wt{\nab}^2 \wt{k} \Vert_{L^2(\wt{\Si}_{t^\ast})} + \Vert \wt{\RRRic} \Vert_{L^2(\wt{\Si}_{t^\ast})}+\Vert \wt{\nab} \wt{k} \Vert_{L^2(\wt{\Si}_{t^\ast})} \\ 
\les& \Vert \nab \RRRic \Vert_{L^2(\Si_{t^\ast})}+\Vert \nab^2 k \Vert_{L^2(\Si_{t^\ast})} +\CMD.
\end{aligned}
 \end{align*}
Hence in our situation, the wave parametrix is controlled at the level of $m=0$ by the quantity $\CMD$, and at the level of $m=1$ by
\begin{align*}
\Vert \nab \RRRic \Vert_{L^2(\Si_{t^\ast})}+\Vert \nab^2 k \Vert_{L^2(\Si_{t^\ast})} +\CMD.
\end{align*}

Taking into account the above, we cite the next estimates proved in \cite{KRS}.

\begin{enumerate}

\item At the level of $m=0$, it holds by Proposition 7.4 and Lemma 8.3 in \cite{KRS} (see also the bottom of page 186 therein) that
\begin{align}\label{EQparametrixEstimate0}
\sup\limits_{\om \in \SSS^2}\norm{\Nd \pr B}_{L^\infty_{{}^\om u} L^2(\HH_{{}^\om u})} \les \CMD.
\end{align}

\item At the level of $m=1$, it holds by Proposition 13.2 in \cite{KRS} that
\begin{align} \label{EQparametrixEstimate2}
\sup\limits_{\om \in \SSS^2}\norm{\Nd_C (\pr^2 B)}_{L^\infty_{{}^\om u} L^2(\HH_{{}^\om u})} \les& \Vert \nab \RRRic \Vert_{L^2(\Si_{t^\ast})}+\Vert \nab^2 k \Vert_{L^2(\Si_{t^\ast})} +\CMD.
\end{align}

\end{enumerate}

In the following, we connect the most important terms of the integrands in $\EE_1$ and $\EE_2$ to the wave parametrix estimates of \cite{KRS}. For ease of presentation, we leave away the tilde-notation in the next sections.

\subsection{Estimation of $\EE_1$} The Weyl tensor $\Lieh_T \R$ has the same symmetries as $\R$ and thus, an inspection of the proof of the trilinear estimate for $m=0$ (see Section 11 in \cite{KRS}) directly yields the trilinear estimate
\begin{align*} \begin{aligned}
\EE_1 \les& \,\CMD \Vert \Lieh_{T} \R \Vert^2_{L^\infty_{t} L^2(\Si_t)} + \CMD \Vert \Lieh_{T} \R \Vert_{L^2(\MM_{t^\ast})} \sup\limits_{\om \in \mathbb{S}^2} \Vert \Lieh_T \Rbf \cdot L \Vert_{L^\infty_{{}^\om u}L^2(\MM_{t^\ast} \cap \HH_{{}^\om u})} \\
\les&\, \CMD \Vert \Lieh_{T} \R \Vert^2_{L^\infty_{t} L^2(\Si_t)} + \CMD \sup\limits_{\om \in \mathbb{S}^2} \Vert \Lieh_{T} \Rbf \cdot L \Vert^2_{L^\infty_{{}^\om u}L^2(\MM \cap \HH_{{}^\om u})}.
\end{aligned} \end{align*}

\subsection{Estimation of $\EE_2$} By direct calculation (see Propositions 7.1.1 and 7.1.2 in \cite{ChrKl93}) and expanding in an orthonormal frame $(T,e_i), i=1,2,3$, we get
\begin{align} \begin{aligned}
&\D^\a Q(\Lieh_{T} \R)_{\a TTT}  \\
=& (\Lieh_{T} \R)_{T \,\,\, T }^{\,\,\,\, a \,\,\,\, b} \left( \pi^{\a \be} \D_\a \R_{\be a T b} + (\mathbf{div} \pi)^\be \R_{\be a T b} \right)  \\
&+ (\Lieh_{T} \R)_{T \,\,\, T }^{\,\,\,\, a \,\,\,\, b} \left((\D_a \pi_{\a \be} - \D_\be \pi_{\a a}) \R^{\a\be}_{\,\,\,\,\,\, T b} + (\D_{T} \pi_{\a \be}- \D_\be \pi_{\a T} ) \R^{\a \,\,\,\, \be}_{\,\,\,\, a \,\,\,\, b} \right) \\
&+ (\Lieh_{T} \R)_{T \,\,\, T }^{\,\,\,\, a \,\,\,\, b} \left(   (\D_b \pi_{\a \be} - \D_\be \pi_{\a b}) \R^{\a \,\,\,\,\,\,\, \be}_{\,\,\, a T} \right) + \text{ dual terms} \\
=&(\Lieh_{T} \R)_{T \,\,\, T }^{\,\,\,\, a \,\,\,\, b} \lrpar{ \pi^{\a \be} \D_\a \R_{\be a T b}-2 \D_{T} \pi_{T j} \Rbf^j_{\,\,\, aT b} + \D_T \pi_{ij} \Rbf^{i \,\,\,\, j}_{\,\,\,\,a \,\,\,\, b}  - \D_T \pi_{ia} \Rbf^{iT}_{\,\,\,\,\,\,\,\, T b} } \\
&+(\Lieh_{T} \R)_{T \,\,\, T }^{\,\,\,\, a \,\,\,\, b} \lrpar{ \D_i \pi^{iT} \Rbf_{T a T b} + \D_i \pi_{j T} \Rbf^{i \,\,\,\, j}_{\,\,\,\, a \,\,\,\, b} + \D_b \pi_{T j} \Rbf^{T \,\,\,\,\,\,\,\,\, j}_{\,\,\,\, a T}+2 \D_i \pi_{T a} \Rbf_{T \,\,\,\, T b}^{\,\,\,\,\, i} } \\
&+ (\Lieh_{T} \R)_{T \,\,\, T }^{\,\,\,\, a \,\,\,\, b}\lrpar{ \D_i \pi^{ij} \Rbf_{jaT b}  + \D_b \pi_{ij}\Rbf^{i \,\,\,\,\,\,\,\,\, j}_{\,\,\, aT} - \D_j \pi_{ib} \Rbf^{i \,\,\,\,\,\,\,\,\, j}_{\,\,\, aT} - \D_j \pi_{ia} \Rbf^{ij}_{\,\,\,\,\,\,\,\, T b} } \\
&+ \text{ dual terms.}
\end{aligned}  \label{eq:exactTermsTrilinearEstimate} \end{align}

We claim that the terms in \eqref{eq:exactTermsTrilinearEstimate} are bounded by the estimates of Sections 11-13 of \cite{KRS}. In the following, we discuss only the following three terms of \eqref{eq:exactTermsTrilinearEstimate} 
\begin{align*}
\TT_1:=& (\Lieh_{T} \R)_{T \,\,\, T }^{\,\,\,\, a \,\,\,\, b} \pi^{\a \be} \D_\a \R_{\be a T b}, \\
 \TT_2 :=& (\Lieh_{T} \R)_{T \,\,\, T }^{\,\,\,\, a \,\,\,\, b}\D_\be \pi_{\a a} \R^{\a\be}_{\,\,\,\,\,\, T b}, \\
  \TT_3:=& (\Lieh_{T} \R)_{T \,\,\, T }^{\,\,\,\, a \,\,\,\, b} \D_\be \pi_{\a T} ) \R^{\a \,\,\,\, \be}_{\,\,\,\, a \,\,\,\, b}.
\end{align*}
Indeed, the other terms in \eqref{eq:exactTermsTrilinearEstimate} are readily related to the formalism of \cite{KRS}.\\

\textbf{Discussion of $\TT_1$.} Using that 
\begin{align*} 
\pi_{T T} =0, \,\, \pi_{T j }= n^{-1} \nab_j n, \,\, \pi_{ij} = -2  k_{ij},
\end{align*}
the most critical term in $\TT_1$ is given by
\begin{align*}
 (\Lieh_{T} \R)_{T \,\,\, T }^{\,\,\,\, A \,\,\,\, B} k^{i j} \D_i \R_{j A T B}.
\end{align*}
By the parametrix representation for $B$ (see Section 11 in \cite{KRS}), it follows that, up to lower order terms, 
\begin{align*}
k^{i j} \D_i \R_{j A T B} \sim \in_{ilm} N^l \D^m \R^j_{\,\,\, A T B} \sim \D_C \R_{j A T B} \sim \Nd_C \pr^2 B.
\end{align*}
By Section 11.1 in \cite{KRS} and \eqref{EQparametrixEstimate2}, it holds that
\begin{align*}
&\int\limits_{\MM_{t^\ast}} \lrpar{ \Lieh_T \Rbf }_{T \,\,\,\, T}^{\,\,\,\, A \,\,\,\, B} \lrpar{k^{ij} \D_i \R_{j A T B}} \\
\les& \CMD \sup\limits_{\om \in \SSS^2 }\norm{\Lieh_T \Rbf \cdot \Nd_C(\pr^2 B)}_{L^2_{{}^\om u}L^1(\HH_{{}^\om u})} + \CMD^2 \\
 \les& \CMD \lrpar{ \norm{ \Lieh_T \Rbf }_{L^2(\MM_{t^\ast})}^2 + \sup\limits_{\om \in \SSS^2 }\norm{\Nd_C (\pr^2 B)}_{L^\infty_{{}^\om u} L^2(\HH_{{}^\om u})}^2} + \CMD^2 \\
 \les& \CMD \norm{ \Lieh_T \Rbf }_{L^\infty_t L^2(\Si_t)}^2 + \CMD \lrpar{\Vert \nab \RRRic \Vert_{L^2(\Si_{t^\ast})}+\Vert \nab^2 k \Vert_{L^2(\Si_{t^\ast})} +\CMD}^2 + \CMD^2.
 \end{align*}
This finishes our outline of the estimation $\TT_1$.\\

\textbf{Discussion of $\TT_2$.} The most important term of $\TT_2$ is given by
\begin{align*}
(\Lieh_{T} \R)_{T \,\,\, T }^{\,\,\,\, A \,\,\,\, B}\D_j \pi_{i A} \R^{ij}_{\,\,\,\,\,\, T B},
\end{align*}
which can be rewritten as (up to bounded terms)
\begin{align*}
(\Lieh_{T} \R)_{T \,\,\, T }^{\,\,\,\, A \,\,\,\, B}\R_{iTjA} \R^{ij}_{\,\,\,\,\,\, T B},
\end{align*}
where we used that for $a,b,c=1,2,3$,
\begin{align*}
\nab_a k_{bc} - \nab_{b}k_{ac} = \R_{cTab}.
\end{align*}
At first glance, this seems to lead to the desastrous $(\Lieh_{T} \R)_{\Lb \cdot \Lb \cdot }\R_{\Lb \cdot \Lb \cdot} \R_{\Lb \cdot \Lb \cdot}$. However, in the following, we show that due to the null structure of the Einstein vacuum equations, this is not the case. Indeed, on the one hand, we note that
\begin{align*}
(\Lieh_{T} \R)_{T \,\,\, T }^{\,\,\,\, A \,\,\,\, B}\R_{iTjA} \R^{ij}_{\,\,\,\,\,\, T B} = (\Lieh_{T} \R)_{T \,\,\, T }^{\,\,\,\, A \,\,\,\, B}\R_{CTNA} \R^{CN}_{\,\,\,\,\,\, \,\,\,\Ttt B} + \lrpar{\Lieh_T \R} \lrpar{\R \cdot L} \R,
\end{align*}
where the second term is easily estimable. \\

On the other hand, using that $\Lieh_T \R$ is a Weyl tensor, it holds that
\begin{align*}
(\Lieh_{T} \R)_{T \,\,\, T }^{\,\,\,\, 1 \,\,\,\, 1} = - (\Lieh_{T} \R)_{T \,\,\, T }^{\,\,\,\, 2 \,\,\,\, 2} + (\Lieh_{T} \R) \cdot L.
\end{align*}
Therefore, the most difficult term is given by
\begin{align*}
&(\Lieh_{T} \R)_{T \,\,\, T }^{\,\,\,\, A \,\,\,\, B}\R_{CTNA} \R^{CN}_{\,\,\,\,\,\, T B} \\
=& (\Lieh_{T} \R)_{T \,\,\, T }^{\,\,\,\, 1 \,\,\,\, 2} \lrpar{\R_{CTN1} \R^{CN}_{\,\,\,\,\,\,\,\,\, T 2}+\R_{CTN2} \R^{CN}_{\,\,\,\,\,\, \,\,\,T 1}} + (\Lieh_{T} \R)_{T \,\,\, T }^{\,\,\,\, 1 \,\,\,\, 1} \lrpar{\R_{CTN1} \R^{CN}_{\,\,\,\,\,\, \,\,\,T 1} - \R_{CTN2} \R^{CN}_{\,\,\,\,\,\,\,\,\, T 2}} \\
=& (\Lieh_{T} \R)_{T \,\,\, T }^{\,\,\,\, 1 \,\,\,\, 2} \lrpar{\R_{1TN1} \R^{1N}_{\,\,\,\,\,\,\,\,\, T 2}+\R_{1TN2} \R^{1N}_{\,\,\,\,\,\, \,\,\,T 1}+ \R_{2TN1} \R^{2N}_{\,\,\,\,\,\,\,\,\, \Ttt 2}+\R_{2TN2} \R^{2N}_{\,\,\,\,\,\, \,\,\,T 1}} \\
&+ (\Lieh_{T} \R)_{T \,\,\, T }^{\,\,\,\, 1 \,\,\,\, 1} \lrpar{\R_{1TN1} \R^{1N}_{\,\,\,\,\,\, \,\,\,T 1} - \R_{1TN2} \R^{1N}_{\,\,\,\,\,\,\,\,\, T 2}+\R_{2TN1} \R^{2N}_{\,\,\,\,\,\, \,\,\,T 1} - \R_{2TN2} \R^{2N}_{\,\,\,\,\,\,\,\,\, T 2}}.
\end{align*}
Given that $\R_{\mu \nu} = 0$, it holds that
\begin{align*}
&\R_{1TN1} \R^{1N}_{\,\,\,\,\,\,\,\,\, T 2}+\R_{1TN2} \R^{1N}_{\,\,\,\,\,\, \,\,\,T 1}+ \R_{2TN1} \R^{2N}_{\,\,\,\,\,\,\,\,\, T 2}+\R_{2TN2} \R^{2N}_{\,\,\,\,\,\, \,\,\,T 1} \\
=& \R_{1NT1} \lrpar{\R_{1NT2}+ \R_{1TN2}} + \R_{2TN2} \lrpar{\R_{2TN1} + \R_{2NT1}} \\
=& \R_{1TN1} \lrpar{\R_{2TN1} + \R_{1TN2}} - \R_{1TN1} \lrpar{\R_{2TN1} + \R_{2NT1}} \\
=&0, 
\end{align*}
and
\begin{align*}
&\R_{1TN1} \R^{1N}_{\,\,\,\,\,\, \,\,\,T 1} - \R_{1TN2} \R^{1N}_{\,\,\,\,\,\,\,\,\, T 2}+\R_{2TN1} \R^{2N}_{\,\,\,\,\,\, \,\,\,T 1} - \R_{2TN2} \R^{2N}_{\,\,\,\,\,\,\,\,\, T 2} \\
=& (-\R_{2TN2})(-\R_{2TN2}) + \R_{2TN1} \R_{2NT1} - \R_{1TN2} \R_{1NT2} - \R_{2TN2} \R_{2TN2} \\
=&0.
\end{align*}
Therefore, the possibly dangerous term turns out to vanish, that is,
\begin{align*}
(\Lieh_{T} \R)_{T \,\,\, T }^{\,\,\,\, A \,\,\,\, B}\R_{CTNA} \R^{CN}_{\,\,\,\,\,\,\,\,\, T B} =0.
\end{align*}
This finishes our discussion $\TT_2$. \\

\textbf{Discussion of $\TT_3$.} The only critical term of $\TT_3$ is
\begin{align*}
(\Lieh_{T} \R)_{T \,\,\, T }^{\,\,\,\, A \,\,\,\, B} \nab_N \nab_N n \R^{N \,\,\,\, N}_{\,\,\,\, A \,\,\,\, B}.
\end{align*}
However, using that $\triangle n = n \vert k \vert^2$ on $\Si_t$, this can be rewritten as
\begin{align*}
(\Lieh_{T} \R)_{T \,\,\, T }^{\,\,\,\, A \,\,\,\, B} \nab_N \nab_N n \R^{N \,\,\,\, N}_{\,\,\,\, A \,\,\,\, B} =& (\Lieh_{T} \R) \lrpar{\Nd \Nd n }\R + \text{ l.o.t.}\\
=& (\Lieh_{T} \R) \lrpar{\Nd \pr B}\R + \text{ l.o.t.}.
\end{align*}
By writing $\Rbf = \pr (\pr B)$, we get by Sections 11.2 and 13.2 in \cite{KRS} and \eqref{EQparametrixEstimate2} that
\begin{align*}
&\int\limits_{\MM_{t^\ast}} (\Lieh_{T} \R) \lrpar{\Nd \pr B}\R \\
\les& \lrpar{\Vert \nab \RRRic \Vert_{L^2(\Si_{t^\ast})}+\Vert \nab^2 k \Vert_{L^2(\Si_{t^\ast})} +\CMD} \sup\limits_{\om \in \SSS^2 }\norm{(\Lieh_{T} \R) \lrpar{\Nd \pr B }}_{L^2_{{}^\om u} L^1(\HH_{{}^\om u})} \\
\les& \lrpar{\Vert \nab \RRRic \Vert_{L^2(\Si_{t^\ast})}+\Vert \nab^2 k \Vert_{L^2(\Si_{t^\ast})} +\CMD} \norm{\Lieh_{T} \R }_{L^2(\MM_{t^\ast})}\sup\limits_{\om \in \SSS^2 } \norm{\Nd \pr B}_{L^\infty_{{}^\om u} L^2(\HH_{{}^\om u})} \\
\les& \lrpar{\Vert \nab \RRRic \Vert_{L^2(\Si_{t^\ast})}+\Vert \nab^2 k \Vert_{L^2(\Si_{t^\ast})} +\CMD} \norm{\Lieh_{T} \R }_{L^\infty_t L^2(\Si_t)} \CMD \\
\les& \CMD \lrpar{\Vert \nab \RRRic \Vert_{L^2(\Si_{t^\ast})}+\Vert \nab^2 k \Vert_{L^2(\Si_{t^\ast})} +\CMD}^2 + \CMD\norm{\Lieh_{T} \R }_{L^\infty_t L^2(\Si_t)}^2.
\end{align*}
This finishes our discussion of $\TT_3$, and hence of $\EE_2$.


\end{document}